\documentclass[10pt]{article}
\usepackage{color}
\usepackage{amsmath}
\usepackage{amssymb}
\usepackage{mathrsfs}
\usepackage[utf8]{inputenc}
\usepackage[english]{babel}
\usepackage{amsthm}
\usepackage{genyoungtabtikz}
\usepackage{tikz}

\usepackage{graphicx}
\graphicspath{{./images/}}

\usepackage{epsfig}

\usepackage[scaled=1]{helvet}
\usepackage{titlesec}
\usepackage{multido}
\usepackage{multirow}
\usepackage{blindtext}
\usepackage{color}
\usepackage{lipsum}
\usepackage{mathtools}
\usepackage[vcentermath]{youngtab}
\usepackage{young}
\usepackage[all,ps,dvips,graph]{xy}
\usepackage[misc,geometry]{ifsym}

\usepackage[left=3cm,top=2cm,right=3cm,bottom=2cm]{geometry}
\usepackage{graphicx}
\numberwithin{equation}{subsection}

\let\OLDthebibliography\thebibliography
\renewcommand\thebibliography[1]{
  \OLDthebibliography{#1}
  \setlength{\parskip}{0pt}
  \setlength{\itemsep}{0pt plus 0.3ex}
}

\newtheorem{theorem}{Theorem}

\newtheorem{corollary}[theorem]{Corollary}
\newtheorem{lemma}[theorem]{Lemma}
\newtheorem{remark}[theorem]{Remark}
\newtheorem{example}[theorem]{Example}
\newtheorem{definition}[theorem]{Definition}

\numberwithin{theorem}{section}

\newcommand{\A}{\mathcal{A}}
\newcommand{\aaa}{\mathfrak{a}}

\newcommand{\Basis}{\mathcal{C}_n}
\newcommand{\BasisOne}{\mathcal{C}^{1,\K}_n}
\newcommand{\B}{\mathcal{B}_m}
\newcommand{\BB}{\mathcal{B}^{\F}_{l,m}(\kappan,\een)}
\newcommand{\BBnu}{\mathcal{B}^{\F}_{l,m}(\bnu,\een)}
\newcommand{\Bb}{\pmb{\mathfrak{b}}}
\newcommand{\BBB}{\mathcal{B}_{n+1}}
\newcommand{\bbb}{\mathfrak{b}}
\newcommand{\BBk}{\mathcal{B}_k}
\newcommand{\bbkap}{\boldsymbol{\kappa}}
\newcommand\bbn{\mathbf{b}}
\newcommand\bbs{\mathsf{s}}
\newcommand\bbt{\boldsymbol{\mathsf{t}}}
\newcommand\bbu{\mathsf{u}}
\newcommand\bbv{\mathsf{v}}
\newcommand{\Bchico}{\mathcal{B}_{n-1}}
\newcommand\be{\mathbb{E}}
\newcommand\ben{\boldsymbol{e}}
\newcommand{\belongs}{ \in }
\newcommand{\Belongs}{ \ni}
\newcommand{\Bg}{\pmb{\mathfrak{g}}}
\newcommand\bi{\boldsymbol{i}}
\newcommand{\bif}{{\underline{\boldsymbol{i}}}}
\newcommand{\bjf}{{\underline{\boldsymbol{j}}}}
\newcommand\bj{\boldsymbol{j}}
\newcommand\blambda{{\boldsymbol\lambda}}
\newcommand\bmu{{\boldsymbol\mu}}
\newcommand\brho{\boldsymbol{\rho}}
\newcommand\bzeta{{\boldsymbol\zeta}}
\newcommand\bn{\boldsymbol{n}}
\newcommand\bnu{{\boldsymbol\nu}}
\newcommand\boxbluek{\color{blue}\boldsymbol{[k]}}
\newcommand\boxbluej{\color{blue}\boldsymbol{[j]}}
\newcommand\boxredk{\color{red}\boldsymbol{[k]}}
\newcommand\boxredj{\color{blue}\boldsymbol{[j]}}
\newcommand{\Bs}{\pmb{\mathfrak{s}}}
\newcommand\bS{\Sigma}
\newcommand\bs{\mathbf{s}}
\newcommand{\bT}{\pmb{\mathfrak{t}}}
\newcommand\bt{\mathbf{t}}
\newcommand{\bTI}{ \bT^{-1} (\bT_{\theta}^{\blambda}(1))}
\newcommand{\bTII}{ \bT^{-1} (\bT_{\theta}^{\blambda}(2))}
\newcommand{\bTj}{ \bT^{-1} (\bT_{\theta}^{\blambda}(j))}
\newcommand{\bTn}{ \bT^{-1} (\bT_{\theta}^{\blambda}(n))}
\newcommand\btau{{\boldsymbol\tau}}
\newcommand{\Bu}{\pmb{\mathfrak{u}}}
\newcommand{\Bv}{\pmb{\mathfrak{v}}}
\newcommand\bu{\mathbf{u}}
\newcommand\bv{\mathbf{v}}

\newcommand{\C}{\mathcal{C}_{\rm{YH}}}
\newcommand\calL{\mathcal{L}}
\newcommand\calD{\mathcal{D}}
\newcommand\calF{\mathcal{F}}
\newcommand\calP{\mathcal{P}}
\newcommand\calQ{\mathcal{Q}}
\newcommand\bcalL{\boldsymbol{\mathcal{L}}}
\newcommand\bcalD{\boldsymbol{\mathcal{D}}}
\newcommand\bcalY{\boldsymbol{\mathcal{Y}}}
\newcommand\bcalW{\boldsymbol{\mathcal{W}}}
\newcommand{\catorce}{ 14}
\newcommand{\catorceB}{\color{red} 14}
\newcommand{\CC}{ \mathbb C }
\newcommand{\ccc}{\mathfrak{c}}
\newcommand{\ch}{{\rm char}}
\newcommand{\cincuentacinco}{55}
\newcommand{\cincuentacincoR}{\color{red}55}
\newcommand{\Comp}{{\mathcal Comp}_n}
\newcommand{\cuarentacuatro}{ 44}
\newcommand{\cupdot}{\mathbin{\mathaccent\cdot\cup}}
\newcommand{\Cs}{\overleftarrow{C}}
\newcommand{\Cd}{\overrightarrow{C}}

\newcommand{\Der}{{\rm Der}}
\newcommand{\diez}{10}
\newcommand{\dieciseis}{16}
\newcommand{\diezR}{\color{red}10}
\newcommand{\doce}{12}
\newcommand{\doceB}{\color{blue} 12}
\newcommand{\doceR}{\color{red} 12}

\newcommand{\E}{ {\mathcal E}_n(q)}
\newcommand{\e}{\mathfrak{e}}
\newcommand{\EE}{ {\mathcal E}_n}
\newcommand\een{\mathbf{e}}
\newcommand\es{\mathbbm{s}}
\newcommand{\End}{{\rm End}}
\newcommand\et{\mathbbm{t}}
\newcommand\eu{\mathbbm{u}}
\newcommand\ev{\mathbbm{v}}
\newcommand{\Exp}{ {\rm \bf exp} }

\newcommand{\F}{ { \mathbb F}}
\newcommand{\FF}{ {\mathcal F}_n}

\newcommand{\g}{  \mathfrak{g}}
\newcommand{\gl}{\mathfrak{gl}}

\newcommand{\h}{{h}}
\newcommand{\HH}{ \mathcal{H}_n}
\newcommand{\HHO}{ \mathcal{H}^{\OO}_n}
\newcommand{\HHK}{ \mathcal{H}^{\K}_n}
\newcommand{\HHtwo}{ \mathcal{H}_2}
\newcommand{\HHKtwo}{ \mathcal{H}^{\K}_2}
\newcommand{\HHOtwo}{ \mathcal{H}^{\OO}_2}
\newcommand{\HHKOne}{ \mathcal{H}^{1,\K}_n}

\newcommand{\II}{I_{\mathbf{e}}}
\newcommand{\IIa}{I_{e}}
\newcommand{\id}{{\rm id}}
\newcommand{\ind}{{\rm ind}}
\newcommand{\inv}{{\rm inv}}

\newcommand{\JM}{ \mathcal L }

\newcommand{\K}{\mathcal{K}}
\newcommand{\kk}{\mathcal{K}}
\newcommand{\kkk}{k-1}
\newcommand{\kappan}{\boldsymbol{\kappa}}

\newcommand{\Li}{\mathcal{L}}
\newcommand{\LL}{\mathbb{L}}

\newcommand{\m}{\mathfrak{m}}
\newcommand{\MC}{{ {\rm Comp}}_{l,n}}
\newcommand{\MCm}{{ {\rm Comp}}_{l,m}}
\newcommand{\MP}{{\rm Par }_{l,n}}
\newcommand{\mfra}{{\mathfrak{a}}}
\newcommand{\mfrb}{{\mathfrak{b}}}
\newcommand{\mfrc}{{\mathfrak{c}}}
\newcommand{\mfrt}{{\mathfrak{t}}}
\newcommand{\mfrs}{{\mathfrak{s}}}
\newcommand{\mfru}{{\mathfrak{u}}}
\newcommand{\mfrv}{{\mathfrak{v}}}
\newcommand{\mfrw}{{\mathfrak{w}}}
\newcommand{\mfrx}{{\mathfrak{x}}}
\newcommand{\mfry}{{\mathfrak{y}}}
\newcommand{\mfrz}{{\mathfrak{z}}}

\newcommand{\N}{ { \mathbb N}}
\newcommand{\No}{ { \mathbb N}_0}
\newcommand{\nstd}{{\rm NStd}}
\newcommand{\NB}{\mathbb{N}\mathcal{B}_k}

\newcommand{\once}{11}
\newcommand{\onceB}{\color{blue}11}
\newcommand{\OnePar}{{ \rm Par}^1_{l,m}}
\newcommand{\OneParn}{{ \rm Par}^1_{l,n}}
\newcommand{\OO}{\mathcal{O}}
\newcommand{\op}{\otimes}

\newcommand{\Par}{{\rm Par}_{l,m}}

\newcommand{\q}{\hat{q}}
\newcommand{\quince}{15}
\newcommand{\quinceB}{\color{red} 15}

\newcommand{\R}{ \mathcal{R}_m}
\newcommand{\Rad}{{\rm Rad}}
\newcommand{\res}{ \textrm{res} }
\newcommand\rrn{\mathbf{r}}

\newcommand{\s}{\mathfrak{s}}
\newcommand{\seq}{{\rm seq}_n}
\newcommand{\shape}{\textsf{shape}}
\newcommand{\Si}{\mathfrak{S}}
\newcommand{\sign}{-}
\newcommand{\sixteenB}{\color{red} 16}
\newcommand{\Snake}{{\rm Snake}}
\newcommand{\spa}{{\rm span}}
\newcommand{\std}{{\rm Std}}

\newcommand{\T}{  \mathfrak{t}}
\newcommand{\tab}{{\rm Tab}}
\newcommand{\tr}{{\rm \textbf{tr}}}
\newcommand{\tR}{ \overline{\R}}
\newcommand{\Tr}{{\rm Tr}}
\newcommand{\trece}{13}
\newcommand{\treintatres}{33}
\newcommand{\treintatresR}{\color{red}33}
\newcommand{\trunc}{ {\B} (\blambda ) }
\newcommand{\truncPrime}{ {\mathbb B}_n^{\prime} (\blambda ) }
\newcommand{\truncSing}{ {\mathbb B}_{\bar n} (\overline{\blambda} ) }
\newcommand{\TT}{{\mathfrak T}}
\newcommand{\TTc}{  \mathcal{T}}

\newcommand{\U}{\mathfrak{u}}
\newcommand{\UU}{\mathbb{U}}
\newcommand{\UUc}{\mathcal{U}}

\newcommand{\V}{\mathfrak{v}}
\newcommand{\veintidos}{22}
\newcommand{\VV}{\mathbb{V}}

\newcommand{\Y}{\mathcal{Y}}
\newcommand{\Yy}{\mathcal{Y}_{r,n}}
\newcommand{\yvc}{\Yvcentermath1}
\newcommand{\YY}{\mathcal{Y}_{r,n}(q)}

\newcommand{\Z}{\mathbb{Z}}

\begin{document}
  \Yvcentermath1
\title{On generalized blob algebras: Vertical idempotent truncations and Gelfand-Tsetlin subalgebras.}
\author{\sc  Diego Lobos }

\maketitle

\pagenumbering{roman}

\begin{abstract}
  We study the structure of the Gelfand-Tsetlin subalgebras of a family of idempotent truncations of the generalized blob algebras. We obtain optimal presentations in terms of generators and relations, monomial basis and dimension formula.
\end{abstract}

Keywords: Generalized blob algebras, Idempotent truncations, Gelfand-Tsetlin subalgebras.
\pagenumbering{arabic}

\section{Introduction}

Given a cellular algebra $A$ provided with a set $\mathcal{J}$ of \emph{Jucys-Murphy} elements, the \emph{Gelfand-Tsetlin} subalgebra of $A$ is defied as the algebra generated by $\mathcal{J}$ \cite{OkunVershik1}. This is a commutative subalgebra of $A,$ fundamental in the study of representation theory of $A,$ as one can see, for example, in the work of Murphy \cite{Murphy1},\cite{Murphy2}, Mathas \cite{Mat-So},\cite{MatCoef},\cite{hu-mathas}, Okounkov and Vershik \cite{OkunVershik1},\cite{OkunVershik2} among other. In particular, due to the work of Mathas \cite{Mat-So}, we know that for a cellular algebra $A,$ equipped with a family of \emph{Jucys-Murphy} elements satisfying certain \emph{separation} conditions, one can construct, from them, seminormal forms and explicitly compute Gram determinant of the irreducible modules. For not separated cases, we can obtain a decomposition of $A,$ into a direct sum of cellular subalgebras, we can construct 'orthogonal by blocks' basis for the different cell modules of $A$ and then to obtain certain control on the set of the irreducible $A-$modules. Therefore, to have an optimal description of the \emph{Gelfand-Tsetlin} subalgebra of $A,$ is an important element, in order to understand how it interacts with the rest of the algebra and how we can use it to obtain relevant information on the representation theory of $A.$

Since their definition due to Martin and Saleur \cite{Mat-Sal} and Martin and Woodcock \cite{martin-wood1} respectively, blob algebras and generalized blob algebras have received considerable amount of interest. It is known that their representation theory is strongly connected with Kazhdan-Lusztig polynomials in characteristic $0$ and with $p-$Kazhdan-Lusztig polynomials in characteristic $p$ (see \cite{bcs}, \cite{martin-wood1}, \cite{Plaza13}, \cite{LiPl}, \cite{bow-cox-hazi}), currently one of the central object of study in Representation Theory.  Few years ago Ryom-Hansen and the author have explicitly built graded cellular basis for generalized blob algebras $\B$, provided with a set of \emph{Jucys-Murphy} elements \cite{Lobos-Ryom-Hansen}. What we are looking now, is a complete description of the \emph{Gelfand-Tsetlin} subalgebras of the many idempotent truncations of the generalized blob algebra $\B.$ This is the first of a series of articles, where we progressively aboard this problem. In this case we work on a particular family of idempotent truncations $\B(\underline{\bi})$ of $\B,$ that we call \emph{vertical idempotent truncations}.


The main results of this article are: Theorem \ref{theo-monomial-basis}, Corollary \ref{coro-dimD},  where we describe a monomial basis and we provide an explicit formula for the dimension of the Gelfand-Tsetlin subalgebra $\mathcal{D}$ of $\B(\underline{\bi}).$ Theorem \ref{theo-isomorphims-calD-bcalD} and Corollary \ref{coro-alternative-optimal-presentation}, where we describe optimal presentations for $\mathcal{D},$ that is, presentations, where the numbers of generators and relations are minimal.

This article is inspired by the work of Espinoza and Plaza \cite{Esp-Pl} where, among other results, the authors developed an equivalent study for the case of level $l=2,$ that is the blob algebra instead of generalized blob algebras. This article can
be seen as a partial generalization of their work.



The structure of this paper is given as follows: In section \ref{ch-generalities} we make a review on the general language that we use in this article (see section \ref{sec-fixing-notation}), particularly we recall the combinatorial elements related with the algebras of our interest (see section \ref{sec-general-combinatorics}). In the last part of section \ref{ch-generalities} we introduce the concept of \emph{blob-possible} (resp. \emph{blob-impossible}) residue sequences, and we obtain a series of small results, that will be relevant for the rest of the article (see section \ref{sec-impossible-res-sec}).

Section \ref{ch-gen-blob-alg}  is devoted to the generalized blob algebras $\B.$ We review the well known diagrammatic presentation of this family of algebras (see section \ref{sec-definition-gen-blob}). We recall the graded cellular basis developed by Ryom-Hansen and the author in \cite{Lobos-Ryom-Hansen} (see section \ref{sec-graded-cellular-basis-of-B}). Finally we recall the definition of \emph{idempotent truncations} of generalized blob algebras (see section  \ref{sec-idemp-trunc-of-B}).

In section \ref{ch-fundamental-idemp-trunc} we introduce the \emph{fundamental vertical truncation} $\B(\bif)$ of $\B$ one of the main object of study of this article. We define particular combinatorial objects related as \emph{fundamental vertical sequence} (section \ref{sec-fund-resid-sec}) and \emph{tableaux defined by blocks} (section \ref{sec-tab-def-blocks}). At the end of section \ref{ch-fundamental-idemp-trunc}, we explicitly describe a set of relevant members of the cellular basis of this idempotent truncation (see section \ref{sec-fund-idemp-and-cellular-basis}).

Finally in section \ref{ch-gelfand} we deeply study the Gelfand-Tsetlin subalgebra $\calD,$ of $\B(\bif).$ We start with a natural definition, that will be optimized throughout section \ref{ch-gelfand}. We first reduce the set of generators (see section \ref{sec-optimized-generators}) and then we find a series of \emph{fundamental relations} that allow a better understanding of this algebra (see section \ref{sec-fundamental-relations-for-D}). Later we study the vector space structure of $\calD$ providing a monomial basis and a dimension formula for this algebra (see section \ref{sec-vector-space-structure}). We also provide an optimal presentation of $\calD,$ where the number of generators and relations is strongly reduced (see section \ref{sec-optimal-presentation}). Finally we obtain a series of direct generalizations of our results to a larger family of \emph{vertical} and \emph{quasi-vertical} idempotent truncations (see section \ref{sec-direct-general}).

\subsection{Acknowledgment}
\medskip

 This work was supported by \emph{Proyecto DI Emergente PUCV 2020}, 039.475/2020. Pontificia Universidad Cat\'olica de Valpara\'iso, Chile.
\medskip

It is a pleasure to thank Steen Ryom-Hansen for his encouragement in this project and multiple suggestions that improve this article.

\medskip

I especially want to thank go to Camilo Martinez Estay, bachelor student in Mathematics at PUCV, for his excellent work as Assistant Researcher in the first part of this project.

\section{Generalities}\label{ch-generalities}
\subsection{Fixing parameters and notations}\label{sec-fixing-notation}

For the whole article we fix positive integers $l,\een>2,$ $m>1$ and a prime number $p>0$ such that  $ \mbox{gcd}(\een,p) =1 $. Let $ \F $ be a field of characteristic $ p $ and suppose that $ q \in \F \setminus \{ 1 \} $ is a primitive $ \een$'th root of unity. Let $\II:=\mathbb{Z}/\een\mathbb{Z}$.

We call \emph{residue sequences modulo} $\een$ of \emph{ length} $m$ to any element $\bi = (i_1, \ldots, i_m) $ of $ \II^m.$ Given two residue sequences (not necessarily with the same length) $\bi=(i_1,\dots,i_m)$ and $\bj=(j_1,\dots,j_n)$ we define the concatenation $\bi\otimes\bj\in\II^{m+n}$ as $\bi\otimes\bj=(i_1,\dots,i_m,j_1,\dots,j_n).$ Particularly if $j\in \II,$ we write:
$\bi\otimes j=(i_1,\dots,i_m,j).$

Let us consider the well known Coxeter group structure of the Symmetric group $\Si_m,$ that is, the pair $ (\Si_m,S) ,$ where $S=\{s_1,\dots,s_{m-1}\},$ and $s_j=(j,j+1),$ is the transposition between $j$ and $j+1.$ Let $\ell:\Si_m\rightarrow \mathbb{Z}_{\geq 0}$ be the length function on $(\Si_m,S)$.

It will be useful to use the notation $[a:b]=\{a,a+1,\dots,b\},$ for intervals of integral numbers $a<b.$ We also denote by $s_{[a,b]}=s_as_{a+1}\cdots s_b.$

More generally, for $a<b$ and $g\in\mathbb{Z}$ such that $b=a+ng$ for certain $n\in\mathbb{N},$ we define
 $$[a:g:b]=\{a,a+g,a+2g\dots,b\},\quad \textrm{and}\quad s_{[a:g:b]}=\prod_{j=0}^{n}s_{a+jg}$$

 We also consider the free monoid $\Si_{m}^{\ast}$ generated by $S$ (the set of words in the alphabet $S$ with concatenation as a product). Let $\ell^{\ast}:\Si_{m}^{\ast}\rightarrow \mathbb{Z}_{\geq 0}$  be the (natural) length function on $\Si_{m}^{\ast}$ and let $\Im:\Si_{m}^{\ast}\rightarrow \Si_{m}$ the natural projection of $\Si_{m}^{\ast}$ on $\Si_{m}$. For each element $x\in\Si_{m}$  let us define the set of \emph{reduced words} of $x$ as:
$\Re(x)=\{w\in\Si_{m}^{\ast}:\Im(w)=x, \ell^{\ast}(w)=\ell(x)\}.$
For the rest of the article we fix for each $x\in\Si_m$ an element $\boldsymbol{x}\in \Re(x).$ We call the element $\boldsymbol{x}$ the \emph{official reduced word} of $x.$

The Symmetric group acts by the left on $ \II^m$ via permutation of
the coordinates, that is  $s_j \cdot \bi   :=(i_1, \ldots,  i_{j+1}, i_j, \ldots, i_m).$

We call \emph{multicharge} of level $l$ to any element $ \boldsymbol{\kappa} = (\boldsymbol{\kappa}_1, \ldots, \boldsymbol{\kappa}_l )$ of ${\mathbb Z}^l$ and we call \emph{multicharge modulo} $\een$ to its natural projection $ \kappa := (\kappa_1, \ldots , \kappa_l)  \in \II^l$.
We shall throughout choose an interval of real numbers $\boldsymbol{J}$ of length equal to $\een,$ and for each $ \kappa_i$, a representative $ \hat{\kappa}_i\in \boldsymbol{J}.$

\begin{definition}\label{def-new-strong-adj-free}
We say that $ \boldsymbol{\kappa} $ is \emph{strongly adjacency-free} if it satisfies
 \begin{itemize}
 \setlength\itemsep{-1.1em}
  \item[i)] $\boldsymbol{\kappa}_{i+1} - \boldsymbol{\kappa}_{i} \geq m $  \\
  \item[ii)] $ \kappa_i-\kappa_j \neq 0, \pm 1  \, \,\mbox{ mod }\,  \een \quad  \mbox{ for all } i\neq j $  \\
  \item[iii)] $ \kappa_1 \neq \kappa_l +2 \,\, \mbox{ mod }\,  \een$ \\
  \item[iv)] $ \hat{\kappa}_1 < \hat{\kappa}_2 < \ldots < \hat{\kappa}_l.$
 \end{itemize}
\end{definition}

We shall in the following always assume that $ \boldsymbol{\kappa} $ is strongly adjacency-free. (in particular we have to assume that  $\een > 2l$).

\begin{remark}\label{remark-new-strong-adj-free}
  The definition of  strongly adjacency-freeness used in this article (definition \ref{def-new-strong-adj-free}) is a modification of the definition used in \cite{Lobos-Ryom-Hansen}, where the interval $\boldsymbol{J}$ was always assumed to be equal to $[0,e-1].$ In \cite{Lobos-Ryom-Hansen} there were crucial steps where the original strongly adjacency-freeness was used, but in all of them we can change it by this new version without affecting the conclusions obtained there.
  This new definition allows us to obtain more general results than with the original one.
\end{remark}

Along this article, we understand \emph{algebra} or $\F-$\emph{algebra}, as an associative and unital algebra over the field $\F.$ If $A$ is an algebra, and $B\subset A$ is also an algebra, then we say that $B$ is a \emph{subalgebra} of $A$ if the identity element of $B$ is  also the identity element of $A$. In other case we only say that $B$ is an algebra contained in $A.$

\subsection{General combinatorics}\label{sec-general-combinatorics}

An $l-$\emph{multipartition} of $m$ is an $l-$ tuple $\blambda=(\lambda^{(1)},\dots,\lambda^{(l)})$ of partitions $\lambda^{(j)}=(\lambda^{(j)}_{1},\dots,\lambda^{(j)}_{{r_j}})$ such that
$\sum_{j=1}^{l}|\lambda^{(j)}|=m,$ where $|\lambda^{(j)}|=\sum_{i=1}^{{r_{j}}}\lambda^{(j)}_i$ and $\lambda^{(j)}_1\geq \dots \geq \lambda^{(j)}_{r_{j}}.$ Let $\Par$ the set of all $l-$multipartitions of $m.$

 The diagram of an $l-$multipartition $\blambda\in \Par$ is the set:
 $$[\blambda]=\left\{(r,c,h):1\leq r\leq r_{j},1\leq c\leq \lambda^{(j)}_r, 1\leq h\leq l \right\}$$
 Each element of $[\blambda]$ is called node of $[\blambda].$ More specifically if $\gamma=(r,c,h)$ is a node of $[\blambda],$ then we say that $\gamma$ is the node in the row $r,$ column $c$ and component $h$ of $[\blambda].$ More generally we call node to each element $\gamma=(r,c,h)\in \mathbb{N}\times \mathbb{N}\times \{1,2,\dots,l\}.$

A \emph{one-column} partition is a partition $\lambda=(\lambda_1,\dots,\lambda_r)$ such that $\lambda_j=1$ for each $j=1,\dots,r.$ We denote a one-column partition by $\lambda=1^{(r)}$

A \emph{one-column} $l-$multipartition of $m$ is an $l-$multipartition such that, each of its components is a one-column partition. We denote by $\OnePar$ the set of all one column $l-$multipartitions of $m.$

A tableau $\bT$ of \emph{shape} $\blambda$ (denoted by $\shape(\bT)=\blambda$) is a bijection $\bT:\{1,\dots,k\}\rightarrow [\blambda].$ A tableau can be seen as a labeling of the boxes of the diagram $[\blambda].$ We denote by $\tab(\blambda)$ the set of all tableau of shape $\blambda.$ Sometimes we will be working with tableaux corresponding to multipartitions of different numbers $m,n,$ etc. Then we will refer to $m-$tableau and $n-$tableau respectively, to make a difference between them.

In particular we say that a tableau is a \emph{one-column} tableau, if its shape is a one-column $l-$multipartition of $m.$

For a one-column $l-$multipartition $\blambda=(1^{(a_1)},\dots,1^{(a_l)})$ we call \emph{blob-addable} node of $\blambda,$ to any of the following nodes:
$(a_1+1,1,1),(a_2+1,1,2),\dots,(a_l+1,1,l).$ If $\bT$ is a one-column tableau with shape $\blambda,$ then we call \emph{blob-addable} node of $\bT$ to any blob-addable node of $\blambda.$

An standard tableau is a tableau in which, in each component, the entries increase along each row and down each column. We denote by $\std(\blambda
)$ the set of all standard tableau of shape $\blambda.$

We define for each node $\gamma=(r,c,h)$ the \emph{residue of the node} by
\begin{equation*}
  \res(\gamma)=\kappa_h+c-r\quad (\mod \een).
\end{equation*}
A node of residue $i$ is called $i-$node. We say that two nodes are \emph{sisters} if they have the same residue. We say that two nodes $\beta,\gamma$ are \emph{cousins} if $\res(\beta)=\res(\gamma)\pm 1.$ Two nodes are \emph{not relatives} if they are not \emph{sisters} nor \emph{cousins}.

For each tableau $\bT$ we define the residue sequence $\bi^{\bT}=(i_1,\dots,i_k)\in \II^{m},$ where $i_j=\res(\bT(j)).$

Let us consider the following orders on nodes, multipartitions and tableaux respectively:

Given two nodes $\gamma_1=(r_1,c_1,h_1)$ and $\gamma_2=(r_2,c_2,h_2)$ we denote $\gamma_1\rhd \gamma_2$ if either $c_1-r_1>c_2-r_2$ or if $c_1-r_1=c_2-r_2$ and $h_1<h_2.$

Given two $l-$multipartitions $\blambda, \bmu$ we denote $\blambda \rhd \bmu$ if for each node $\gamma_0$ we have:

$$|\{\gamma\in [\blambda]:\gamma\rhd \gamma_0\}|\geq|\{\gamma\in[\bmu]:\gamma\rhd \gamma_0\}|,$$ and at least for one node $\gamma_1$ we have
$$|\{\gamma\in [\blambda]:\gamma\rhd \gamma_1\}|>|\{\gamma\in[\bmu]:\gamma\rhd \gamma_1\}|.$$

Finally given two tableaux $\bT$ and $\Bs$ we denote $\bT \rhd \Bs$ if for each $1\leq k \leq m$ we have $\shape(\bT|_k)\trianglerighteq \shape(\Bs|_k),$ and at least for some $k_0$ we have
$\shape(\bT|_{k_0})\rhd \shape(\Bs|_{k_0}).$
Here $\bT|_k$ and $\Bs|_k$ are the corresponding restriction of the function $\bT$ and $\Bs$ to the set $\{1,\dots,k\}.$

\begin{example}
  Let
   $$\Bs=\, \gyoung(;1,;3,;6,;8),\gyoung(;4,:,:,:),\gyoung(;2,;5,;7,:) \,\quad\textrm{and}\quad\bT=\, \gyoung(;1,;3,;4,;8),\gyoung(;5,:,:,:),\gyoung(;2,;6,;7,:) \,. $$

   then $\Bs\rhd\bT.$
\end{example}



Let us concentrate in the case that we are interested in this article. For each $\blambda\in \OnePar$ the set $\std(\blambda)$ has a unique maximal element denoted by $\bT^{\blambda}.$ This element can be explicitly constructed by filling the nodes of $[\blambda]$ with the numbers $1,\dots,m$ decreasingly with respect to the order $\rhd.$ (Details in \cite{Lobos-Ryom-Hansen}).




For each $\blambda,$ there is a transitive action of the symmetric group  $ \Si_m $ on the right over $\tab(\blambda)$  via permutation of entries of each tableau. Particularly for each tableau $\bT\in \std(\blambda)$ we can associate a well defined element $d(\bT)\in \Si_m, $ by $\bT^{\blambda}d(\bT)=\bT.$





\begin{definition}\label{def-liftable-nodes}
  Given an $l-$multipartition $\blambda$ and let $\bT:\{1,\dots,m\}\rightarrow[\blambda]$ an standard tableaux with shape $\blambda.$ We say that a number $a\in\{1,\dots,m-1\}$ is liftable in $\bT$ if the tableau $\Bs:=\bT s_a$ satisfies

  $  \Bs\rhd \bT.$  We denote by $\mathbb{L}(\bT)$ the set of all liftable numbers in $\bT.$
\end{definition}



  Note that $\bT=\bT^{\blambda}$ if and only if $\mathbb{L}(\bT)=\emptyset$ (see \cite{Lobos-Ryom-Hansen} for details).

Fix $\blambda\in \OnePar.$ We define an element $W(\bT)\in\mathfrak{S}_m^{\ast}$ (notation in section \ref{sec-fixing-notation}) for any tableaux $\bT\in \std(\blambda)$  by reverse induction on the order $\rhd:$

\begin{definition}\label{def-Word-bT}

Let $\bT\in \std(\blambda),$ then

\begin{enumerate}
  \item If $\bT=\bT^{\blambda},$ we define
  \begin{equation}\label{eq-def-Word-bT-1}
    W(\bT)=1.
  \end{equation}

  \item If $\bT\neq\bT^{\blambda}.$ Assume that we have already defined the element $W(\Bs)$ for any $\Bs\rhd \bT.$
  Let $a=\max(\mathbb{L}(\bT)),$ the greatest liftable number in $\bT.$ Let $\Bs:=\bT s_a.$
  We define
  \begin{equation}\label{eq-def-Word-bT-2}
    W(\bT)=W(\Bs)s_a\in\mathfrak{S}_m^{\ast}.
  \end{equation}
\end{enumerate}
\end{definition}

\begin{example}
  Let $$\bT=\, \gyoung(;1,;2,;3,;4,;5,;6,;7),\gyoung(;8,;9,;\diez,;\once,;\doce,:,:),\gyoung(;\trece,:,:,:,:,:,:) \,.$$
  Then we can check that:
  \begin{equation*}
    W(\bT)=s_3s_2s_{[5:-1:3]}s_{[7:-1:4]}s_{[9:-1:5]}s_{[11:-1:6]}s_{[12:-1:7]}s_{[9:12]}.
  \end{equation*}
\end{example}

\begin{lemma}
  Let $\bT$ be a one-column standard tableau. The element $W(\bT)\in\mathfrak{S}_n^{\ast}$ is a reduced word of the element $d(\bT)\in \mathfrak{S}_n.$
\end{lemma}

\begin{proof}
  To see that the projection $\Im(W(\bT))\in \mathfrak{S}_n$ is equal to $d(\bT)$ see \cite{Lobos-Ryom-Hansen}.

  To see that $W(\bT)$ is a reduced word see \cite{Bjorner-Brenti}.
\end{proof}

From here and so on, we define for this article as the \emph{official reduced word} of $d(\bT)$ the expression $\boldsymbol{d}(\bT)=W(\bT)$ obtained following the algorithm described in definition \ref{def-Word-bT}.

\subsection{Possible and impossible residue sequences}\label{sec-impossible-res-sec}

Given a residue sequence $\bi\in\II^{n},\quad (n\in\mathbb{N})$  we say that $\bi$ is \emph{$\kappa$-blob possible} if there is a one column standard tableau $\bT,$ such as  $\bi=\bi^{\bT}.$ In other case we say that  $\bi$ is \emph{$\kappa$-blob impossible}.

\begin{example}\label{ex-blob-impossible}
  Since we are assuming that $\boldsymbol{\kappa}$ is strongly adjacency-free, then any residue sequence $\bi=(i_1,\dots,i_n)$ such that $i_1=i_2$ is $\kappa$-blob impossible.

   More generally $\bi=(i_1,\dots,i_n)$ is a residue sequence, and suppose that, there is a $1\leq j\leq n-l$ such that $i_j=i_{j+1}=\cdots=i_{j+l},$ then $\bi$ is $\kappa$-blob impossible.
\end{example}

\begin{lemma}\label{lemma-bj-possible-bjp-too}
  If $\bi=(i_1,\dots,i_n)\in\II^{n}$ is a $\kappa-$blob possible residue sequence, then for any $k<n$ the restriction $\bi|_{k}=(i_1,\dots,i_k)\in\II^{k}$ is a $\kappa-$blob possible residue sequence.
\end{lemma}

\begin{proof}
  Since $\bi$ is a $\kappa-$blob possible residue sequence, then there is an standard $n-$tableau $\bT$ such that  $\bi=\bi^{\bT}.$ Since the restriction $\Bs:=\bT|_{k}$ is an standard $k-$tableau and $\bi|_{k}=\bi^{\Bs}$ then we conclude that $\bi|_{k}$ is a $\kappa-$blob possible residue sequence.
\end{proof}

\begin{corollary}
   If $\bi=(i_1,\dots,i_n)\in\II^{n}$ is a $\kappa-$blob impossible residue sequence, then for any residue $j\in\II,$ we have that
  $\bj=\bi\otimes j\in\II^{n+1}$ is a $\kappa-$blob impossible residue sequence.
\end{corollary}

The following lemma have several consequences with respect to the concept of $\kappa-$blob possible (impossible) residue sequences.

\begin{lemma}\label{lemma-same-addable-residues}
  If $\bT,\Bs$ are two one-column standard tableaux with the same residue sequence, then their blob-addable nodes have the same set of residues.
\end{lemma}

\begin{proof}
  Let $\bi=(i_1,\dots,i_n)$ the residue sequence of both $\bT$ and $\Bs.$ Let $\blambda=\shape(\bT)$ and $\bmu=\shape(\Bs).$

  If $n=1,$ then by strong adjacency-freeness of $\kappan,$ we have that $\bT=\Bs$ and the assertion of the lemma is trivial.

  If we assume that $n>1,$  $\bT|_{n-1}=\Bs|_{n-1},$ but $\bT\neq \Bs.$ In this case take $\gamma=(r,1,h)=\bT(n)$ and $\beta=(r',1,h')=\Bs(n).$ Since $\res(\bT)=\res(\Bs),$ then $\res(\gamma)=\res(\beta)$ therefore the corresponding blob-addable nodes $(r+1,1,h)$ and $(r'+1,1,h')$ have the same residue. The other blob-addable nodes of $\bT$ and $\Bs$ can be seen as blob-addable nodes of $\bT|_{n-1}=\Bs|_{n-1},$ then it is clear that the set of residues of them is the same for both $\bT$ and $\Bs.$

  Lets we assume that $n$ is large enough, $\bT\neq \Bs.$  Again we can take $\gamma=(r,1,h)=\bT(n)$ and $\beta=(r',1,h')=\Bs(n).$ Since $\res(\bT)=\res(\Bs),$ then $\res(\gamma)=\res(\beta)$ therefore the corresponding blob-addable nodes $(r+1,1,h)$ and $(r'+1,1,h')$ have the same residue. The other addable nodes of $\bT$ and $\Bs$ can be seen as addable nodes of $\bT|_{n-1},\Bs|_{n-1},$ respectively, then by induction the set of residues of them is the same. Therefore we conclude the assertion for $\bT$ and $\Bs.$
\end{proof}

\begin{corollary}\label{coro-bt-j-impossible-sequence}
  Let $\bi=(i_1,\dots,i_n)\in\II^{n}$ the residue sequence of a one-column standard tableau $\bT.$ If $j\in\II$ is not the residue of a blob-addable node of $\bT,$ then the extended residue sequence $\bj=\bi\otimes j\in\II^{n+1}$ is a $\kappa-$blob impossible residue sequence.
\end{corollary}

\begin{proof}
  If we assume that $\bj=(i_1,\dots,i_n,j)$ is $\kappa-$blob possible, then there exist a one-column standard tableau (of $n+1$), $\Bs$ such that $\bi=\bi^{\Bs}.$ But then $\bT$ and $\Bs|_{n}$ are two one-column standard tableaux (of $n$) with the same residue sequence $\bi.$ Then by lemma  \ref{lemma-same-addable-residues} they have the same set of residues for their corresponding addable nodes. Since $\Bs(n+1)$ is an addable node for $\Bs|_{n}$ and it has residue $j$ we obtain a contradiction.
\end{proof}

The two next corollaries are very important for our purposes in the rest of this article:

\begin{corollary}\label{coro-bif-j-not-possible}
  Let $t\in\{1,\dots,l\}.$ If $\bif=(i_1,i_2,\dots,i_n)$ is the residue sequence given by:
  $$i_r=\kappa_t-1-r\quad (\mod\een),$$
  and $j\in\II$ is a residue such that
  $j\neq i_n-1\wedge j\notin\{\kappa_s: s\neq t\},$
  then $\bj=\bif\otimes j$ is a $\kappa-$blob impossible residue sequence.
\end{corollary}

\begin{proof}
  We can assume that $t=1$ without loss of generality.

  Lets define $\blambda=(1^{(n)},1^{(0)},\dots,1^{(0)})$ and $\bT$ the only possible standard tableau with shape $\blambda,$ that is $$\bT(r)=(r,1,1),\quad 1\leq r\leq n.$$
  Is not difficult to see that the blob-addable nodes of $\bT$ are $(n+1,1,1),(1,1,2),(1,1,3),\dots,(1,1,l),$ then the set of residues for blob-addable nodes is given by $\{i_n-1,\kappa_2,\dots,\kappa_l\}.$
  By our assumption on $j$ we conclude that $\bj$ is $\kappa-$blob impossible residue sequence.

  (Note that corollary \ref{coro-bt-j-impossible-sequence} implies that our argument is enough to conclude the assertion of this corollary.)
\end{proof}

\begin{corollary}\label{coro-bif-j-not-possible2}
  Let $t\in\{1,\dots,l\}. $If $\bif=(i_1,i_2,\dots,i_n)$ is the residue sequence given by:
  $$i_r=\kappa_t-1-r\quad (\mod\een),$$
  and $j\in\{\kappa_s:s\neq t\}$
  then $\bj=(i_1,\dots,i_n,j,j+1,j+1)$ is a $\kappa-$blob impossible residue sequence.
\end{corollary}

\begin{proof}
  We can assume that $t=1$ and $j=\kappa_2$ without loss of generality.

   If $j+1=i_n-1,$ then define $\blambda=(1^{(n+1)},1^{(1)},1^{(0)},\dots,1^{(0)}),$ and $\bT:\{1,\dots,n+2\}\rightarrow [\blambda]$ given by:
  \begin{equation*}
    \bT(k)=\left\{\begin{array}{cc}
             (k,1,1) & \quad\textrm{if}\quad 1\leq k\leq n \\
             \quad & \quad \\
             (1,1,2) & \quad\textrm{if}\quad k=n+1\\
             \quad & \quad \\
             (k+1,1,1) & \quad\textrm{if}\quad k=n+2
           \end{array}\right.
  \end{equation*}
  Then $\bT$ is an standard $(n+2)-$tableau with shape $\blambda.$ Note that the blob-addable nodes of $\bT$ are: $$(k+2,1,1),(2,1,2),(1,1,3),\dots,(1,1,l)$$ and their corresponding residues are:
  $$j,j-1,\kappa_3,\dots,\kappa_l.$$
  By strong adjacency-freeness of $\kappan$ we have that $j+1$ is not in the set of residues for blob-addable nodes of $\bT,$ then $\bj$ is $\kappa-$blob impossible (corollary \ref{coro-bt-j-impossible-sequence}).

  If $j+1\neq i_n-1$ then let $\blambda=(1^{(n)},1^{(1)},1^{(0)},\dots,1^{(0)}),$ and $\bT:\{1,\dots,n+1\}\rightarrow[\blambda],$ given by

  \begin{equation*}
    \bT(k)=\left\{\begin{array}{cc}
             (k,1,1) & \quad\textrm{if}\quad 1\leq k\leq n \\
             \quad & \quad \\
             (1,1,2) & \quad\textrm{if}\quad k=n+1
           \end{array}\right.
  \end{equation*}
  Then $\bT$ is an standard $(n+1)-$tableau with shape $\blambda.$ Note that the blob-addable nodes of $\bT$ are in this case:
  $$(k+1,1,1),(2,1,2),(1,1,3),\dots,(1,1,l)$$ and their corresponding residues are:
  $$i_n-1,j-1,\kappa_3,\dots,\kappa_l$$
  By strong adjacency-freeness of $\kappan$ we have that $j+1$ is not in the set of residues for blob-addable nodes of $\bT,$ then $\bj'=(i_1,\dots,i_n,j,j+1)$ is $\kappa-$blob impossible and so is $\bj.$ (corollary \ref{coro-bt-j-impossible-sequence} and lemma \ref{lemma-bj-possible-bjp-too}).
\end{proof}

\begin{corollary}\label{coro-bif-j-not-possible3}
  Let $s,t\in\{1,\dots,l\}$ with $s\neq t$ and $\bjf=(j_1,j_2,\dots,j_n),$ the residue sequence given by:
  \begin{equation*}
    j_r=\left\{
    \begin{array}{cc}
      \kappa_s & \quad \textrm{if}\quad r=1 \\
      \quad & \quad \\
      \kappa_t-2-r & \quad\textrm{if}\quad r>1
    \end{array}\right.
  \end{equation*}
  Then the residue sequence $\bi=(j_1,\dots,j_n,\kappa_s-1,\kappa_s,\kappa_s )$ is a $\kappa-$blob impossible residue sequence.
\end{corollary}
\begin{proof}
  We can assume that $t=1$ and $s=2,$ without loss of generality.

  We have two cases:
  \begin{enumerate}
    \item If $j_n=\kappa_2+1.$ In this case we define the standard $(n+2)-$tableau $\bT$ as follows:
    \begin{equation*}
      \bT(r)=\left\{
      \begin{array}{cc}
        (1,1,2) & \quad\textrm{if}\quad r=1 \\
        \quad & \quad \\
        (2,1,2) & \quad\textrm{if}\quad r=n+1 \\
        \quad & \quad \\
        (n,1,1) & \quad\textrm{if}\quad r=n+2 \\
        \quad & \quad \\
        (r-1,1,1) & \quad\textrm{if}\quad 1<r<n+1
      \end{array}
      \right.
    \end{equation*}
    then the set of blob-addable node of $\bT$ is: $(n+1,1,1),(3,1,2),(1,1,3),\dots,(1,1,l),$
    and the corresponding set of residues is: $\kappa_2-1,\kappa_2-2,\kappa_3,\dots,\kappa_l.$
    The strong adjacency-freeness of $\kappan$ implies that the sequence $\bj$ is a $\kappa-$blob impossible residue sequence.
    \item If $j_n\neq\kappa_2+1.$ In this case we define the standard $(n+1)-$tableau $\bT$ as follows:
    \begin{equation*}
      \bT(r)=\left\{
      \begin{array}{cc}
        (1,1,2) & \quad\textrm{if}\quad r=1 \\
        \quad & \quad \\
        (2,1,2) & \quad\textrm{if}\quad r=n+1 \\
        \quad & \quad \\
        (r-1,1,1) & \quad\textrm{if}\quad 1<r<n+1
      \end{array}
      \right.
    \end{equation*}

    then the set of blob-addable node of $\bT$ is: $(n,1,1),(3,1,2),(1,1,3),\dots,(1,1,l),$
    and the corresponding set of residues is: $j_n-1,\kappa_2-2,\kappa_3,\dots,\kappa_l.$
    The strong adjacency-freeness of $\kappan$ implies that the sequence $\bj$ is a $\kappa-$blob impossible residue sequence.
    \end{enumerate}
\end{proof}

\section{Generalized Blob Algebras}\label{ch-gen-blob-alg}

\subsection{Definition}\label{sec-definition-gen-blob}


\medskip

\medskip

\medskip

A \emph{Khovanov-Lauda diagram} $ D$, or simply a \emph{KL-diagram}, on $m$ strings consists of
$m$ points on each of two parallel edges (the top edge and the bottom edge) and $m$ strings connecting
the points of the top edge with the points
of the bottom edge. Strings may intersect, but triple intersections are not allowed. Each string
may be decorated with a finite number of dots,
but dots cannot be located on the intersection of two strings. Finally, each string is labelled with an element of $\II$.
This defines two residue sequences $t(D),b(D)  \in \II^m$ (top and bottom) associated with the diagram $D$
obtained by reading the residues of the extreme points from left to right.

In a \emph{KL-diagram} we say that two strings are \emph{sisters} if they have the same label or we say that they are \emph{cousins} if their labels differ in one. We say that two strings are \emph{not relatives} if they are not sisters nor cousins.

For example, let $\een=5$ and $m=9$. Let $D$ be the following KL-diagram:

\begin{equation*}
  \includegraphics[scale=0.17]{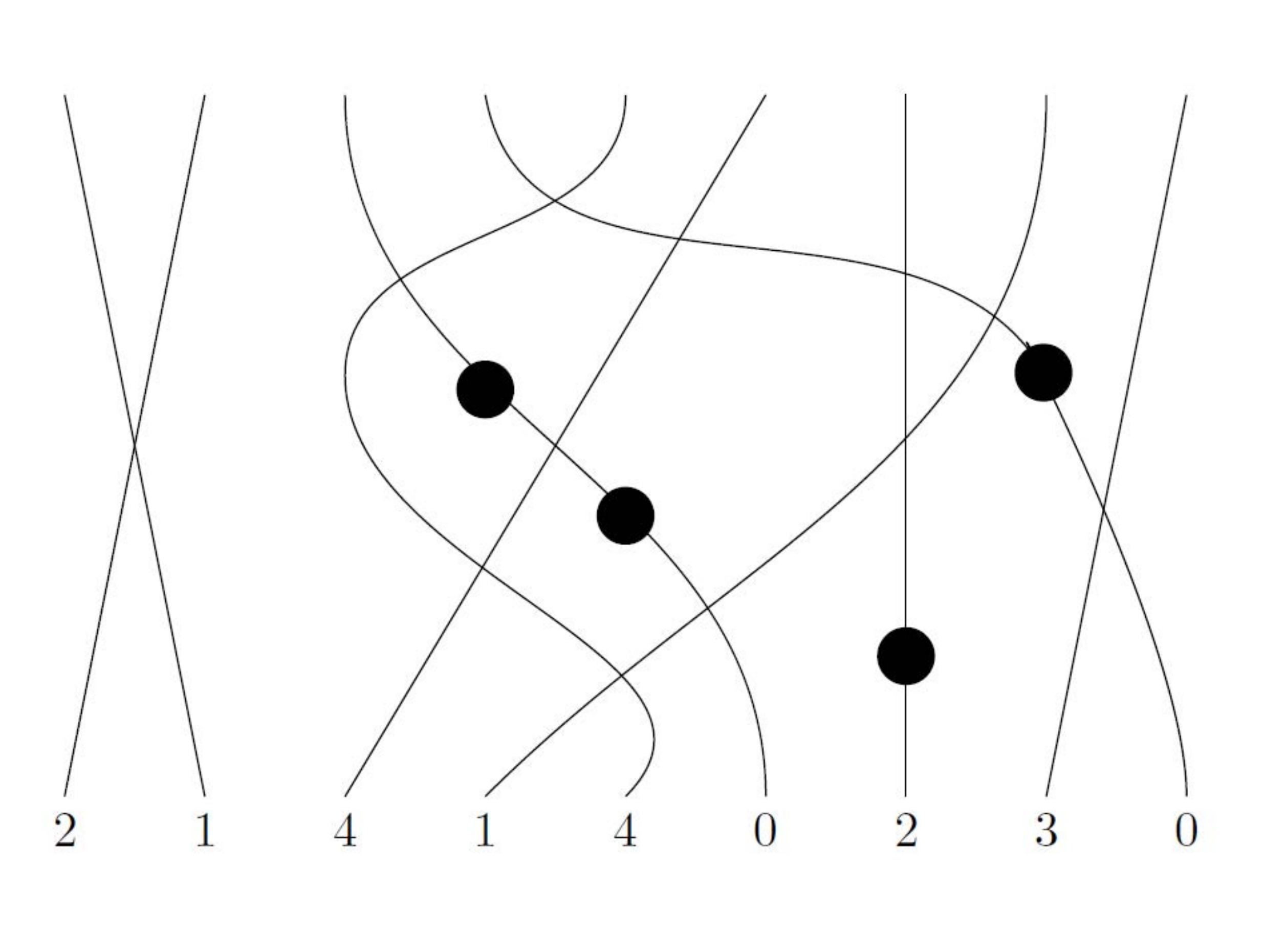}
\end{equation*}

then $b(D)=(2,1,4,1,4,0,2,3,0)$ and $t(D)=(1,2,0,0,4,4,2,1,3).$

The generalized blob algebra $\B=\mathcal{B}_{l,m}^{\F}(\kappan,\een)$ is the $ \F$-vector space consisting
of all $\F$-linear
combinations of KL-diagrams on $m$ strings, modulo planar isotopy and modulo the following relations:

\begin{itemize}
  \item \begin{equation}\label{mal-inicio}
  \begin{tikzpicture}[xscale=0.5,yscale=0.5]
  \draw[](0,0)--(0,3);
  \draw[](1,0)--(1,3);
  \node at(2.5,1.5) {$\dots$};
  \draw[](4,0)--(4,3);
  \node[below] at (0,0) {$i_1$};
  \node[below] at (1,0) {$i_2$};
  \node[below] at (4,0) {$i_m$};
  \node at(5,1.5){$\quad= 0$};
  \node at(10,1.5){if $i_1\not\in{\{\kappa_1,\dots,\kappa_l\}}${\color{black}}};
  \end{tikzpicture}
  \end{equation}
\item \begin{equation}\label{otro-mal-inicio}
\begin{tikzpicture}[xscale=0.5,yscale=0.5]
  \draw[](0,0)--(0,3);
  \draw[](1,0)--(1,3);
  \node at(2.5,1.5) {$\dots$};
  \draw[](4,0)--(4,3);
  \node[below] at (0,0) {$i_1$};
  \node[below] at (1,0) {$i_2$};
  \node[below] at (4,0) {$i_m$};
  \node at(5,1.5){$\quad=\quad 0$};
  \node at(12,1.5){if $ i_2 = i_1 +1${\color{black}}};
\end{tikzpicture}
\end{equation}
\item \begin{equation}\label{dot-al-inicio}
\begin{tikzpicture}[xscale=0.5,yscale=0.5]
  \draw[](0,0)--(0,3);
  \draw[](1,0)--(1,3);
  \draw[](4,0)--(4,3);
  \draw[fill] (0,1.5) circle [radius=0.2];
  \node[below] at (0,0) {$i_1$};
  \node[below] at (1,0) {$i_2$};
  \node[below] at (4,0) {$i_m$};
  \node at (5,1.5) {$\quad=\quad0$};
  \node at (2.5,1.5) {$\dots$};
\end{tikzpicture}
\end{equation}
\item 
\begin{equation}\label{punto-arriba}
\begin{tikzpicture}[xscale=0.5,yscale=0.6]
\draw[] (0,0)--(2,2);
\draw[] (0,2)--(2,0);
\draw[fill] (1.5,0.5) circle [radius=0.2];
\draw[] (3,0)--(5,2);
\draw[] (3,2)--(5,0);
\draw[fill] (3.5,1.5) circle [radius=0.2];
\draw[] (7.5,0)--(7.5,2);
\draw[] (9,0)--(9,2);
\node at (2.5,1) {=};
\node at (6.3,1) {$- \delta_{ij}$\quad};
\node[below] at (0,0) {$i$};
\node[below] at (2,0) {$j$};
\node[below] at (3,0) {$i$};
\node[below] at (5,0) {$j$};
\node[below] at (7.5,0) {$i$};
\node[below] at (9,0) {$j$};
\end{tikzpicture}
\end{equation}
\item \begin{equation}\label{punto-abajo}
\begin{tikzpicture}[xscale=0.5,yscale=0.6]
\draw[] (0,0)--(2,2);
\draw[] (0,2)--(2,0);
\draw[fill] (1.5,1.5) circle [radius=0.2];
\draw[] (3,0)--(5,2);
\draw[] (3,2)--(5,0);
\draw[fill] (3.5,0.5) circle [radius=0.2];
\draw[] (7.5,0)--(7.5,2);
\draw[] (9,0)--(9,2);
\node at (2.5,1) {=};
\node at (6.3,1) {$-\delta_{ij}$\quad};
\node[below] at (0,0) {$i$};
\node[below] at (2,0) {$j$};
\node[below] at (3,0) {$i$};
\node[below] at (5,0) {$j$};
\node[below] at (7.5,0) {$i$};
\node[below] at (9,0) {$j$};
\end{tikzpicture}
\end{equation}
where $ \delta_{ij} $ is the usual Kronecker's delta function.
\item \begin{equation}\label{cruce-pasa}
\begin{tikzpicture}[xscale=0.5,yscale=0.65]
\draw[] (0,0) to [out=45,in=270](2,2);
\draw[] (2,0) to [out=90,in=315](0,2);
\draw[] (1,0) to [out=135,in=225] (1,2);
\draw[] (3,0) to [out=90,in=225](5,2);
\draw[] (5,0) to [out=135,in=270](3,2);
\draw[] (4,0) to [out=45,in=315](4,2);
\draw[] (7,0)--(7,2);
\draw[] (8,0)--(8,2);
\draw[] (9,0)--(9,2);
\node at (2.5,1) {=};
\node at (6,1) {$\quad+\alpha\quad$\quad};
\node[below] at (0,0) {$i$};
\node[below] at (1,0) {$j$};
\node[below] at (2,0) {$k$};
\node[below] at (3,0) {$i$};
\node[below] at (4,0) {$j$};
\node[below] at (5,0) {$k$};
\node[below] at (7,0) {$i$};
\node[below] at (8,0) {$j$};
\node[below] at (9,0) {$k$};
\end{tikzpicture}
\end{equation}
where $$\alpha=\left\{\begin{array}{cc}
                       -1 & \quad\textrm{if}\quad i=k=j-1 \\
                       1 & \quad\textrm{if}\quad i=k=j+1 \\
                       0 & \quad\textrm{otherwise}
                     \end{array}\right.$$
\item \begin{equation}\label{lazo}
\begin{tikzpicture}[xscale=0.5,yscale=0.65]
\draw[](0,0)to[out=90,in=270](1,1);
\draw[](1,1)to[out=90,in=270](0,2);
\draw[](1,0)to[out=90,in=270](0,1);
\draw[](0,1)to[out=90,in=270](1,2);
\node[below] at (0,0) {$i$};
\node[below] at (1,0) {$j$};
\node at (2,1){$=\beta$\quad};
\draw[] (3,0)--(3,2);
\draw[] (4,0)--(4,2);
\node[below] at (3,0) {$i$};
\node[below] at (4,0) {$j$};
\node at (5,1) {$+\gamma$\quad};
\draw[] (6,0)--(6,2);
\draw[] (7,0)--(7,2);
\node[below] at (6,0) {$i$};
\node[below] at (7,0) {$j$};
\draw[fill] (7,1) circle [radius=0.15];
\node at (8,1) {$-\gamma$\quad};
\draw[] (9,0)--(9,2);
\draw[] (10,0)--(10,2);
\node[below] at (9,0) {$i$};
\node[below] at (10,0) {$j$};
\draw[fill] (9,1) circle [radius=0.15];
\end{tikzpicture}
\end{equation}
where $$\beta=\left\{\begin{array}{cc}
                       1 & \quad\textrm{if}\quad |i-j|>1 \\
                       0& \quad\textrm{otherwise}
                     \end{array}\right.$$
and $$\gamma=\left\{\begin{array}{cc}
                      1 & \quad\textrm{if}\quad j=i+1 \\
                      -1 & \quad\textrm{if}\quad j=i-1 \\
                      0 &\quad\textrm{otherwise{\color{black}.}}
                    \end{array}\right.$$
\end{itemize}

\medskip
Multiplication $DD'$ between two diagrams $D$ and $D'$ is defined by vertical concatenation with $D$ above of $D'$ if $b(D)=t(D')$.
If $b(D)\neq t(D')$ the product is defined to be zero.
The product is extended by linearity to the rest of elements.

\medskip
There is an alternative definition for $\B,$ in terms of generators $e(\bi),\psi_j,y_j$ and relations (see \cite{LiPl} for example). The connection between generators and diagrams is given as follows:
\begin{equation*}
\begin{tikzpicture}[xscale=0.5,yscale=0.6]
  \draw[](0,0)--(0,2);
  \draw[](1,0)--(1,2);
  \node at(2,1) {$\dots$};
  \node at (-1,1){$=$};
  \node at (-2,1){$e(\bi)$};
  \draw[](3,0)--(3,2);
  \node[below] at (0,0) {$i_1$};
  \node[below] at (1,0) {$i_2$};
  \node[below] at (3,0) {$i_m$};
  \node[below] at (3.3,1) {$,$};
\end{tikzpicture}
\begin{tikzpicture}[xscale=0.5,yscale=0.6]
  \draw[](0,0)--(0,2);
  \draw[](1.8,0)--(1.8,2);
  \draw[fill] (1.8,1) circle [radius=0.15];
  \node at(0.9,1) {$\dots$};
  \node at(2.8,1) {$\dots$};
  \node at(-1,1){$=$};
  \node at(-2.5,1){$y_re(\bi)$};
  \draw[](3.5,0)--(3.5,2);
  \node[below] at (0,0) {$i_1$};
  \node[below] at (1.8,0) {$i_r$};
  \node[below] at (3.5,0) {$i_m$};
  \node[below] at (3.8,1) {$,$};
\end{tikzpicture}
\begin{tikzpicture}[xscale=0.5,yscale=0.6]
  \draw[](1,0)--(1,2);
  \draw[](3,0)--(4,2);
  \draw[](3,2)--(4,0);
   \node at(2,1) {$\dots$};
  \node at(4.8,1) {$\dots$};
  \node at(0,1){$=$};
  \node at(-1.5,1){$\psi_re(\bi)$};
  \draw[](6,0)--(6,2);
  \node[below] at (1,0) {$i_1$};
  \node[below] at (3,0) {$i_r$};
  \node[below] at (4,0) {$i_{r+1}$};
  \node[below] at (6,0) {$i_m$};

\end{tikzpicture}
\end{equation*}

For the rest of this article we will refer to generators $e(\bi)$ as \emph{idempotents}, generators $y_r$ as \emph{dots} and generators $\psi_r$ as \emph{crosses}

We denote by $ \ast $ the anti-involution of $ \B$, defined on diagrams as vertical flip and extended by linearity. We also have the following degree function on generators:

\begin{equation}\label{Degree-function-v1}
  \deg{e(\bi)}=0,\quad \deg{y_r}=2,\quad \deg{\psi_re(\bi)}=\left\{
  \begin{array}{cc}
    1 & \quad \textrm{if}\quad i_{r}\neq i_{r+1}\pm 1 \\
    -2 & \quad \textrm{if}\quad i_{r}= i_{r+1} \\
    0 & \quad \textrm{otherwhise}
  \end{array}
  \right.
\end{equation}

\medskip

We next recall some useful relations that can be derived directly from the definitions (details in \cite{Lobos-Ryom-Hansen}):


\begin{itemize}
\item \begin{equation}
        \includegraphics[scale=0.3]{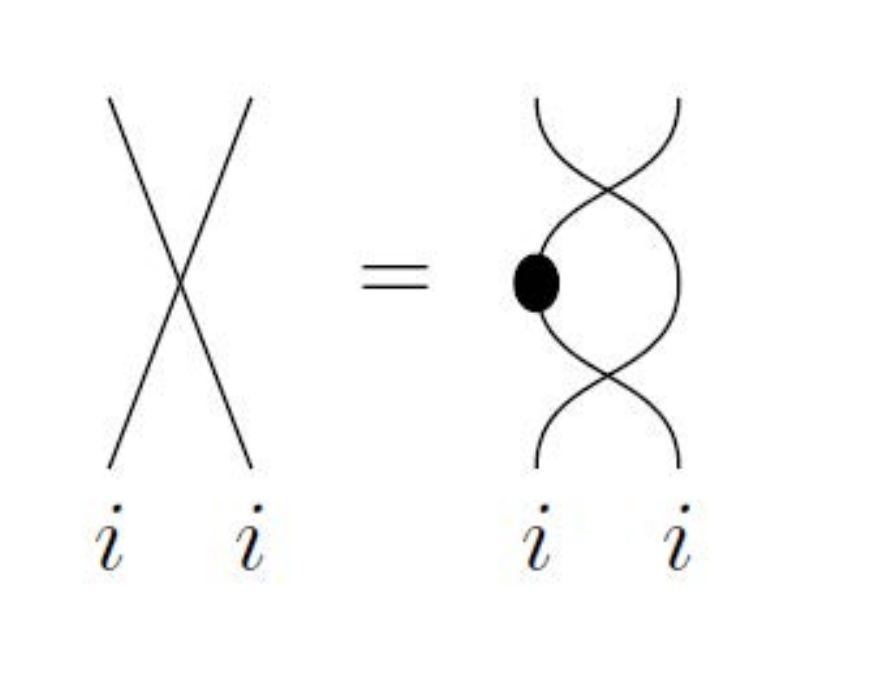}\quad   \includegraphics[scale=0.33]{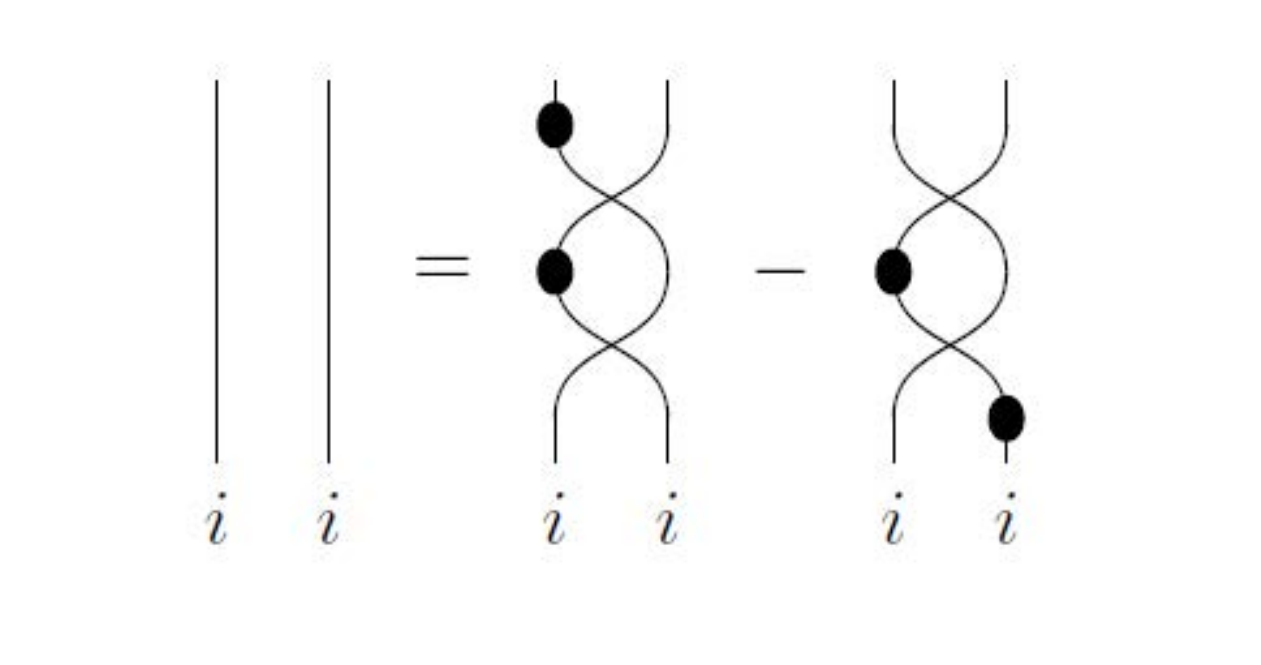}
      \end{equation}

\item

\begin{equation}\label{dot-salta}
        \includegraphics[scale=0.3]{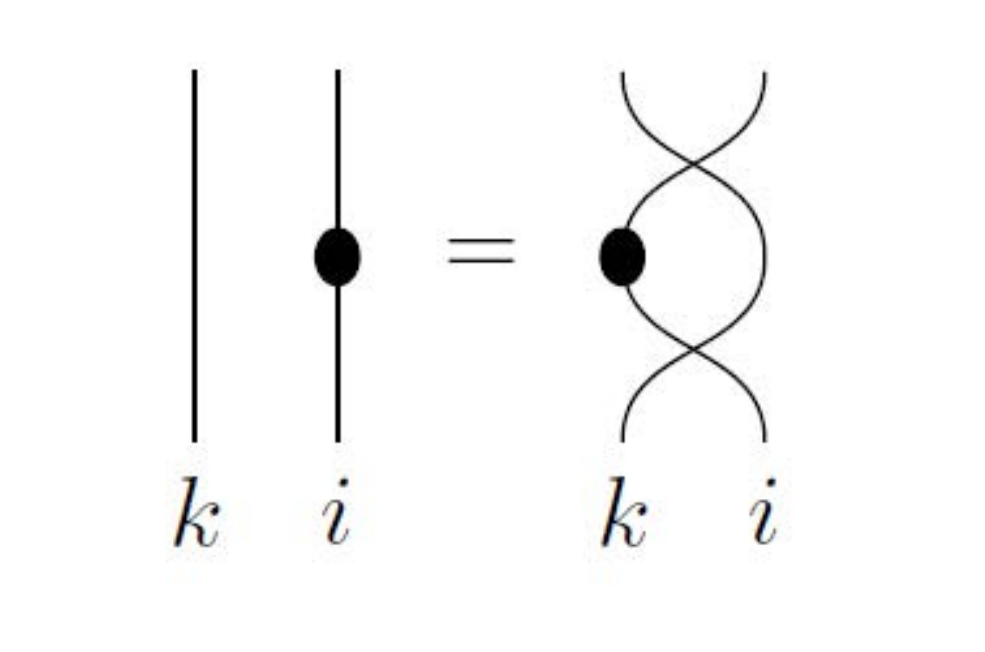}\quad   \includegraphics[scale=0.33]{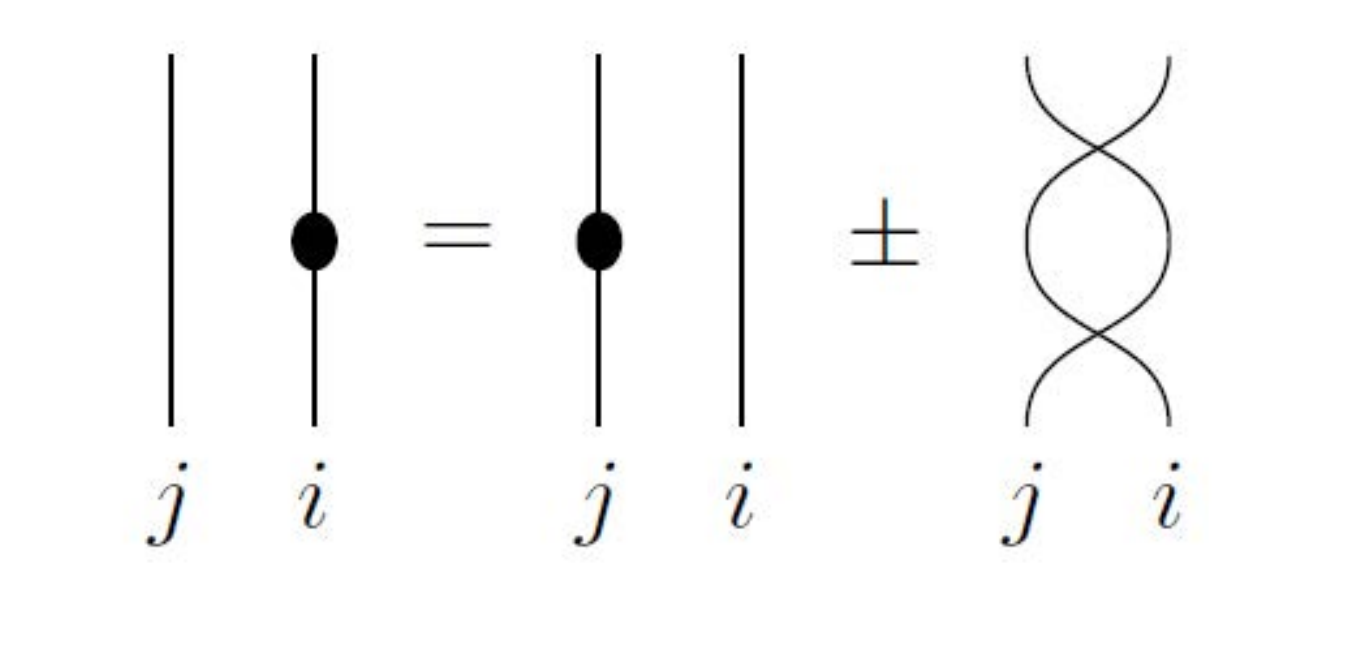}
      \end{equation}

where $|i-k|>1$ and $|i-j|=1$ (the positive sign appears when $j=i-1$ and the negative sign when $j=i+1.$)

\item
If $k=i+1$ and $j=i-1$ then we have
\begin{equation}\label{trio}
 \includegraphics[scale=0.37]{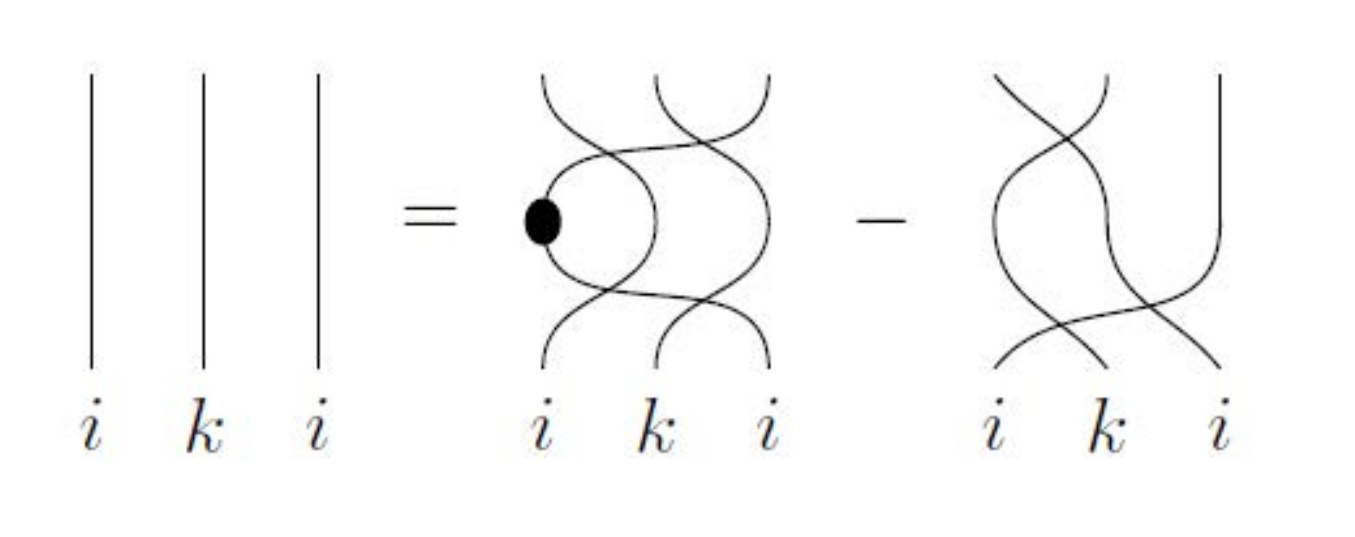}\quad   \includegraphics[scale=0.37]{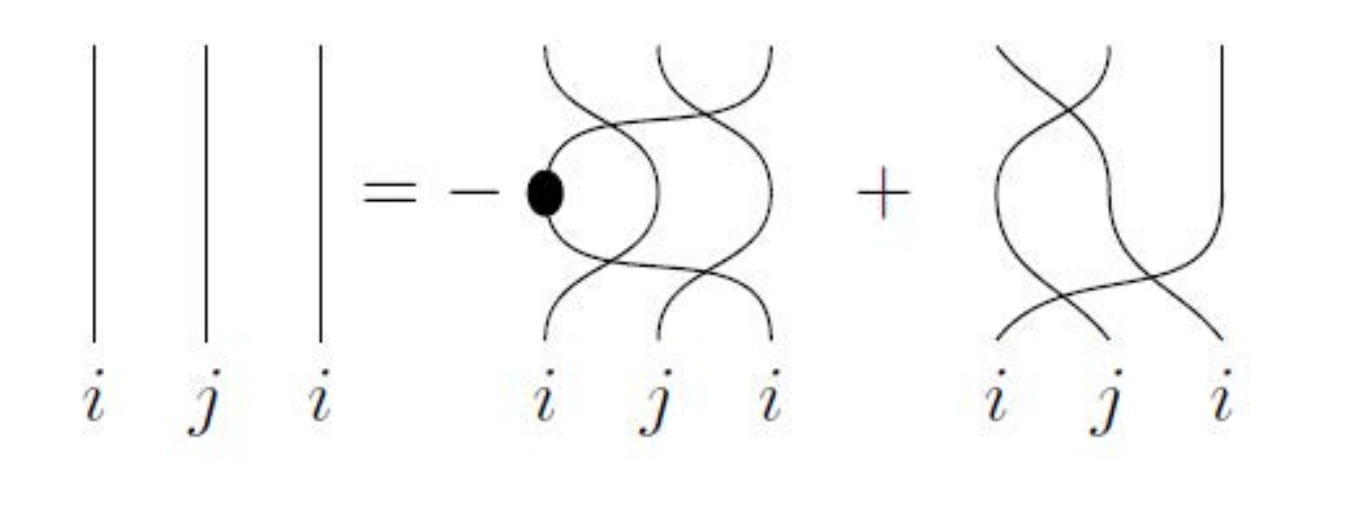}
\end{equation}
\end{itemize}

The following definition will be very useful for the rest of the article:
\begin{definition}\label{def-L-general}
  For each $1\leq r \leq m$ we define the element
  \begin{equation}
    L_r=\left\{
    \begin{array}{cc}
      -y_1 & \quad \textrm{if}\quad r=1 \\
      \quad & \quad \\
      y_{r-1}-y_r & \quad \textrm{if}\quad r>1.
    \end{array}
    \right.
  \end{equation}
\end{definition}

\subsection{A Graded Cellular basis for $\B.$}\label{sec-graded-cellular-basis-of-B}

In this section we shall refer to the concepts of \emph{Cellular basis} defined by Graham and Lehrer (see \cite{GL}) and \emph{Graded Cellular basis} defined by Hu and Mathas (see \cite{hu-mathas}).

The next lemma comes from \cite{hu-mathas}, (this is indeed an adapted version for generalized blob algebras, see \cite{Lobos-Ryom-Hansen} for details):

\begin{lemma}\label{lemma-first-idemp-annihilation}
  Given a residue sequence $\bi\in \II^{m}$ such that the corresponding idempotent $e(\bi)\in \B$ is not zero. Then $\bi=\bi^{\bT}$ for some one-column standard tableau $\bT.$
\end{lemma}

\begin{corollary}
  If $\bi$ is a $\kappa-$blob impossible residue sequence, then $e(\bi)=0$ in $\B.$
\end{corollary}

 We can associate to each tableau $\bT$ an element $\Psi_{\bT}^{\blambda}\in\B,$ as follows:

 Let $\boldsymbol{d}(\bT)\in\Si_{m}^{\ast}$ be the official reduced word of $d(\bT)$ (see section \ref{sec-general-combinatorics}).
 If $\boldsymbol{d}(\bT)=s_{j_{1}}\cdots s_{j_{r}},$ then we define:

 \begin{equation}\label{def-of-Psidt1}
  \Psi_{\bT}^{\blambda}=e_{\blambda}\psi_{j_{1}}\cdots\psi_{j_{r}}.
\end{equation}
where $e_{\blambda}=e(\bi^{\bT^{\blambda}}).$

Moreover given two tableaux $\Bs,\bT\in \std(\blambda),$ we define the element $\Psi_{\Bs,\bT}^{\blambda}\in\B,$ as follows:

\begin{equation}\label{def-of-Psidt1}
  \Psi_{\Bs,\bT}^{\blambda}=(\Psi_{\Bs}^{\blambda})^{\ast}\Psi_{\bT}^{\blambda}.
\end{equation}

The following is the main result of \cite{Lobos-Ryom-Hansen}
\begin{theorem}\label{theo-cellular-basis}
  The set $$\{\Psi_{\bT,\Bs}^{\blambda}:\bT,\Bs\in\std(\blambda),\blambda\in \OnePar\}$$
  is a graded cellular basis of $\B$ under the order $\rhd$ and the degree function defined by $\deg(\bT)=\deg(\Psi_{\bT}^{\blambda}).$
\end{theorem}

As was mentioned in remark \ref{remark-new-strong-adj-free}, the construction of the cellular basis of \cite{Lobos-Ryom-Hansen} was obtained using the (original) concept of strong adjacency-freeness. Although in this article we are using a modified version of this concept, this new version is absolutely compatible with all the results of \cite{Lobos-Ryom-Hansen}.

\subsection{Idempotent truncations of $\B$}\label{sec-idemp-trunc-of-B}

By lemma \ref{lemma-first-idemp-annihilation} we have that
\begin{equation*}
  \B=\bigoplus_{\bi,\bj}e(\bi)\B e(\bj)
\end{equation*}
where the direct sum runs over the set of all pairs $\bi,\bj$ of $\kappa-$blob possible residue sequences, and

\begin{equation*}
  e(\bi)\B e(\bj)=\{e(\bi)a e(\bj):a\in \B\}.
\end{equation*}

In particular if $\bi=\bj,$ the subspace $e(\bi)\B e(\bi)$ is indeed an algebra contained in $\B$ (note that they do not share the same identity element).

For each $\kappa-$blob possible residue sequence $\bi\in \II ^{m},$ we denote $\B (\bi)=e(\bi)\B e(\bi).$  We say that $\B(\bi)$ is an algebra obtained by an idempotent truncation from $\B$ or simply an \emph{idempotent truncation of} $\B.$ If we denote by $\std(\bi)$ the set of all standard tableaux $\bT$ (any shape) such as $\bi^{\bT}=\bi,$ and we denote $\std_{\blambda}(\bi)=\std(\blambda)\cap\std(\bi).$ Then it comes directly from theorem \ref{theo-cellular-basis} that:

\begin{corollary}\label{coro-cellular-basis-truncated}
  The set $$\{\Psi_{\bT,\Bs}^{\blambda}:\bT,\Bs\in\std_{\blambda}(\bi),\blambda\in \OnePar\}$$
is a graded cellular basis for $\B(\bi).$
\end{corollary}
\begin{proof}
  In \cite{Lobos-Plaza-Ryom-Hansen} there is a proof for certain particular cases in level $l=2.$ The general case is obtained analogously.
\end{proof}

\section{The fundamental vertical truncation of $\B$}\label{ch-fundamental-idemp-trunc}

We are interested in the study of an idempotent truncation $\B(\bif)$ of $\B,$ corresponding to a particular residue sequence $\bif$ that we call \emph{the fundamental vertical sequence}. More precisely we are interested in the study of a particular commutative subalgebra of $\B (\bif).$ (later in section \ref{ch-gelfand}).

\subsection{The fundamental vertical sequence}\label{sec-fund-resid-sec}

For the rest of the article, unless otherwise indicated, we will assume that $m=\epsilon+k\een$ for certain positive integer $k>1.$

If $\kappa=(\kappa_1,\dots,\kappa_l),$ we define the \emph{fundamental vertical sequence} as
\begin{equation}\label{fund-res-seq}
 \bif=(i_1,\dots,i_m),\quad \textrm{where}\quad i_j=\kappa_1-j+1 \mod \een.
\end{equation}

The fundamental vertical residue sequence $\bif$ is $\kappa-$blob possible. To see this, note that $\bif=\bi^{\bT},$ where $\bT=\bT^{\underline{\blambda}}$ and $\underline{\blambda}=(1^{(m)},0,\dots,0).$

We define the following numbers:

\begin{equation}\label{def-b-numbers}
\left\{
\begin{array}{cc}
  \epsilon=\hat{\kappa}_1-\hat{\kappa}_l+\een & \quad \\
  \quad & \quad \\
   b_{j}=\hat{\kappa}_{l-j}-\hat{\kappa}_{l-j-1} & \textrm{if}\quad 0\leq j \leq l-3\\
  \quad & \quad \\
 b_{l-2}=\hat{\kappa}_2-\hat{\kappa}_l+\een & \quad
\end{array}
\right.
\end{equation}
(Recall that the representatives $\hat{\kappa}_1,\dots,\hat{\kappa}_l $ satisfy the following relation $\hat{\kappa}_1<\hat{\kappa}_2\cdots<\hat{\kappa}_l < \hat{\kappa}_1+\een$).

We also define recursively the following numbers:

\begin{equation}\label{def-m-numbers1}
\left\{
\begin{array}{cc}
  m_{(1,0)}=\epsilon+1 & \quad \\
  \quad & \quad \\
   m_{(1,j)}=m_{(1,j-1)}+b_{j-1}& 1\leq j\leq l-2
\end{array}
\right.
\end{equation}

and more generally

\begin{equation}\label{def-m-numbers2}
\begin{array}{cc}
  m_{(r,j)}=m_{(r-1,j)}+\een& 0\leq j\leq l-2,\quad 1< r\leq k
\end{array}
\end{equation}

Now we define the following intervals of integer numbers, that we call \emph{blocks}:

\begin{equation}\label{def-blocks1}
\left\{
  \begin{array}{cc}
    N =[1:\epsilon]&\quad  \\
    \quad & \quad \\
     B_{(r,j)}=[m_{(r,j)}:m_{(r,j+1)}-1]& 0\leq j<l-2,\quad 1\leq r\leq k \\
    \quad & \quad  \\
   B_{(r,l-2)}=[m_{(r,l-2)}:m_{(r+1,0)}-1]& 1\leq r< k
  \end{array}
  \right.
\end{equation}


The set of blocks $B_{(r,j)}$ is totally ordered by the relation $\succ$ given by:
\begin{equation*}
  B_{(r,j)}\succ B_{(r',j')}\quad \textrm{if and only if}\quad m_{(r,j)}>m_{(r',j')}.
\end{equation*}

The restriction of $\bif=(i_1,\dots,i_m)$ to the different blocks is given by:

\begin{equation*}
\left\{
  \begin{array}{cc}
    \bif|_N=  (i_1,\dots,i_{\epsilon})&\quad \\
    \quad&\quad \\
    \bif|_{(B(r,j))}=(i_{a},\dots,i_{b}) & \quad\textrm{where}\quad B_{(r,j)}=[a:b] \\
  \end{array}\right.
\end{equation*}

It is not difficult to see that $\bif|_{(B(r,j))}=\bif|_{(B(r',j'))}$ if and only if $j=j'.$

\begin{example}
  If $\een=13,l=4,m=32$ and $\kappa=(0,2,5,7)$ then $\epsilon=6,$ $b_{0}=2,$ $b_{1}=3,$ $b_{2}=8.$
  $m_{(1,0)}=7,$ $m_{(1,1)}=9,$ $m_{(1,2)}=12,$ $m_{(2,0)}=20,$ $m_{(1,1)}=22,$ $m_{(1,2)}=25.$
  Then blocks are $N =[1:6],$ $B_{(1,0)}=[7:8],$ $B_{(1,1)}=[9:11],$ $B_{(1,2)}=[12:19],$ $B_{(2,0)}=[20:21],$ $B_{(2,1)}=[22:24],$ $B_{(2,2)}=[25:32].$

  It will be useful for small cases of $l$ to add colors on residues in order to identify different blocks $B_{(r,j)}$ with the same index $j$.
  In this case $\bif|_N=(0,12,11,10,9,8),$ $\bif|_{B(1,0)}=\bif|_{B(2,0)}={\color{blue}(7,6)},$ $\bif|_{B(1,1)}=\bif|_{B(2,1)}={\color{red}(5,4,3)},$ $\bif|_{B(1,2)}=\bif|_{B(2,2)}={\color{green}(2,1,0,12,11,10,9,8)}.$

  Therefore

   $$\bif=\bif|_N\otimes \bif|_{B_{(1,0)}}\otimes\bif|_{B_{(1,1)}}\otimes\bif|_{B_{(1,2)}}\otimes
  \bif|_{B_{(2,0)}}\otimes\bif|_{B_{(2,1)}}\otimes\bif|_{B_{(2,2)}}.$$

  $$\bif=(0,12,11,10,9,8,{\color{blue}7,6},{\color{red}5,4,3},{\color{green}2,1,0,12,11,10,9,8},
  {\color{blue}7,6},{\color{red}5,4,3},{\color{green}2,1,0,12,11,10,9,8})$$
\end{example}

\begin{lemma}{\label{lemma-sum-of-bj}}
  \begin{equation*}
    \sum_{j=0}^{l-2}b_j=\een
  \end{equation*}
\end{lemma}
\begin{proof}
  It follows directly, applying a telescopic argument in the sum.
\end{proof}

\begin{lemma}\label{lemma-residues-mrj}
 Let $\bif=(i_1,\dots,i_m)$, then:
  \begin{enumerate}
    \item $i_s=\kappa_{1}$ if and only if $s=1+h\een$ for some $0\leq h<k.$

    \item $ i_{s}\in\{\kappa_2,\dots,\kappa_l\}$ if and only if $ s=m_{(r,j)},$ for some pair $(r,j)$ such that
    $1\leq r\leq k$ and $ 0\leq j\leq l-2.$

    Moreover if $s=m_{(r,j)},$ then $i_s=\kappa_{l-j}.$
  \end{enumerate}

 \end{lemma}

\begin{proof}
   By definition $i_{s}=\kappa_1-s+1\mod \een,$ then $\bif$ is periodic with period $\een,$ in the sense that $i_{j+\een}=i_j\mod \een,$ for all $j.$ Moreover
   $i_s=i_{s'}\quad \textrm{if and only if}\quad s=s'\mod \een.$
   Therefore we have
      \begin{enumerate}
        \item $i_{s}=\kappa_1$ if and only if $s=1\mod \een,$ that is $s=1+h\een$ for some $0\leq h<k.$
        \item If $s=m_{(r,j)}$ then $i_s=i_{m_{(1,j)}}$ (see equation \ref{def-m-numbers2}).

        If $j=0,$ then
        $$i_s=i_{\epsilon+1}=\kappa_1-(\hat{\kappa}_1-\hat{\kappa}_l+\een+1)+1=\kappa_l\mod \een.$$
        If $0<j\leq l-2$ then $i_s=i_{s'}$ where
        $$s'=m_{(1,j)}=m_{(1,0)}+\sum_{t=0}^{j-1}b_t=m_{(1,0)}+\sum_{t=0}^{j-1}(\hat{\kappa}_{l-t}-\hat{\kappa}_{l-t-1})$$
        by a telescopic argument we obtain:
        $$s'=\epsilon+1+\hat{\kappa}_{l}-\hat{\kappa}_{l-j}=\hat{\kappa_1}-\hat{\kappa_l}+\een+1+\hat{\kappa}_{l}-\hat{\kappa}_{l-j},$$
        therefore
        $$i_s=\kappa_1-{s'}+1=\kappa_{l-j}\mod \een.$$
      \end{enumerate}
\end{proof}

\subsection{Tableaux defined by blocks}\label{sec-tab-def-blocks}

In this subsection we describe the set $\std(\bif).$ Recall that we are considering $m$ with the form $m=\epsilon+k\een,$ for some $k\in \mathbb{Z^{+}},$

In the set of nodes $\mathbb{N}\times\{1\}\times \{1,\dots,l\}$ we define certain subsets that we call \emph{Blocks of nodes}:

We define the block of nodes $[N]$ as:
\begin{equation}\label{block-nod-N}
  [N]=\{(r,1,1):1\leq r\leq \epsilon\}
\end{equation}

For each $0\leq j \leq l-2,$ we define the blocks of \emph{type} $j$ as any set of nodes of the form:

\begin{equation}\label{block-nod-B-type-j}
  \{(r,1,h):x\leq r\leq x+b_j-1; \res(x,1,h)=\kappa_{l-j} \}
\end{equation}
We denote any of this blocks by $B^{(j)}$

Note that if $\res(B^{(j)})$ denotes the sequences of residues of the  nodes (ordered from top to bottom) in $B^{(j)},$ then $\res(B^{(j)})=\bif|_{B(1,j)}.$ We also have $\res([N])=\bif|_N.$



\begin{lemma}\label{lemma-tableux-blocks}
  If $\bT\in\std(\bif),$ then:
  \begin{enumerate}
    \item $\bT(N)=[N]$
    \item $\bT(B_{(r,j)})$ is a block of type $j.$
  \end{enumerate}
\end{lemma}

\begin{proof}
  Recall that $\boldsymbol{\kappa}=(\boldsymbol{\kappa}_1,\dots,\boldsymbol{\kappa}_l)$ is a strongly adjacency-free multicharge, and $\bif|_N=(\kappa_1,\kappa_1-1,\dots,\kappa_1-\epsilon+1).$

  If  $\bT\in\std(\bif)$ then:

   \begin{enumerate}
     \item $\bT(1)=(1,1,h),$ since $\bT$ is standard. But $\res(1,1,h)=\kappa_1$ if and only if $h=1,$ then $\bT(1)=(1,1,1).$  Analogously  $\bT(2)=(2,1,1),$ since this is the only addable node $\gamma$ of $\bmu_1=(1^{(1)},1^{(0)},\dots,1^{(0)})$ with $\res(\gamma)=\kappa_1-1.$  Following in this way we can see that for each $1\leq s\leq \epsilon$ there is only one possible node $\gamma$ such that $\bT(s)=\gamma,$ that is $\gamma=(s,1,1).$
       Then we conclude that $\bT(N)=[N].$
     \item Since $\bT\in\std(\bif),$ we have for $s=m_{(r,j)}$ that $i_s=\kappa_{l-j}$ (lemma \ref{lemma-residues-mrj}) then $\bT(s)=\gamma$ for some node $\gamma$ such that $\res(\gamma)=\kappa_{l-j}.$ By definition, $\gamma$ is the \emph{first} (highest) node on a block of type $j.$

         In particular for $s=m_{(1,0)}$ there are exactly two possible addable nodes of $\bmu=(1^{(\epsilon)},1^{(0)},\dots,1^{(0)}),$ with residue equal to $i_s=\kappa_l,$ they are $\gamma_1=(m_{(1,0)},1,1)$ and $\gamma_2=(1,1,l).$
         If $\bT(s)=\gamma_1$ then the images (under $\bT$) of the elements of the block $B_{(1,0)}$ are totally determined:

          If $a=s+\delta\in B_{(1,0)}$ then $\bT(a)=(s+\delta,1,1),$ and:
         $$\bT(B_{(1,0)})=\{(r,1,1):r\in B_{(1,0)} \},$$
         and this is a block of type $0.$

         Analogously, if $\bT(s)=\gamma_2$ then the images (under $\bT$) of the elements of the block $B_{(1,0)}$ are totally determined:

          If $a=s+\delta\in B_{(1,0)}$ then $\bT(a)=\{(1+\delta,1,1)\},$ and:
         $$\bT(B_{(1,0)})=\{(r,1,l):1\leq r<m(1,1) \},$$
         and this is a block of type $0.$

         If $s=m_{(1,j)}$ with $j>0$ and we assume that for each $1\leq j'<j$ we have that $\bT(B_{(1,j')})$ is a block of type $j'.$ Then there are exactly two addable nodes of $\bmu=\shape(\bT|_{m_{(1,j)}-1})$ with residue equal to $\kappa_{l-j}.$ They are, $\gamma_1,$ the one below of the block $B^{(j-1)}$ corresponding to $\bT(B(1,j-1))$ and $\gamma_2=(1,1,l-j).$ Any other addable node of $\bmu$ has residue equal to $\kappa_x$ with $x\neq l-j.$

         For any choice we take for $\bT(s),$ the rest of the images for elements of the  block $B_{(1,j)}$ will be forced to complete the block of type $j$ whose \emph{first} node is the chosen one for $\bT(s).$ Then $\bT(B_{(1,j)})$ is a node of type $j.$

         Finally if $s=m(r,j)$ with $r>1$ and $0\leq j \leq l-2$ and we assume that for any block $B_{(r',j')}$ such that $ B_{(r,j)}\succ B_{(r',j')}$ we have that $\bT(B_{(r',j')})$ is a block of type $j'.$ Then for $\bT(s)$ there are two possible addable nodes of $\bmu=\shape(\bT|_{s-1})$ with residue $\kappa_{l-j}.$ Let's call them $\gamma_1$ and $\gamma_2$ respectively. They depend on the two following situations:
         \begin{enumerate}
           \item If $j=0$ we have that $\gamma_1$ is the addable node below to $\bT(B_{(r-1,l-2)})$ and $\gamma_2$ could be either, equal to $(m_{(1,0)},1,1)$ (if it haven't been used yet) or is an addable node below to some block $B^{(l-2)}$ used by   $\bT(B_{(r',l-2)})$ for some $r'<r-1.$
           \item If $j>0$ we have that $\gamma_1$ is the addable node below to $\bT(B_{(r,j-1)})$ and $\gamma_2$ could be either, equal to $(1,1,l-j)$ (if it haven't been used yet) or is an addable node below to some block $B^{(j-1)}$ used by   $\bT(B_{(r',j-1)})$ for some $r'<r-1.$
         \end{enumerate}
         In any case, given a choice for $\bT(s)$ the rest of the block $B_{(r,j)}$ will be totally determined as above. Then $\bT(B_{(r,j)})$ is a block of type $j.$
   \end{enumerate}
   \end{proof}

\begin{corollary}\label{coro-tableaux-blocks-bifurcation}
  Let $\bT\in\std(\bif),$ $1\leq r \leq k,$ $0\leq j \leq l-2$ and  $a=m_{(r,j)}-1.$
  Then the set of residues of the addable nodes for $\bT|_a$ is equal to $\left\{\kappa_2,\dots,\kappa_l\right\}.$
  Moreover, there are exactly two addable nodes for $\bT|_a$ with residue equal to $\kappa_{l-j}$ and exactly one with residue $\kappa_i,$ for each $i\neq 1, l-j.$
\end{corollary}

By corollary \ref{coro-tableaux-blocks-bifurcation}, in the process of building a tableau in $\std(\bif)$ there are relevant steps where we have to chose between two nodes with the same residue. We can differentiate them as the highest and lowest with respect to the order $\rhd.$  The other steps are totally determined by the residue sequence $\bif$ and our previous choices. Therefore we introduce the following notation:

\begin{definition}\label{def-notation-tableaux-blocks}
  Let $$M=\left\{m_{(r,j)}:1\leq r \leq k; 0\leq j\leq l-2\right\}.$$ For any subset $A\subset M$ we define the tableaux $\bT\langle A\rangle$ as follows:
  \begin{enumerate}
    \item If $A = \emptyset,$ we define $\bT\langle A\rangle$ as the only possible standard tableaux of shape $\blambda=(1^{(m)},0,0,\dots,0).$
    \item If $A \neq \emptyset,$ we define recursively:
    \begin{enumerate}
      \item $\bT\langle A\rangle(N)=[N].$
      \item If $a=m_{(r,j)}\in A,$ then $\bT\langle A\rangle(a)$ is defined as the highest addable node of $\bT\langle A\rangle|_{a-1}$ with residue $i_a.$ Otherwise, if $a=m_{(r,j)}\notin A,$ then $\bT\langle A\rangle(a)$ is defined as the lowest addable node of $\bT\langle A\rangle|_{a-1}$ with residue $i_a.$
      \item If $a>\epsilon$ and $a\neq m_{(r,j)}$ for any pair $(r,j),$ then $\bT\langle A\rangle(a)$ is totally determined by the nodes $\bT\langle A\rangle(m_{(r,j)})$ with $m_{(r,j)}<a.$
    \end{enumerate}
  \end{enumerate}
  We also introduce the following notation:

    Let $1\leq r \leq k,$ $0\leq j \leq l-2$, let $M_{m_{(r,j)}}=\left\{m_{(s,t)}:m_{(s,t)}\leq m_{(r,j)}\right\},$ then we define:
    \begin{equation*}
      \bT_{(r,j)}=\bT\langle M_{m_{(r,j)}}\rangle,\quad \blambda_{(r,j)}=\shape(\bT_{(r,j)}),\quad \bT^{(r,j)}=\bT^{\blambda_{(r,j)}}.
    \end{equation*}
\end{definition}

\begin{example}
If $\een=13,l=4$ $\kappa=(0,4,6,10)$ and $m=29$ then:

\begin{equation*}
  \includegraphics[scale=0.27]{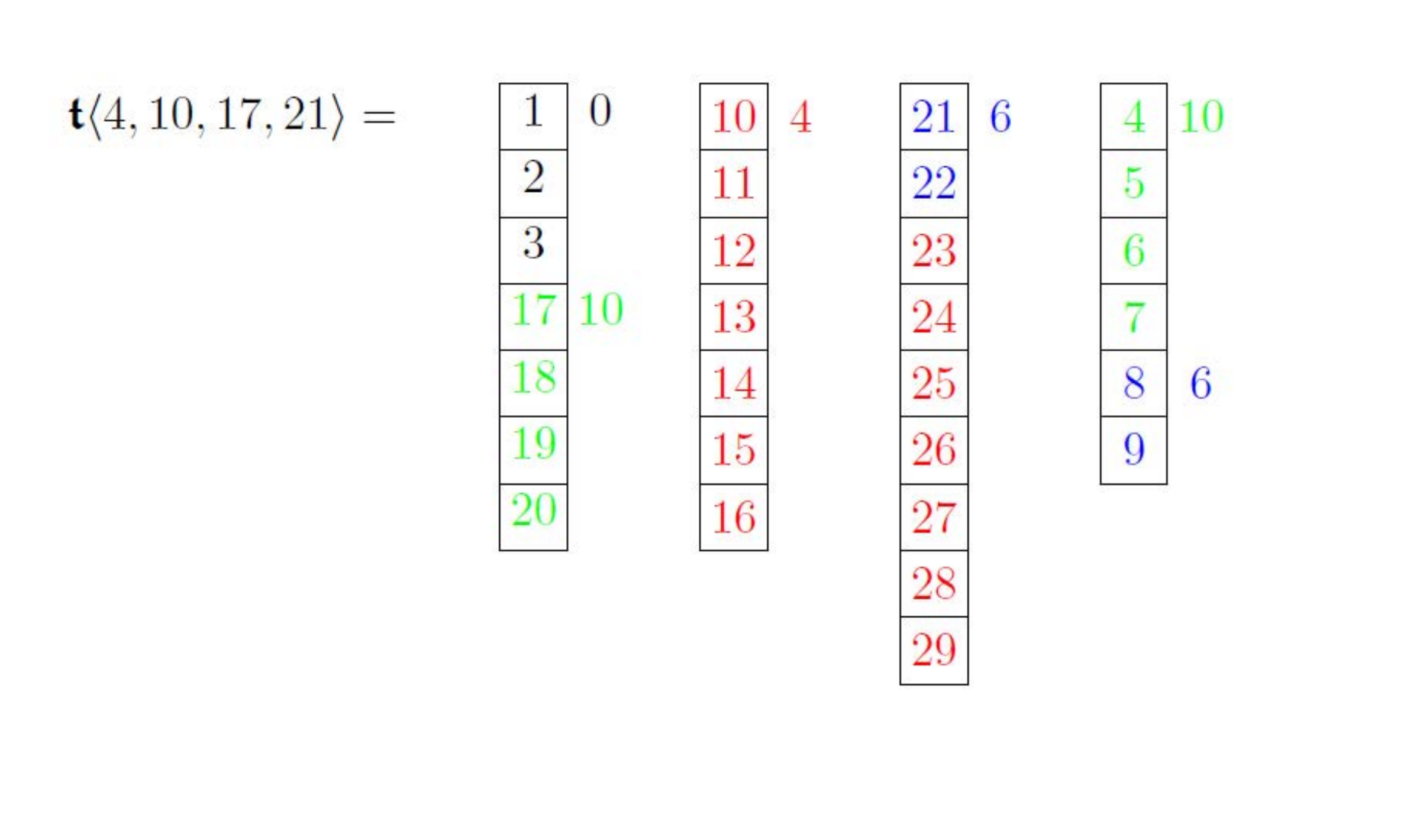}\includegraphics[scale=0.27]{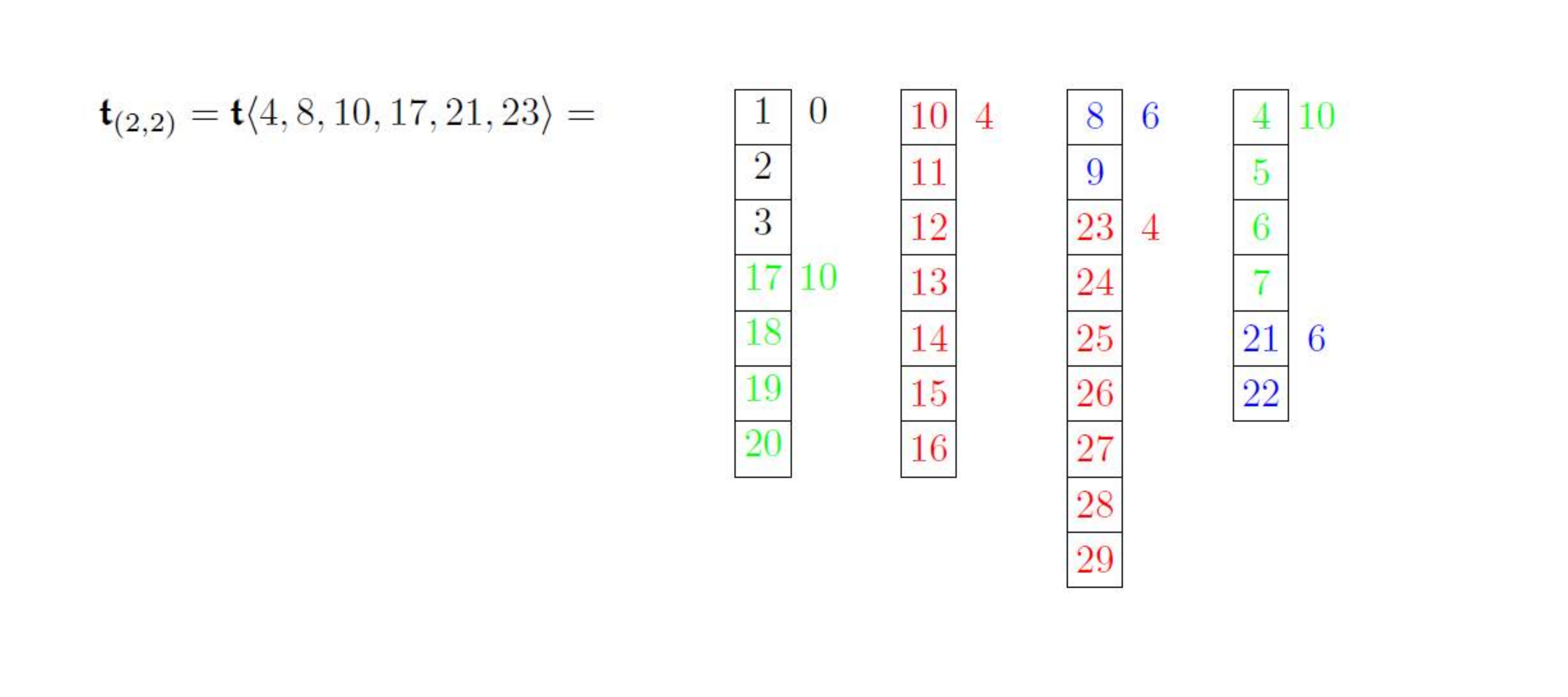}
\end{equation*}
\end{example}

\begin{corollary}\label{coro-cardinality-of-std-bif}
Let $n=k(l-1)$ then
  \begin{equation*}
    |\std(\bif)|=2^{n}.
  \end{equation*}
\end{corollary}
\begin{proof}
  By lemma \ref{lemma-tableux-blocks}, corollary \ref{coro-tableaux-blocks-bifurcation} and definition \ref{def-notation-tableaux-blocks}, the assignment $A\mapsto \bT\langle A\rangle$ define a bijection between the power set of $M=\left\{m_{(r,j)}:1\leq r \leq k,\quad 0\leq j\leq l-2\right\}$ and $\std(\bif).$
\end{proof}

\subsection{The fundamental vertical truncation and its cellular basis}\label{sec-fund-idemp-and-cellular-basis}

We define the \emph{fundamental vertical truncation} of $\B$ as the algebra $\B(\bif)$ obtained by idempotent truncation with respect to the fundamental vertical sequence $\bif.$

The following is a particular case of corollary \ref{coro-cellular-basis-truncated}:

\begin{corollary}\label{coro-cellular-basis-truncated-2}
  The set $$\{\Psi_{\bT,\Bs}^{\blambda}:\bT,\Bs\in\std_{\blambda}(\bif),\blambda\in \OnePar\}$$
is a graded cellular basis for $\B(\bif).$
\end{corollary}

In particular we have

\begin{corollary}\label{coro-dimension-B-bif}
  \begin{equation*}
    \dim(\B(\bif))=\sum_{\blambda\in\OnePar}|\std_{\blambda}(\bif)|^{2}.
  \end{equation*}
\end{corollary}


Following the notation introduced in definition \ref{def-notation-tableaux-blocks-2} we define:

\begin{definition}\label{def-H-factors}
  Let $1\leq r \leq k,$ $0\leq j \leq l-2.$ For each pair $(s,t)$ such that $m_{(s,t)}\leq m_{(r,j)}$ we define the tableau $\Bs_{(r,j)}^{(s,t)}$ as follows:
  \begin{enumerate}
    \item If $a=m_{(s,t)}-1,$ then
    \begin{equation*}
      \Bs_{(r,j)}^{(s,t)}|_a=\bT_{(r,j)}|_a
    \end{equation*}
    \item The restricted function $\Bs_{(r,j)}^{(s,t)}:\{m_{(s,t)},\dots,m\}\rightarrow [\blambda_{(r,j)}]$ is strictly decreasing with respect to the order $\rhd$ on nodes in $[\blambda_{(r,j)}].$
  \end{enumerate}
  We also define the elements $H_{(r,j)}^{(s,t)}\in\mathfrak{S}_m$ and $\underline{H}_{(r,j)}^{(s,t)}\in\B$ respectively as follows:
  \begin{equation*}
    H_{(r,j)}^{(s,t)}=d(\Bs),\quad  \underline{H}_{(r,j)}^{(s,t)}=\Psi_{\Bs}^{\blambda_{(r,j)}},\quad\textrm{where}\quad \Bs=\Bs_{(r,j)}^{(s,t)}.
  \end{equation*}

  Finally, if $m_{(s,t)}>m_{(1,0)},$ we define the elements $T_{(r,j)}^{(s,t)}\in\mathfrak{S}_m$ and $\underline{T}_{(r,j)}^{(s,t)}\in\B$ respectively by:
  \begin{equation*}
     H_{(r,j)}^{(s,t)}= H_{(r,j)}^{(u,v)}T_{(r,j)}^{(s,t)},\quad \underline{H}_{(r,j)}^{(s,t)}= \underline{H}_{(r,j)}^{(u,v)}\underline{T}_{(r,j)}^{(s,t)},
  \end{equation*}

  where $m_{(u,v)}$ is the antecessor of $m_{(s,t)}$ in $M_{m_{(r,j)}}.$
  \end{definition}

    By definition of the official reduced word of $d(\bT)$ for one-column tableaux, given in section \ref{sec-general-combinatorics}, all the equations given in definition  \ref{def-H-factors} are consistent.

    By definition, it is clear that:
    \begin{equation*}
      \Bs_{(r,j)}^{(r,j)}=\bT_{(r,j)},\quad H_{(r,j)}^{(r,j)}=d(\bT_{(r,j)}),\quad \underline{H}_{(r,j)}^{(r,j)}=\Psi_{\bT_{(r,j)}}^{\blambda_{(r,j)}}.
   \end{equation*}

\begin{lemma}\label{lemma-factor-Psit-in-triangles}
 Let $1\leq r \leq k$ and   $0\leq j \leq l-2.$ Then we have

 \begin{equation*}
   d(\bT_{(r,j)})=H_{(r,j)}^{(1,0)}\left(\prod_{m_{(1,0)}<m_{(s,t)}\leq m_{(r,t)}}T_{(r,j)}^{(s,t)}\right)
 \end{equation*}
 and
 \begin{equation*}
   \Psi_{\bT_{(r,j)}}^{\blambda_{(r,j)}}=\underline{H}_{(r,j)}^{(1,0)}\left(\prod_{m_{(1,0)}<m_{(s,t)}\leq m_{(r,t)}}\underline{T}_{(r,j)}^{(s,t)}\right).
 \end{equation*}
\end{lemma}
\begin{proof}
  It follows directly from definition \ref{def-H-factors}.
\end{proof}

\begin{example}
If $\een=13,l=4$ $\kappa=(0,4,6,10)$ and $m=29$ and we take $(r,j)=(2,2)$ (see equation \ref{Ex-t-bounded-langle-rangle-4}), then with the help of the algorithm described in section \ref{sec-general-combinatorics} we can see:

\begin{equation*}
   \includegraphics[scale=0.25]{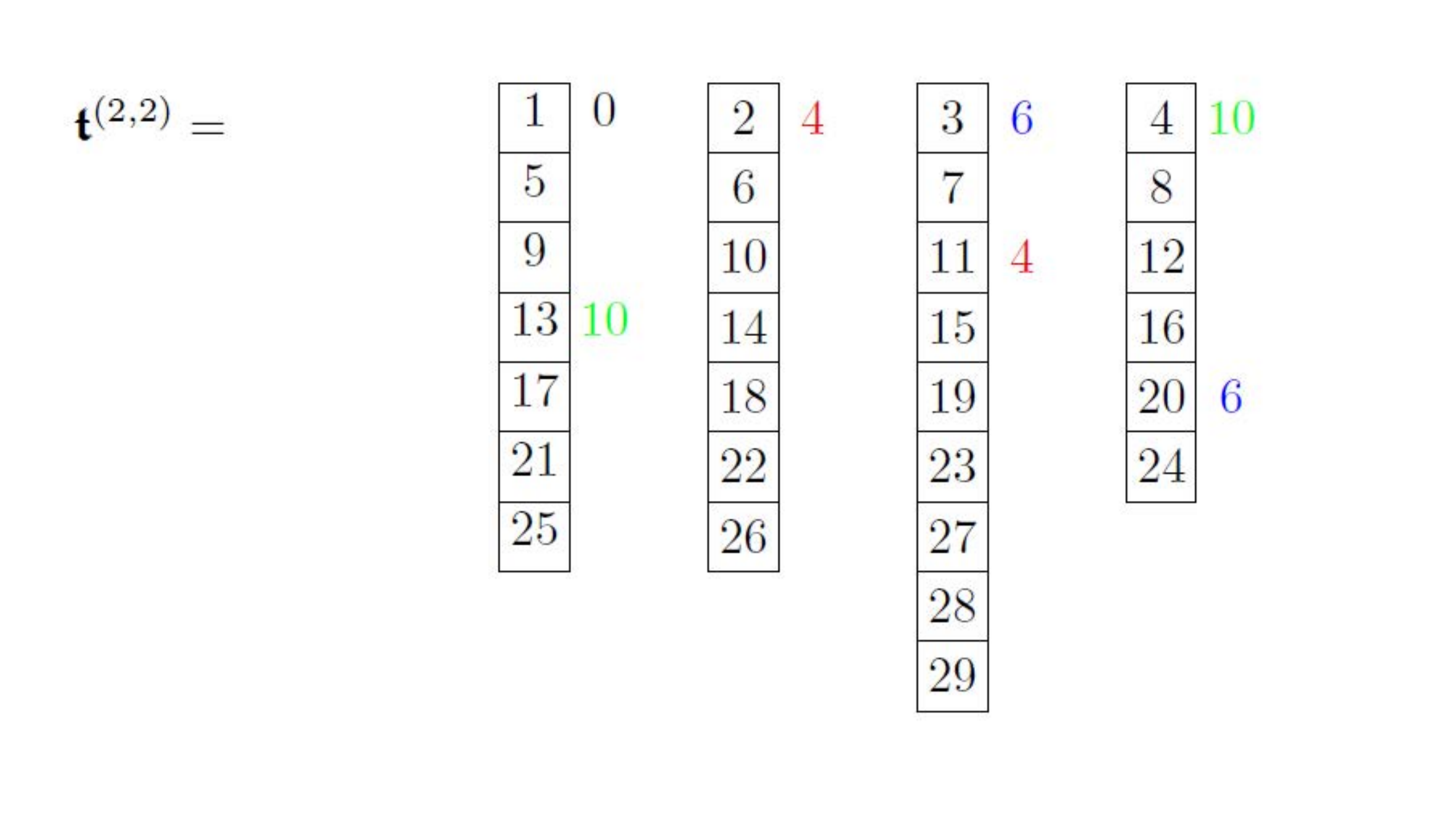}  \includegraphics[scale=0.2]{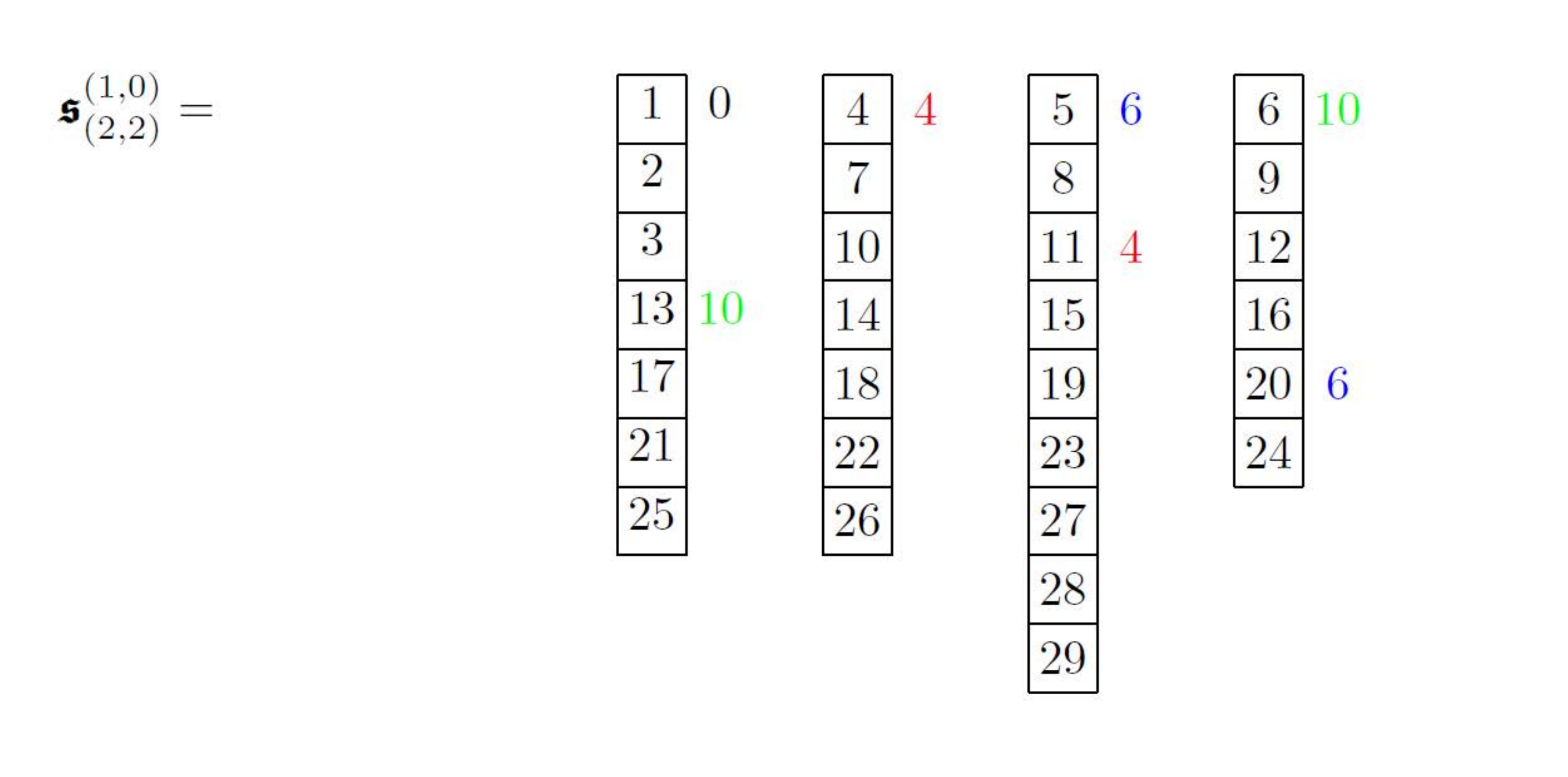}
  \end{equation*}

\begin{equation*}
   \includegraphics[scale=0.2]{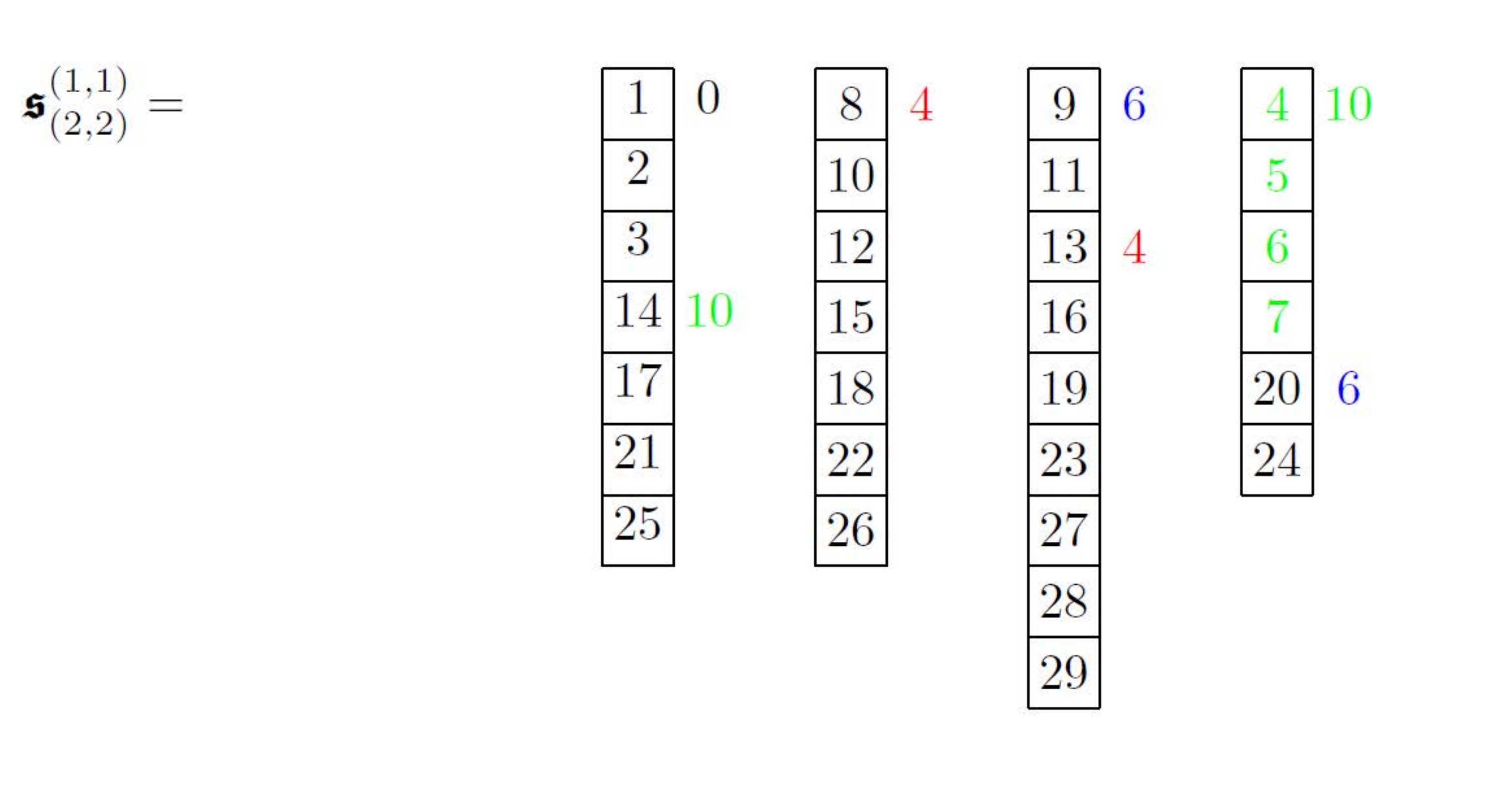} \includegraphics[scale=0.2]{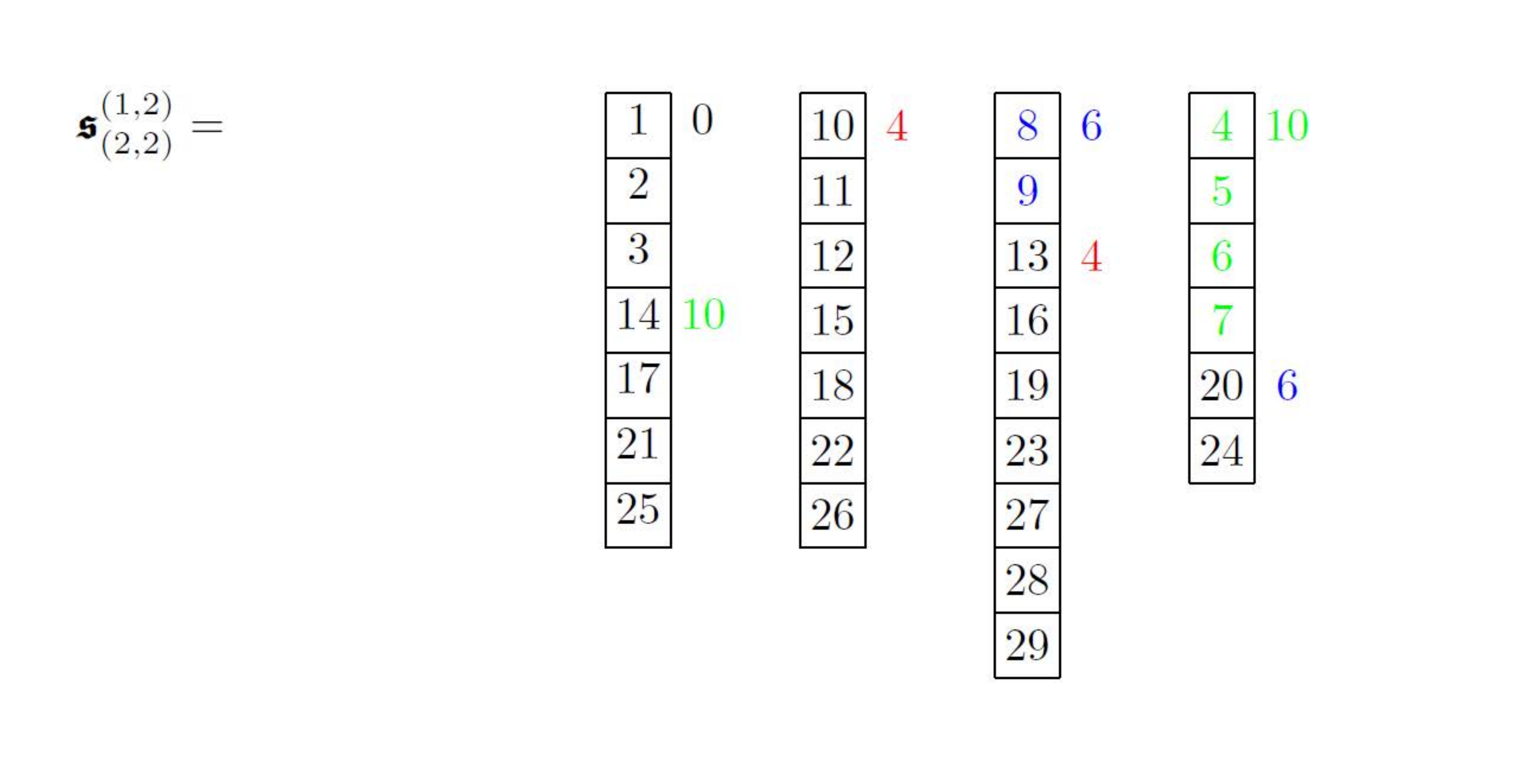}
\end{equation*}

\begin{equation*}
 \includegraphics[scale=0.25]{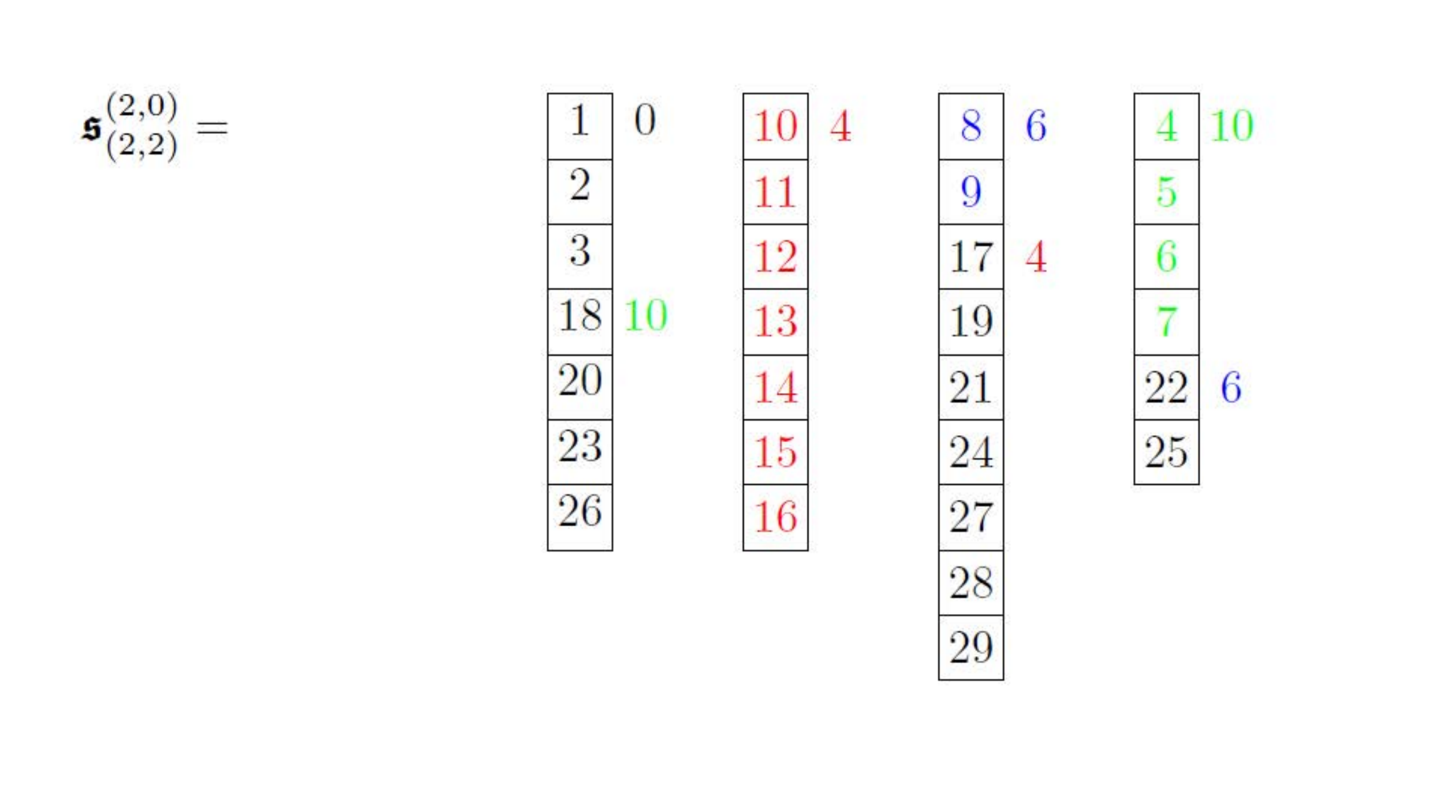} \includegraphics[scale=0.25]{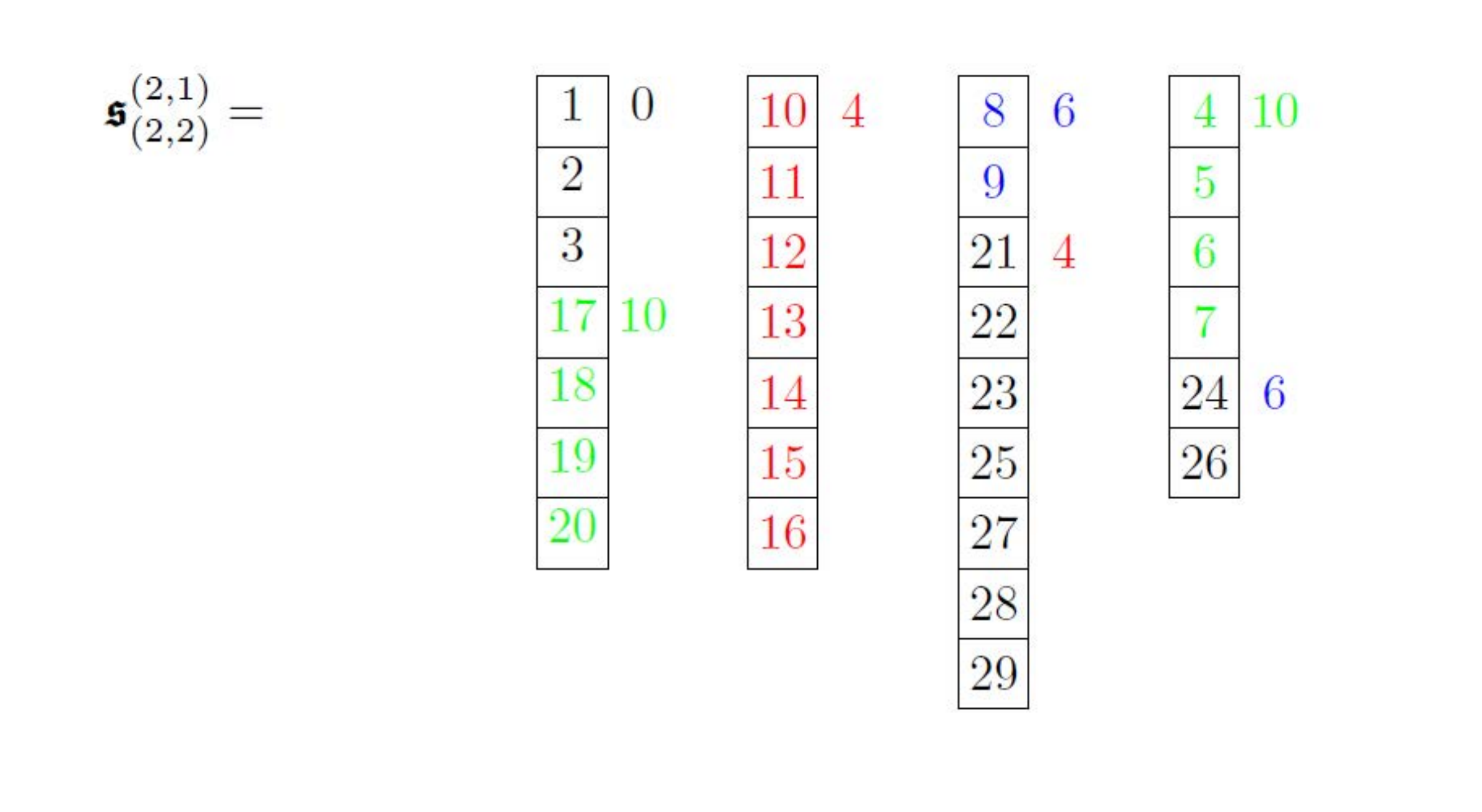}
  \end{equation*}

 \begin{equation*}
 \includegraphics[scale=0.25]{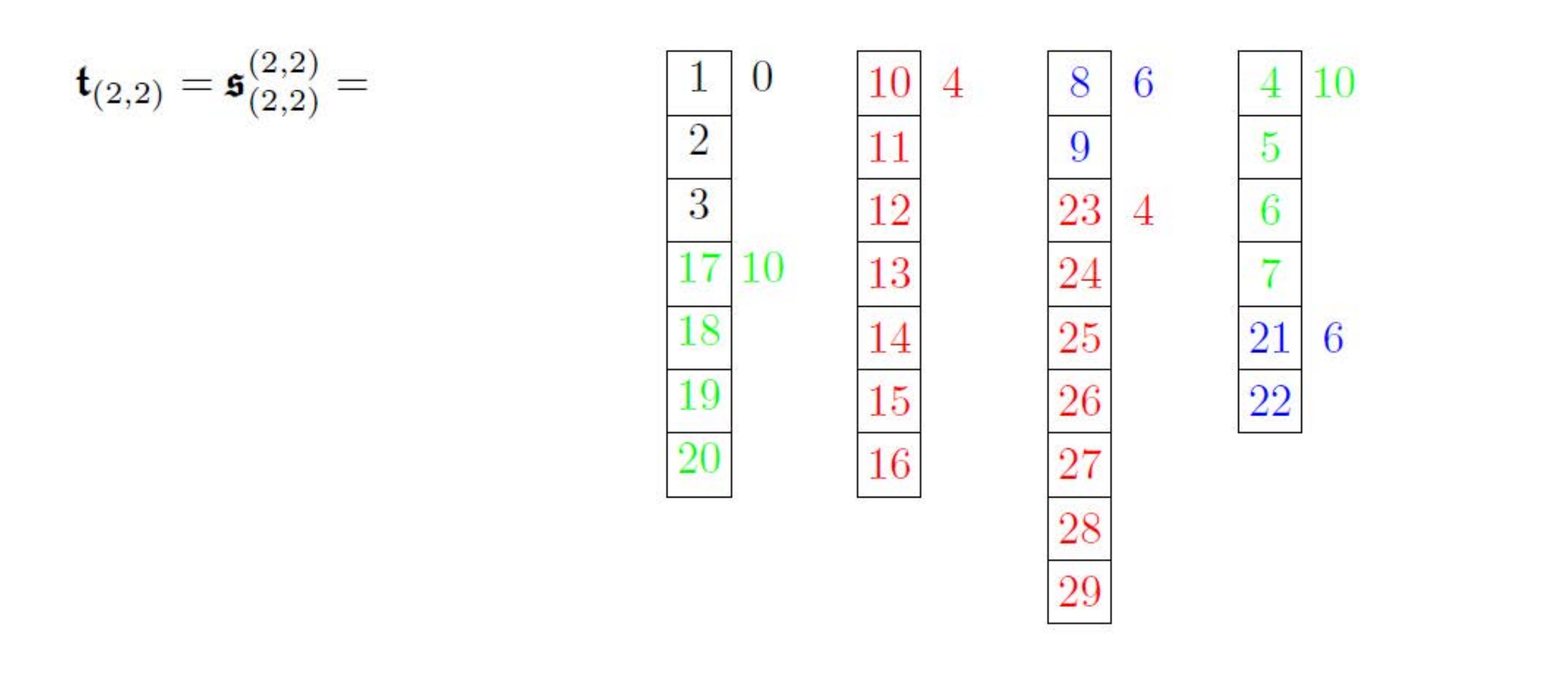}\includegraphics[scale=0.3]{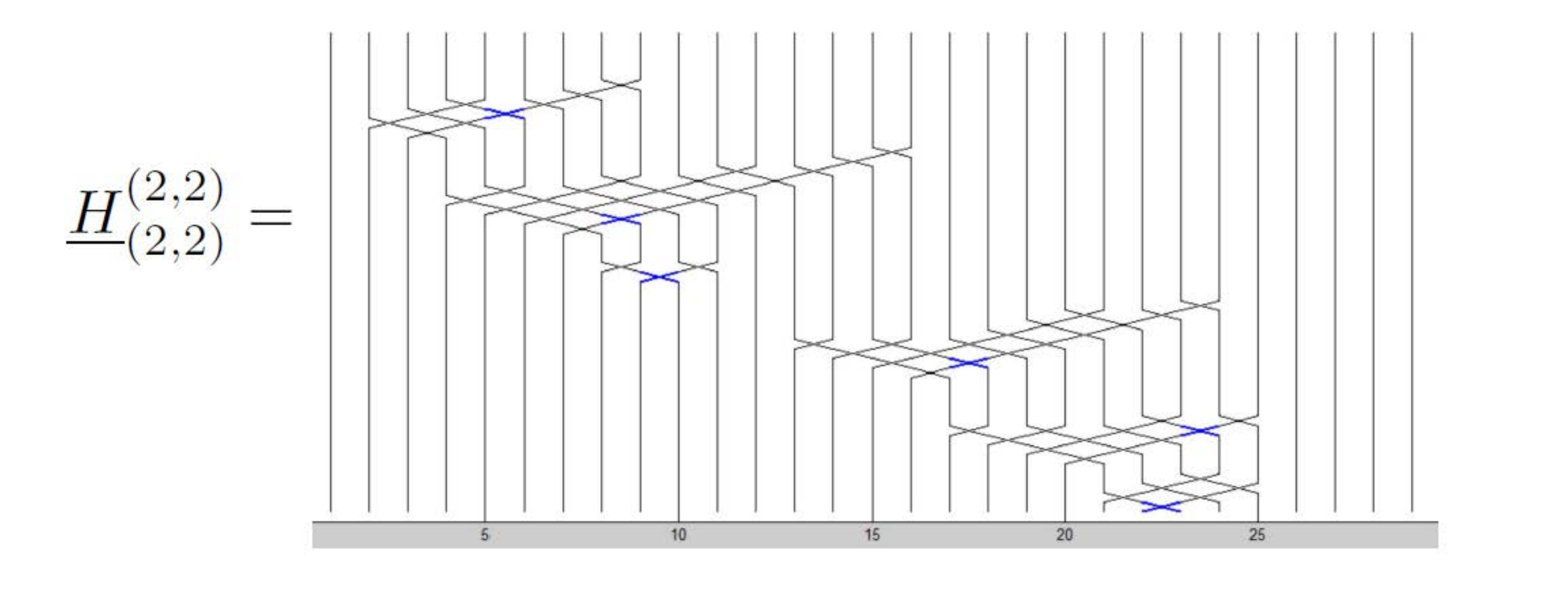}
  \end{equation*}

  We have added the color blue to highlight a cross between two cousins strings (we will use this convention for the rest of this section). One can check that
\begin{equation*}
  \begin{array}{cc}
    H_{(2,2)}^{(1,0)}=s_{[4:-1:2]}s_{[8:-1:3]}, & T_{(2,2)}^{(1,1)}=s_{5}s_{4}s_{[8:-1:5]}s_{[11:-1:6]}s_{[15:-1:7]}, \\
    \quad & \quad\\
    T_{(2,2)}^{(1,2)}=s_{8}s_{10}s_{9}, & T_{(2,2)}^{(2,0)}=s_{13}s_{[16:-1:14]}s_{[20:-1:15]}s_{[24:-1:16]},  \\
    \quad & \quad\\
    T_{(2,2)}^{(2,1)}=s_{17}s_{19}s_{18}s_{[22:-1:19]}s_{[25:-1:20]}, & T_{(2,2)}^{(2,2)}=s_{[23:-1:21]}s_{[25:-1:22]}.
  \end{array}
\end{equation*}
\end{example}

\begin{lemma}\label{lemma-no-sister-in-principal-tableaux}
  In the diagrammatic representation of $\Psi_{\bT_{(r,j)}}^{\blambda_{(r,j)}}$ there are no crosses between sisters strings.
\end{lemma}
\begin{proof}
  Following the algorithm described in section \ref{sec-general-combinatorics}, in each step of the transformation from $\bT^{(r,j)}$ to $\bT_{(r,j)},$ when we apply a simple reflection $s_a$ we always move the entry $a$ from a node to a lower node (with respect to the order $\rhd.$)

  By lemma \ref{lemma-factor-Psit-in-triangles} and by definition \ref{def-H-factors}, we can make this process of transformation progressively locating first the lower entries in their final position and then working with the greater entries.

  Let $1\leq a<b\leq m$ two entries corresponding with two sister nodes in   $\bT_{(r,j)}.$ By definition of $\bT^{(r,j)}$ and $\bT_{(r,j)}$ we have
  \begin{equation*}
   \bT^{(r,j)}(a)\rhd \bT^{(r,j)}(b) \quad \textrm{and}\quad \bT_{(r,j)}(a)\rhd \bT_{(r,j)}(b),
  \end{equation*}
  this implies that in the whole process, the entry $a$ will never be in a lower node than the entry $b.$

  Finally this implies that in the diagrammatic representation of $\Psi_{\bT_{(r,j)}}^{\blambda_{(r,j)}}$ the strings corresponding with $a$ and $b$ will not cross one each other.

\end{proof}

\begin{lemma}\label{lemma-HvsH-reduction}
  Let $1\leq r \leq k,$ $0\leq j\leq l-2.$
  Then
  \begin{equation*}
    \Psi_{\bT_{(r,j)},\bT_{(r,j)}}^{\blambda_{(r,j)}}=e(\bif)\left(\prod_{m_{(s,t)}\leq m_{(r,j)}}{L_{m_{(s,t)}}}\right).
  \end{equation*}
\end{lemma}

\begin{proof}
 Let $\bi=\bi^{\bT^{(r,j)}}.$

 We first prove the following assertion:

 If $m_{(s,t)},m_{(u,v)}\in M_{m_{(r,j)}}$ are such that $m_{(1,0)}\leq m_{(u,v)}< m_{(s,t)}$ and $m_{(u,v)}$ is the antecessor of $m_{(s,t)}$ in $M_{m_{(r,j)}}.$
  Then
  \begin{equation*}
    \left(\underline{H}_{(r,j)}^{(s,t)}\right)^{\ast}\underline{H}_{(r,j)}^{(s,t)}
    =\left(\underline{T}_{(r,j)}^{(s,t)}\right)^{\ast}e(\bj)\underline{T}_{(r,j)}^{(s,t)}\left(\prod_{m_{(w,x)}\leq m_{(u,v)}} L_{m_{(w,x)}}\right).
  \end{equation*}
  where $\bj=H_{(r,j)}^{(u,v)}\cdot \bi.$

 We proceed by induction on the numbers $m_{(s,t)}\in M_{m_{(r,j)}}:$

  Let $(s,t)=(1,1).$ Let $$\Bs=\bT^{(r,j)}H_{(r,j)}^{(1,1)}=\bT^{(r,j)}H_{(r,j)}^{(1,0)}T_{(r,j)}^{(1,1)}.$$

  By definition of $H_{(r,j)}^{(1,0)},$ in some point of the process of transformation from $\bT^{(r,j)}$ to $\bT^{(r,j)}H_{(r,j)}^{(1,0)},$ it will be necessary to apply once a simple reflection between the entries of the cousin nodes $\beta=(\epsilon,1,1)$ and $\gamma=(1,1,l)$ (note that $\res(\beta)=\res(\gamma)-1=\kappa_l-1$). The rest of the permutations involved in this process will be between not related nodes (lemma \ref{lemma-no-sister-in-principal-tableaux}). This implies in the context of the generalized blob algebra that, by relation \ref{lazo}, we will have:
  \begin{equation*}
    \left(\underline{H}_{(r,j)}^{(1,0)}\right)^{\ast}\underline{H}_{(r,j)}^{(1,0)}=e(\bj)L_a
  \end{equation*}
  for an adequate $a\in\{1,\dots,m\}$ and where $\bj=H_{(r,j)}^{(1,0)}\cdot\bi.$

  Therefore, we have

  \begin{equation*}
    \left(\underline{H}_{(r,j)}^{(1,1)}\right)^{\ast}\underline{H}_{(r,j)}^{(1,1)}
    =\left(\underline{T}_{(r,j)}^{(1,1)}\right)^{\ast}e(\bj)L_a\underline{T}_{(r,j)}^{(1,1)}
  \end{equation*}

  By definition of $H_{(r,j)}^{(1,1)},$ we have that $\Bs^{-1}(\gamma)=m_{(1,0)}$ and $\Bs^{-1}(\beta)=m_{(1,0)}-1.$ This implies that
  \begin{equation*}
    T_{(r,j)}^{(1,1)}(a-1)=m_{(1,0)}-1,\quad \textrm{and}\quad T_{(r,j)}^{(1,1)}(a)=m_{(1,0)}.
  \end{equation*}

  Therefore, by relations \ref{punto-arriba},\ref{punto-abajo} and \ref{cruce-pasa}, we have:

  \begin{equation*}
    \left(\underline{H}_{(r,j)}^{(1,1)}\right)^{\ast}\underline{H}_{(r,j)}^{(1,1)}
    =\left(\underline{T}_{(r,j)}^{(1,0)}\right)^{\ast}e(\bj)\underline{T}_{(r,j)}^{(1,0)}L_{m_{(1,0)}}.
  \end{equation*}

  This is the base of our induction.

  Let's assume now that $m_{(s,t)}>m_{(1,1)}.$

    Let $$\Bs=\bT^{(r,j)}H_{(r,j)}^{(s,t)}=\bT^{(r,j)}H_{(r,j)}^{(u,v)}T_{(r,j)}^{(s,t)}$$
  By definition \ref{def-H-factors} and inductive hypothesis we have:
  \begin{equation*}
     \left(\underline{H}_{(r,j)}^{(s,t)}\right)^{\ast}\underline{H}_{(r,j)}^{(s,t)}
    =\left(\underline{T}_{(r,j)}^{(s,t)}\right)^{\ast}\left(\underline{T}_{(r,j)}^{(u,v)}\right)^{\ast}
    e(\boldsymbol{h})\underline{T}_{(r,j)}^{(u,v)}\left(\prod_{m_{(w,x)}<m_{(u,v)}}{L_{m_{(w,x)}}}\right)\underline{T}_{(r,j)}^{(s,t)}.
  \end{equation*}
  for an appropriate residue sequence $\boldsymbol{h}.$

  Let $m_{(u',v')}$ is the antecessor of $m_{(u,v)}$ in $M_{m_{(r,j)}}.$

  By definition \ref{def-H-factors} we can see that $T_{(r,j)}^{(s,t)}(a)=a$ for each $a\leq m_{(u',v')}.$  Therefore we have

  \begin{equation*}
     \left(\underline{H}_{(r,j)}^{(s,t)}\right)^{\ast}\underline{H}_{(r,j)}^{(s,t)}
    =\left(\underline{T}_{(r,j)}^{(s,t)}\right)^{\ast}\left(\underline{T}_{(r,j)}^{(u,v)}\right)^{\ast}
    e(\boldsymbol{h})\underline{T}_{(r,j)}^{(u,v)}\underline{T}_{(r,j)}^{(s,t)}\left(\prod_{m_{(w,x)}<m_{(u,v)}}{L_{m_{(w,x)}}}\right).
  \end{equation*}

  Lets we denote by $B^{(v)}$ and $B^{(v')}$ the blocks of nodes in $[\blambda_{(r,j)}]$ corresponding with $\bT_{(r,j)}(B_{(u,v)})$ and $\bT_{(r,j)}(B_{(u',v')})$ respectively. Let $\beta$ the lowest node in $B^{(v')}$ and $\gamma$ the highest node in $B^{(v)}.$ Note that those nodes are cousins.

  By definition \ref{def-H-factors} one can see that in some point of the process of transforming $\bT^{(r,j)}H_{(r,j)}^{(u',v')}$ into $\bT^{(r,j)}H_{(r,j)}^{(u,v)}$ we will be forced to permute once the entries of $\beta$ and $\gamma.$   The rest of the permutations involved in this process will be between not related nodes (lemma \ref{lemma-no-sister-in-principal-tableaux}). This implies that

  \begin{equation*}
     \left(\underline{T}_{(r,j)}^{(u,v)}\right)^{\ast}
    e(\boldsymbol{h})\underline{T}_{(r,j)}^{(u,v)}=e(\boldsymbol{k})L_{a}.
  \end{equation*}
  where $\boldsymbol{k}=T_{(r,j)}^{(u,v)}\cdot\boldsymbol{h}$ and $a$ is certain number greater that $m_{(u',v')}.$

  Now by definition of $H_{(r,j)}^{(s,t)}$ we have that $\Bs^{-1}(\beta)=m_{(u,v)}-1$ and $\Bs^{-1}(\gamma)=m_{(u,v)}.$ This implies that

  \begin{equation*}
     \left(\underline{T}_{(r,j)}^{(s,t)}\right)^{\ast}e(\boldsymbol{k})L_{a}\underline{T}_{(r,j)}^{(s,t)}
     =\left(\underline{T}_{(r,j)}^{(s,t)}\right)^{\ast}e(\boldsymbol{k})\underline{T}_{(r,j)}^{(s,t)}L_{m_{(u,v)}}.
  \end{equation*}
  and therefore
  \begin{equation*}
     \left(\underline{H}_{(r,j)}^{(s,t)}\right)^{\ast}\underline{H}_{(r,j)}^{(s,t)}
    =\left(\underline{T}_{(r,j)}^{(s,t)}\right)^{\ast}
    e(\boldsymbol{k})\underline{T}_{(r,j)}^{(s,t)}\left(\prod_{m_{(w,x)}\leq m_{(u,v)}}{L_{(w,x)}}\right).
  \end{equation*}

Particularly we have:
  \begin{equation*}
    \Psi_{\bT_{(r,j)},\bT_{(r,j)}}^{\blambda_{(r,j)}}=
     \left(\underline{H}_{(r,j)}^{(r,j)}\right)^{\ast}\underline{H}_{(r,j)}^{(r,j)}
    =\left(\underline{T}_{(r,j)}^{(r,j)}\right)^{\ast}
    e(\boldsymbol{h})\underline{T}_{(r,j)}^{(r,j)}\left(\prod_{m_{(s,t)}< m_{(r,j)}}{L_{(s,t)}}\right).
  \end{equation*}

and under the same arguments as above, we can see that:

\begin{equation*}
  \left(\underline{T}_{(r,j)}^{(r,j)}\right)^{\ast}
    e(\boldsymbol{h})\underline{T}_{(r,j)}^{(r,j)}\left(\prod_{m_{(w,x)}< m_{(r,j)}}{L_{(w,x)}}\right)=
    e(\bif)\left(\prod_{m_{(s,t)}\leq m_{(r,j)}}{L_{m_{(s,t)}}}\right).
\end{equation*}
as desired.
\end{proof}

\begin{example}

If $\een=13,l=4$ $\kappa=(0,4,6,10)$ and $m=29$ and we take $(r,j)=(2,2),$ then
  \begin{equation*}\label{im-ex-HvsH-factors1}
    \Psi_{\bT_{(2,2)},\bT_{(2,2)}}^{\blambda_{(2,2)}}
    =\raisebox{-.6\height}{\includegraphics[scale=0.17]{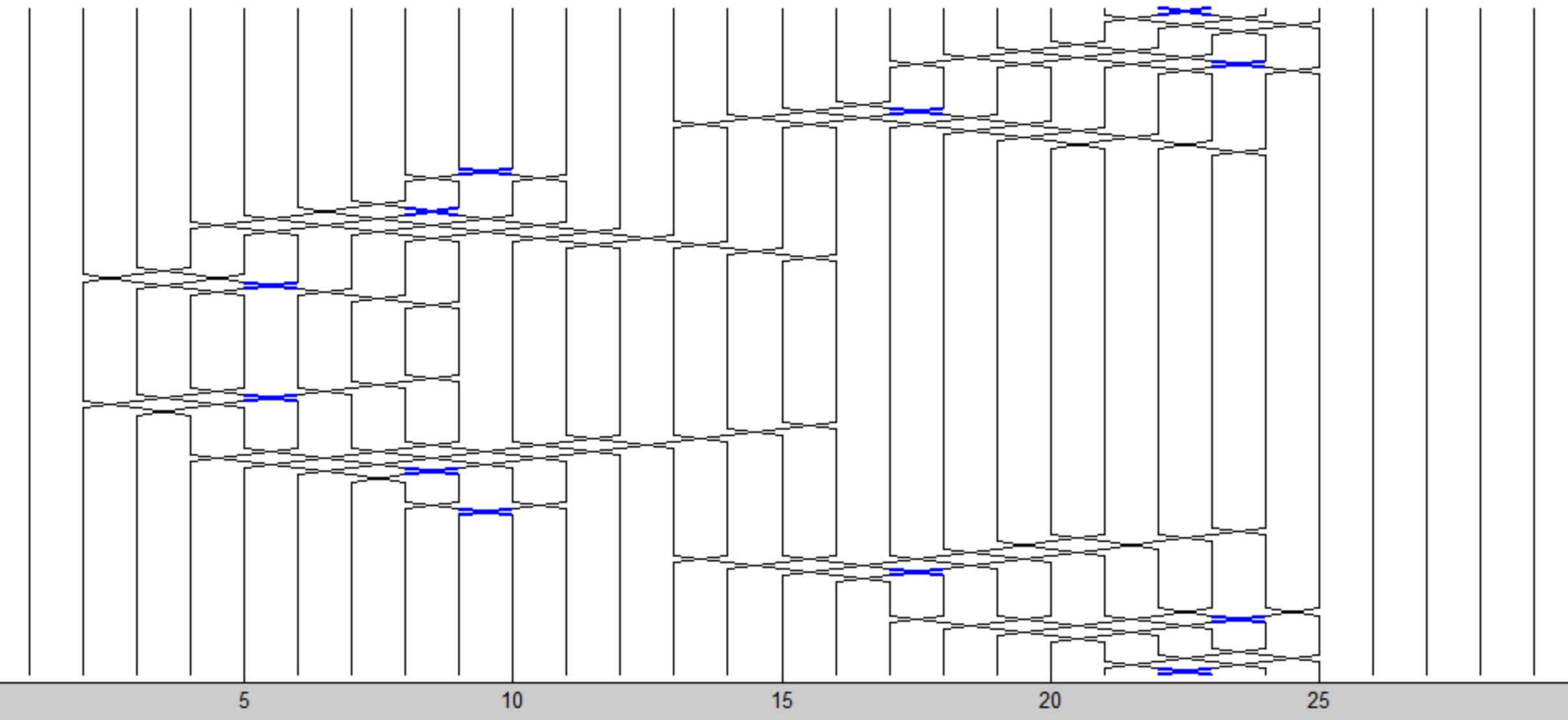}}
  \end{equation*}

If we follow the proof of lemma \ref{lemma-HvsH-reduction}, applying relations \ref{punto-arriba},\ref{punto-abajo}, \ref{cruce-pasa} and \ref{lazo}, we can see that

  \begin{equation*}
    \left(\underline{H}_{(2,2)}^{(1,0)}\right)^{\ast}\underline{H}_{(2,2)}^{(1,0)}
    =\raisebox{-.6\height}{\includegraphics[scale=0.08]{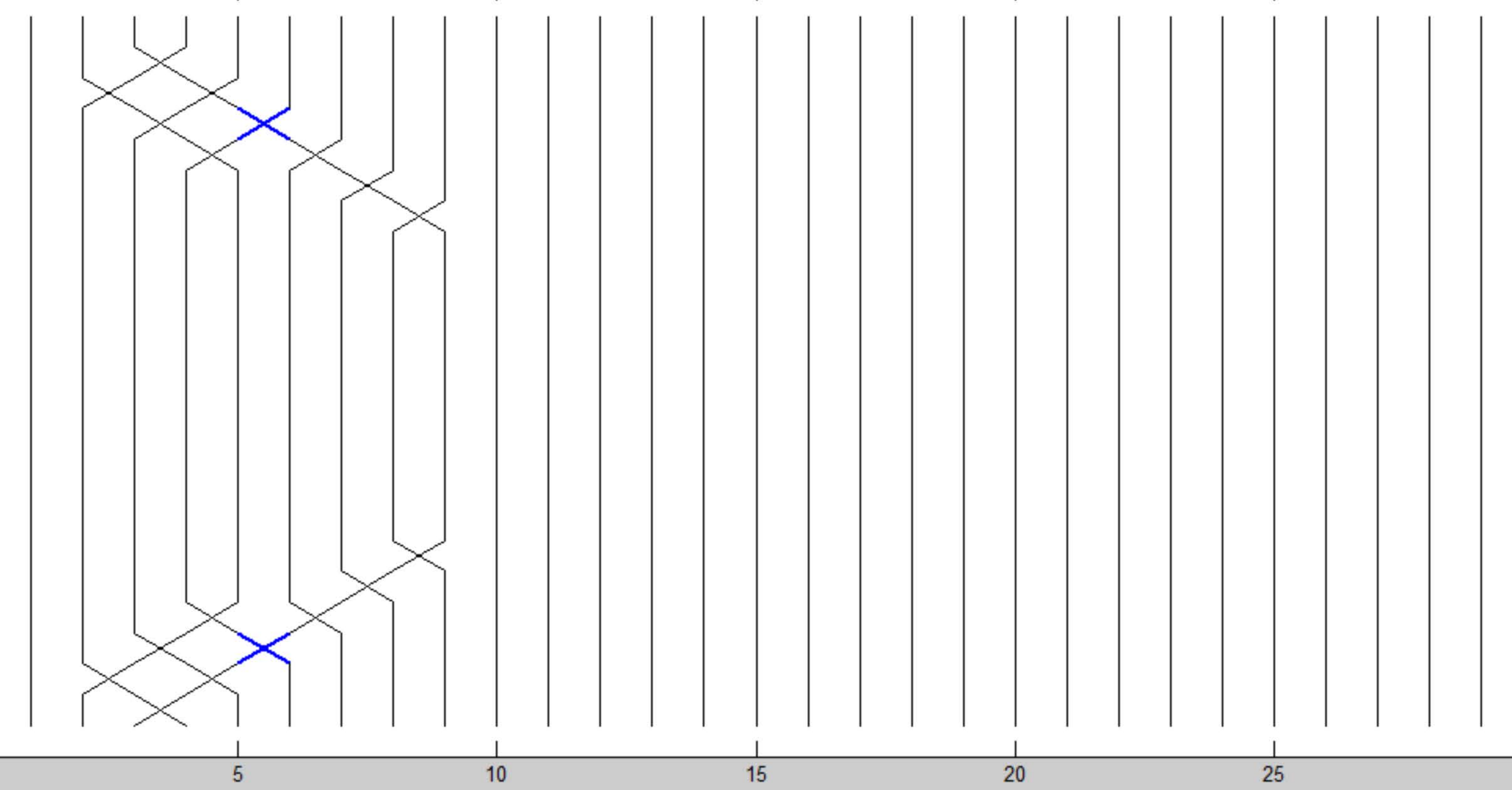}}=
    \raisebox{-.6\height}{\includegraphics[scale=0.08]{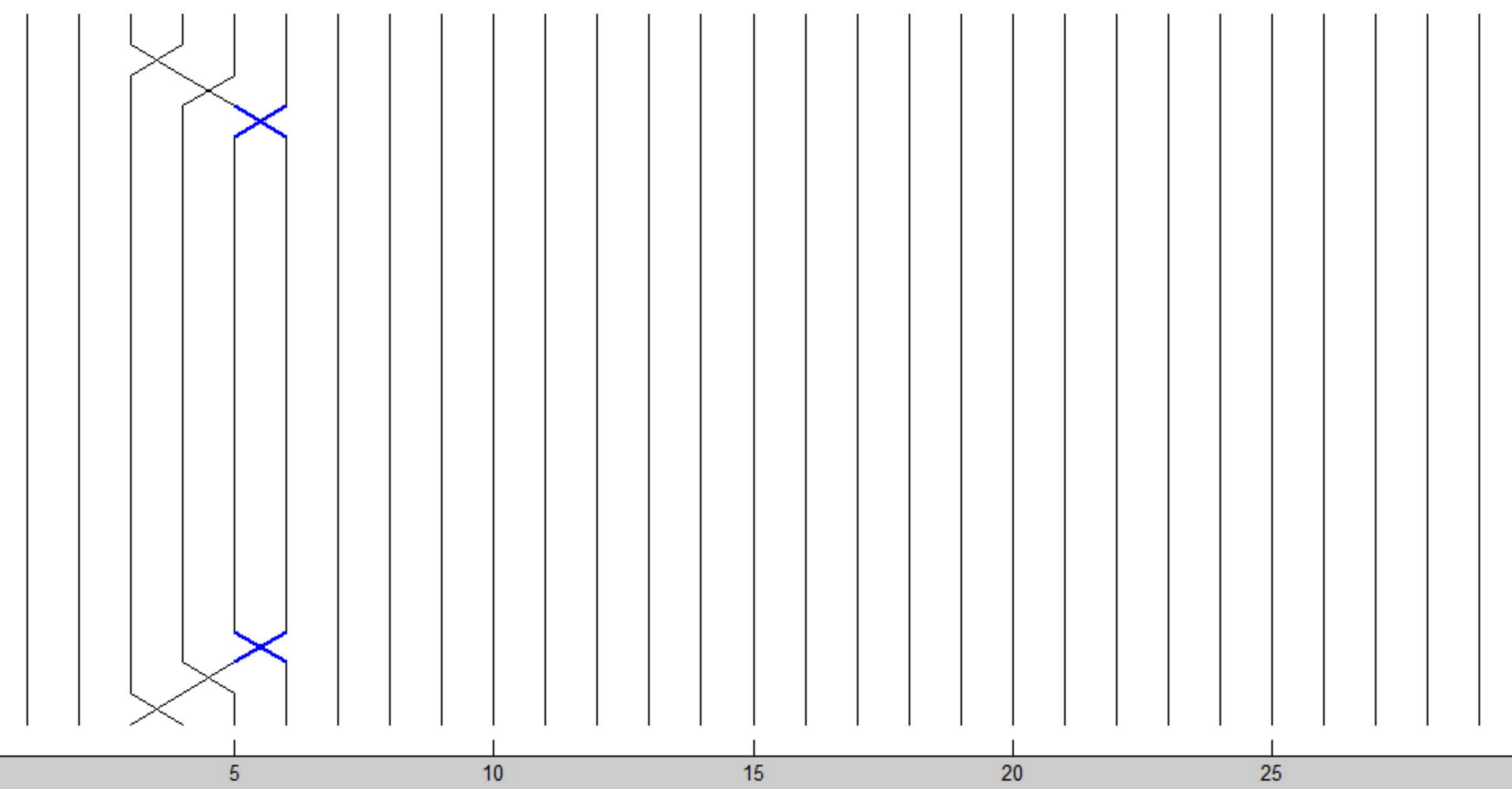}}
  \end{equation*}

The next step is:
  \begin{equation*}
    \left(\underline{H}_{(2,2)}^{(1,1)}\right)^{\ast}\underline{H}_{(2,2)}^{(1,1)}
    =\raisebox{-.6\height}{\includegraphics[scale=0.08]{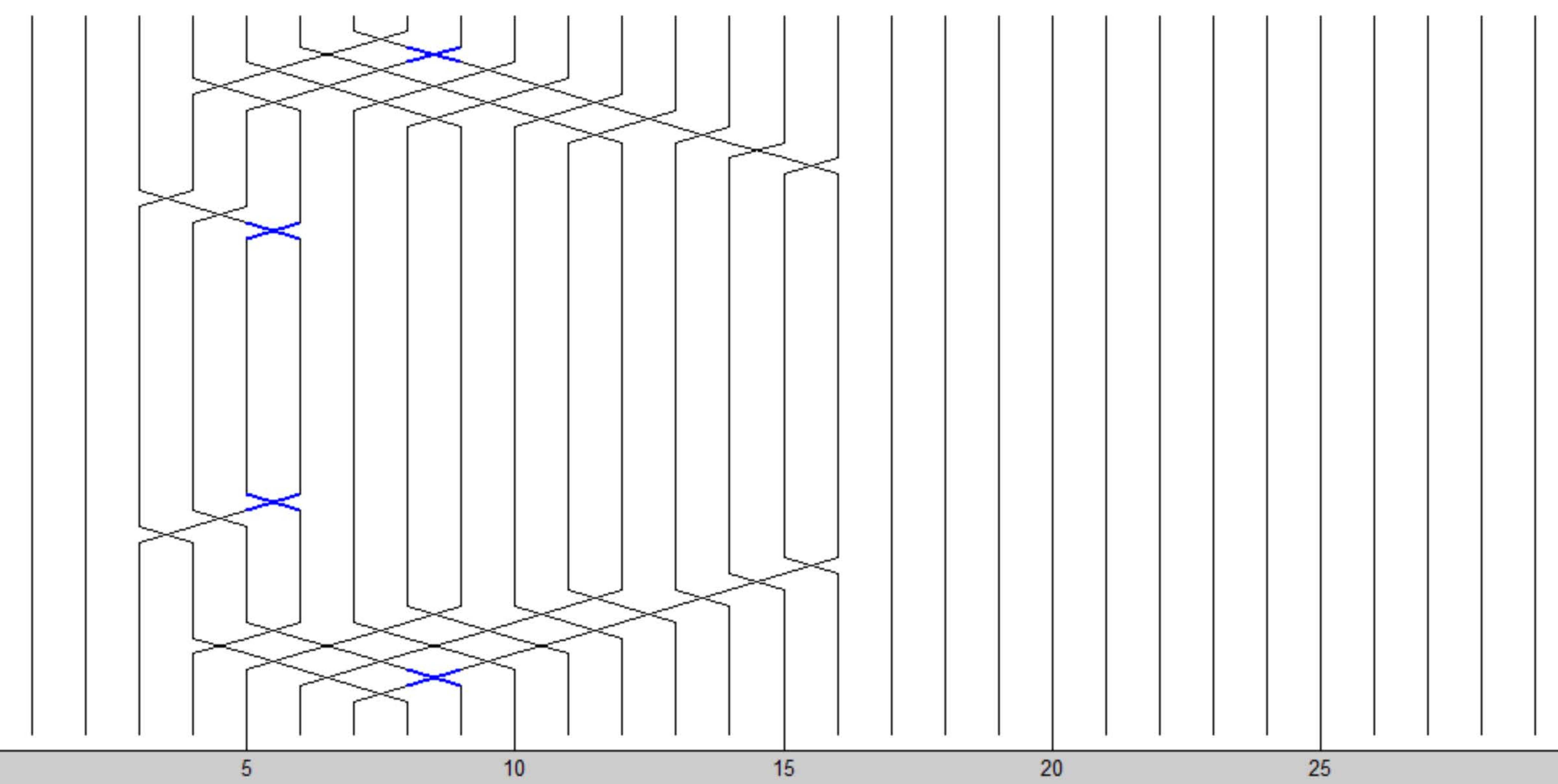}}
    =\raisebox{-.6\height}{\includegraphics[scale=0.08]{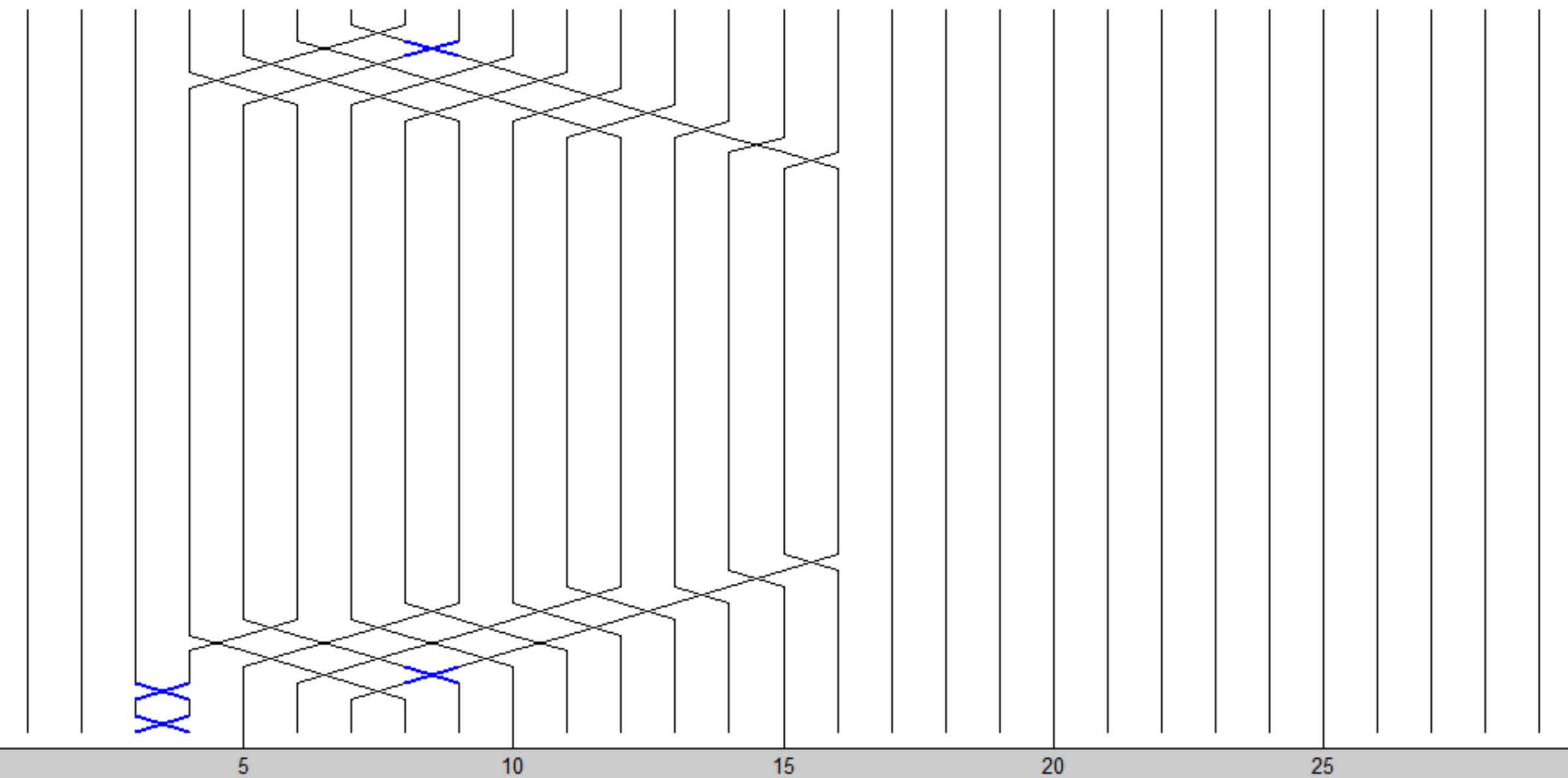}}
  \end{equation*}
then
\begin{equation*}
  \left(\underline{H}_{(2,2)}^{(1,1)}\right)^{\ast}\underline{H}_{(2,2)}^{(1,1)}
  =\left(\underline{T}_{(2,2)}^{(1,1)}\right)^{\ast}e(\bj)\underline{T}_{(2,2)}^{(1,1)}L_{m_{(1,0)}}.
\end{equation*}
following in this way, we can see
  \begin{equation*}
    \left(\underline{H}_{(2,2)}^{(1,2)}\right)^{\ast}\underline{H}_{(2,2)}^{(1,2)}
    =\raisebox{-.6\height}{\includegraphics[scale=0.08]{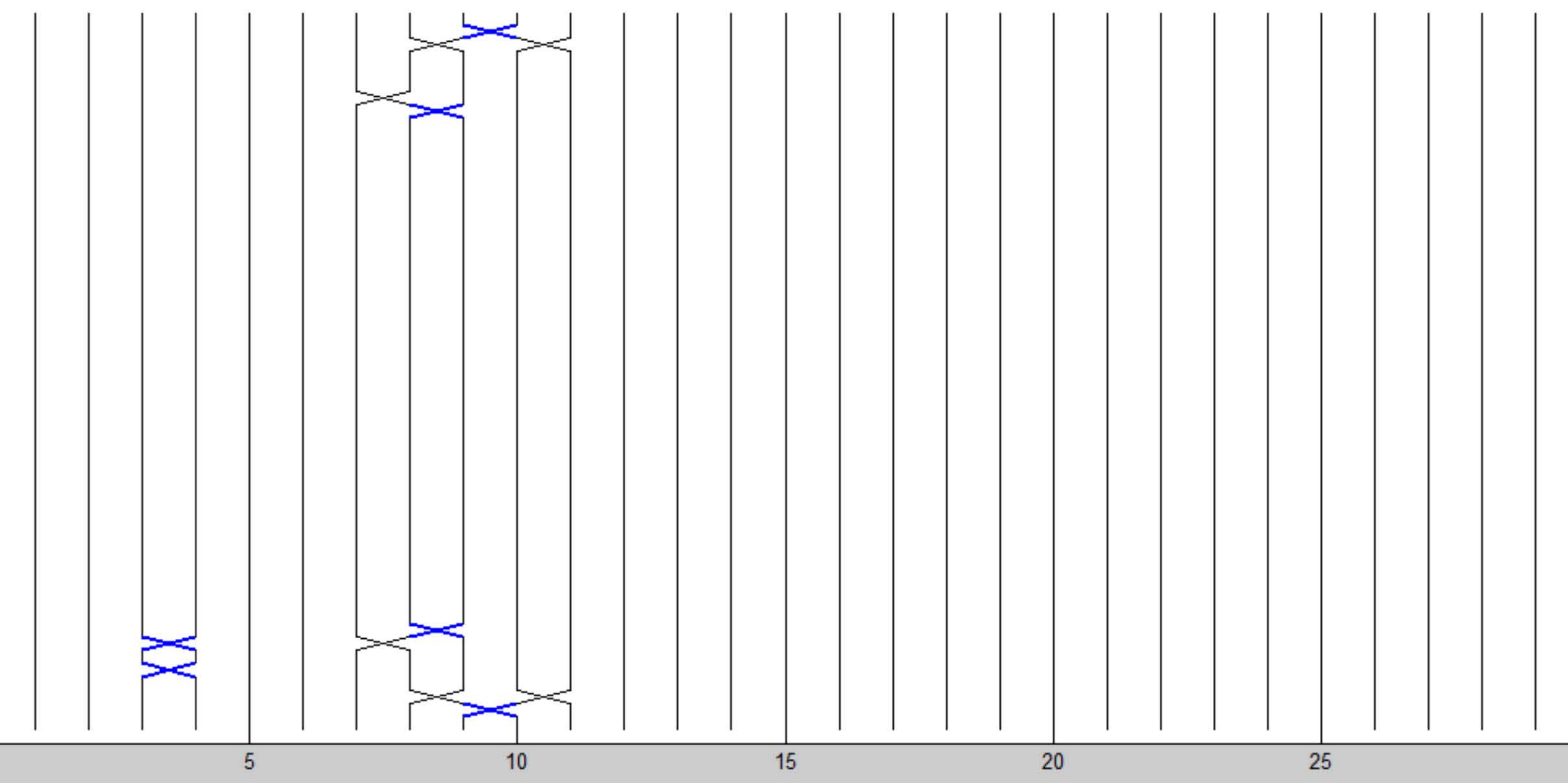}}=
    \raisebox{-.6\height}{\includegraphics[scale=0.08]{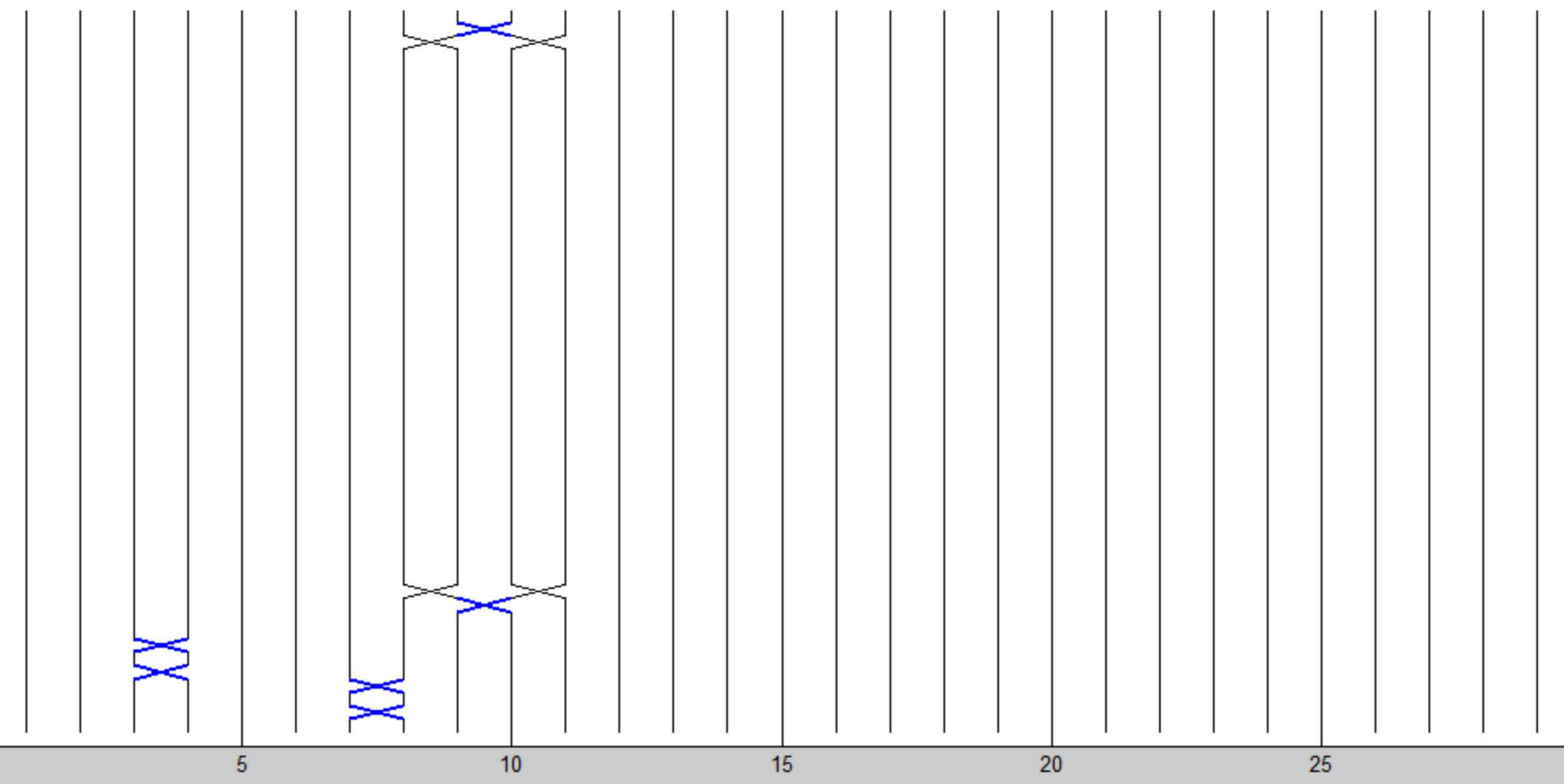}}
  \end{equation*}

  \begin{equation*}
    \left(\underline{H}_{(2,2)}^{(1,2)}\right)^{\ast}\underline{H}_{(2,2)}^{(1,2)}
    =\left(\underline{T}_{(2,2)}^{(1,2)}\right)^{\ast}e(\boldsymbol{h})\underline{T}_{(2,2)}^{(1,2)}L_{m_{(1,0)}}L_{m_{(1,1)}}.
  \end{equation*}
 and so on
  \begin{equation*}
    \raisebox{-.6\height}{\includegraphics[scale=0.09]{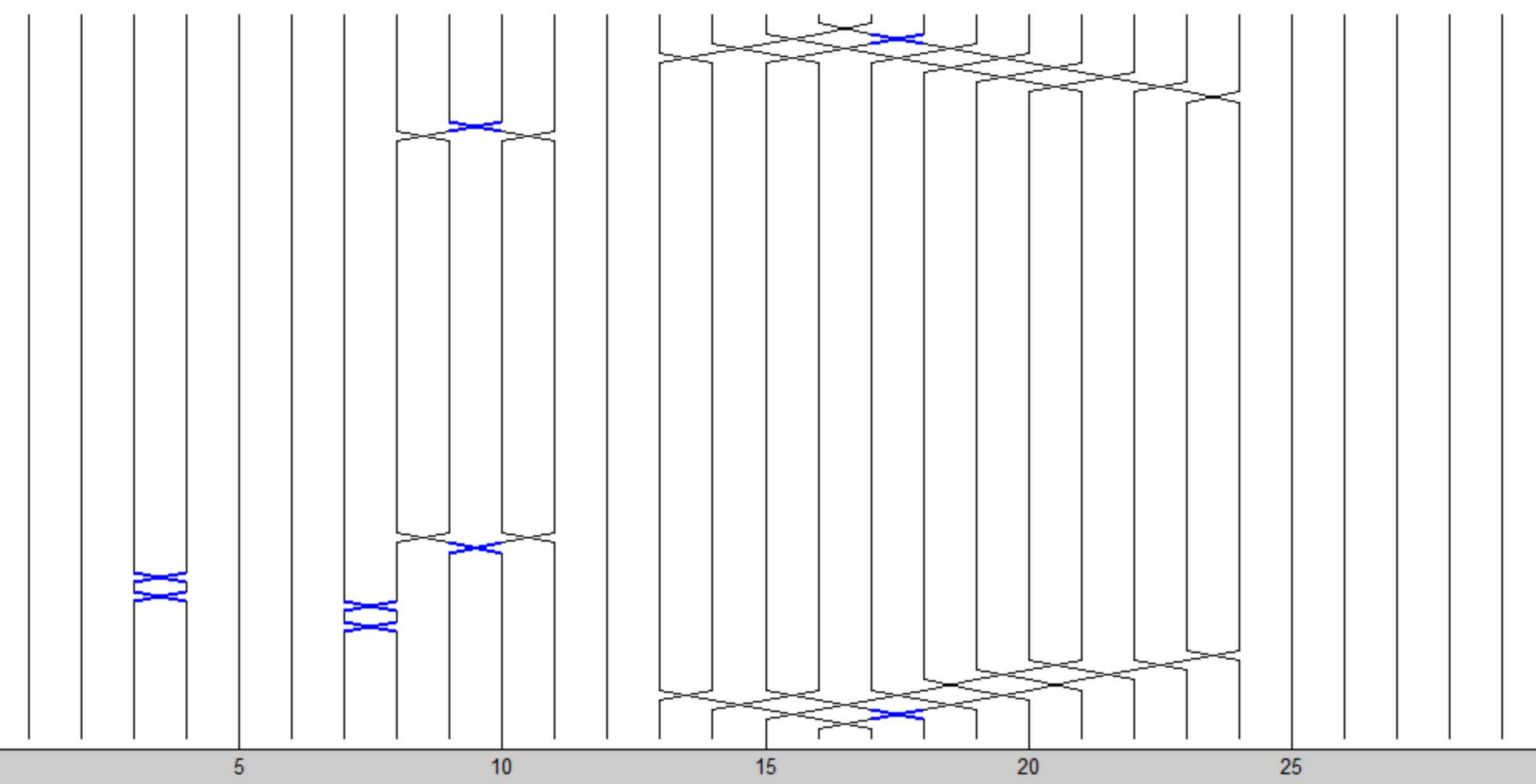}} =
    \raisebox{-.6\height}{\includegraphics[scale=0.09]{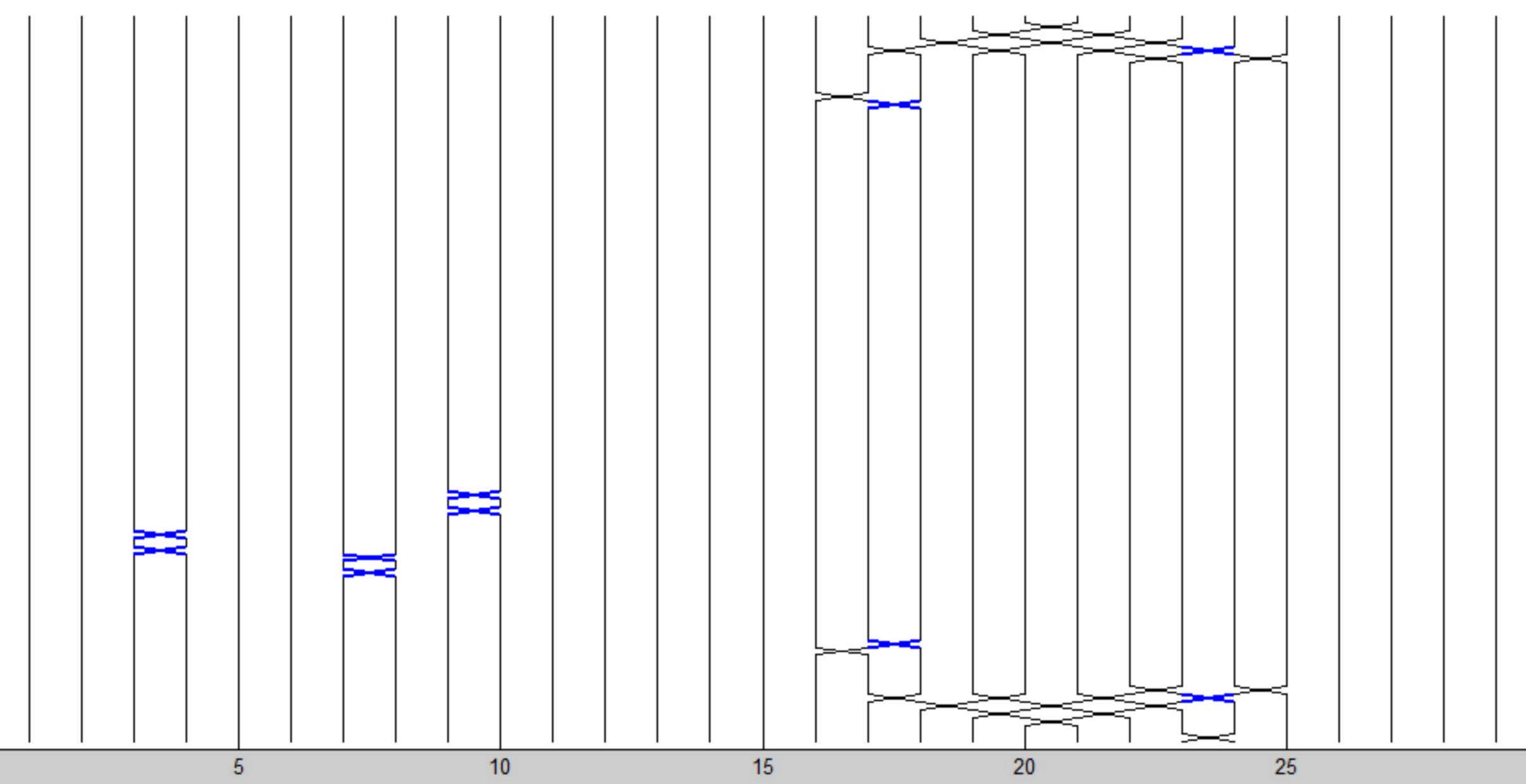}}
  \end{equation*}

  \begin{equation*}
    \raisebox{-.6\height}{\includegraphics[scale=0.09]{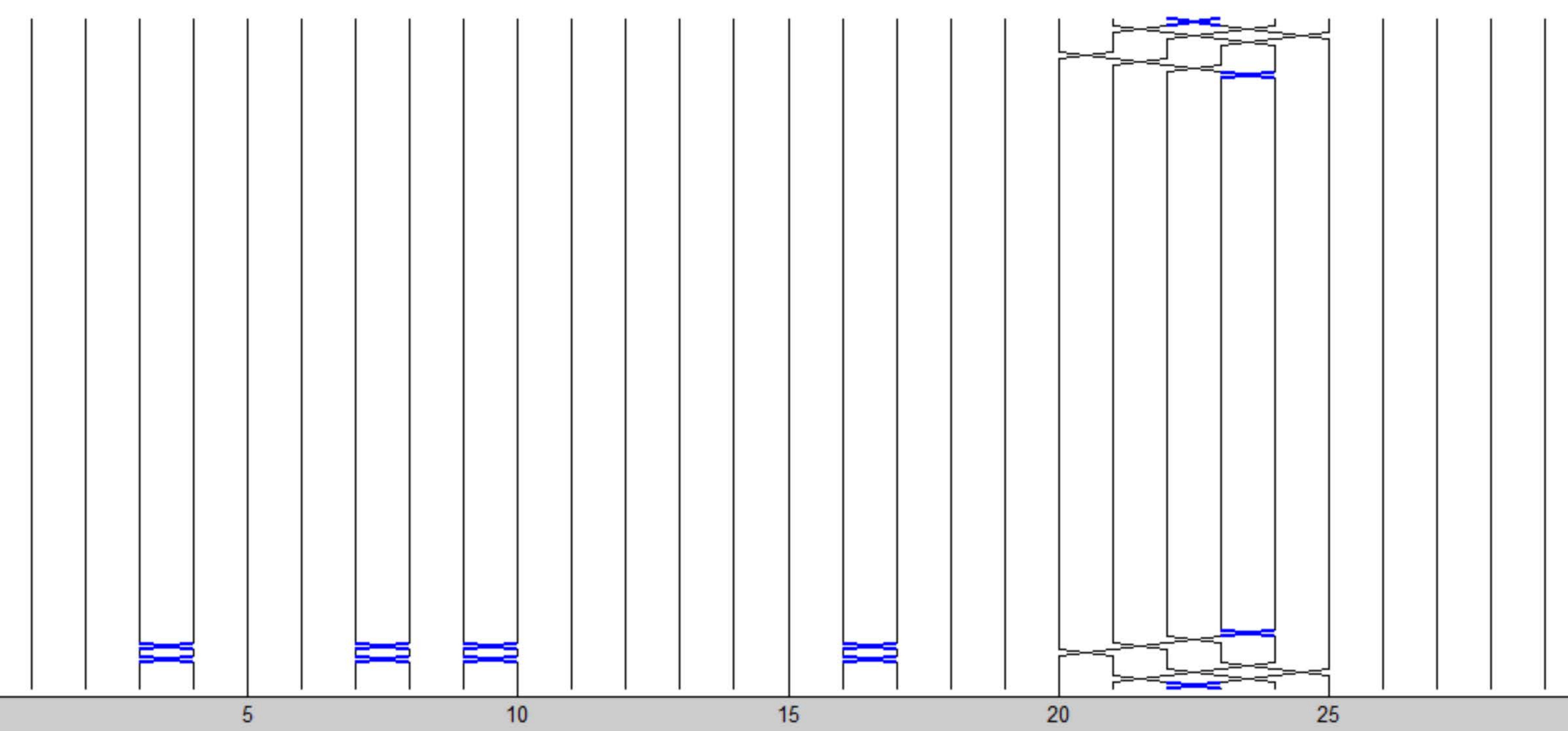}}= \raisebox{-.6\height}{\includegraphics[scale=0.09]{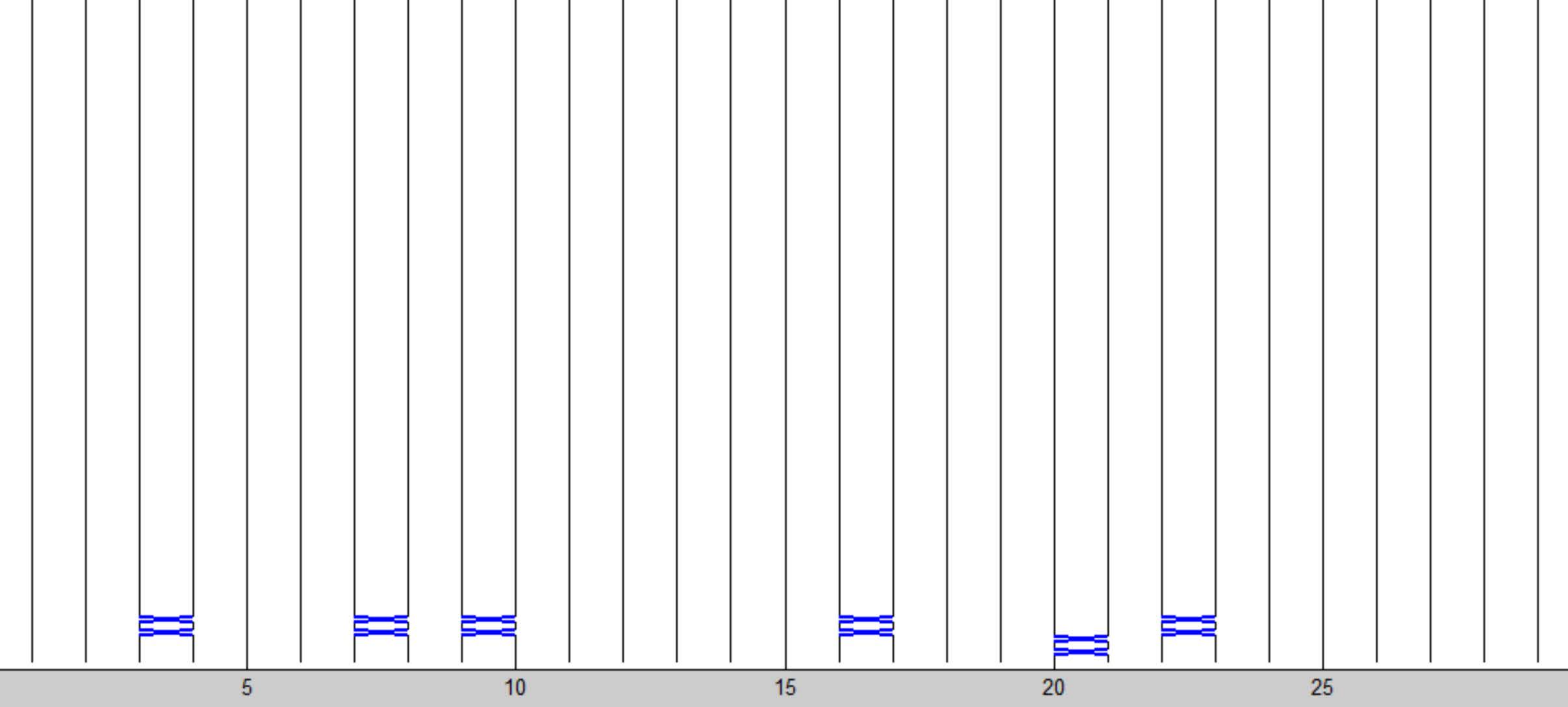}}
  \end{equation*}

  Then
  \begin{equation*}
    \Psi_{\bT_{(2,2)},\bT_{(2,2)}}^{\blambda_{(2,2)}}=e(\bif)L_{m_{(1,0)}}L_{m_{(1,1)}}L_{m_{(1,2)}}L_{m_{(2,0)}}L_{m_{(2,1)}}L_{m_{(2,2)}}.
  \end{equation*}
\end{example}


\section{Gelfand-Tsetlin subalgebra}\label{ch-gelfand}

Lets we denote by $\calD$ the subalgebra of $\B(\bif),$ generated by the set
$$ \{y_je(\bif): j=1,\dots,m\}.$$


By definition of $\B,$ we know that the algebra $\calD$ is a commutative subalgebra of $\B(\bif).$

From \cite{Lobos-Ryom-Hansen} we know that the cellular basis of theorem \ref{theo-cellular-basis}  admits a family of Jucys-Murphy elements in the sense of Mathas (see \cite{Mat-So}). Particularly if we project this family of elements into $\B(\bif)$ we obtain a family of Jucys-Murphy elements $\{\mathbb{J}_j:1\leq j\leq m\},$ for the algebra  $\B(\bif)$ with respect to the cellular basis of corollary \ref{coro-cellular-basis-truncated-2}. The elements $\mathbb{J}_j$ have the form:

\begin{equation}\label{JM-elements-for-bif}
  \mathbb{J}_j=(c(j)-y_j)e(\bif)
\end{equation}
for certain scalars $c(j)\in\F$ (see \cite{Lobos-Ryom-Hansen}).

Is not difficult to see that the algebra $\calD$ is also generated by the set $$ \{\mathbb{J}_j: j=1,\dots,m\}.$$

Following the terminology of Okounkov and Vershik (see \cite{OkunVershik1}, \cite{OkunVershik2}),  we call the algebra $\calD,$ the \emph{Gelfand-Tsetlin} subalgebra of the fundamental vertical truncation $\B(\bif).$

\subsection{An optimized set of generators}\label{sec-optimized-generators}

Following the notation of blocks given in section \ref{sec-fund-resid-sec}, we have:
\begin{lemma}\label{lemma-yz-inN-die}
  If $z\in N$ then
  \begin{equation*}
    y_{z}e(\bif)=0.
  \end{equation*}
\end{lemma}
\begin{proof}
  We proceed by induction on $z\in N=[1:\epsilon].$

  If $z=1,$ then $y_1e(\bif)=0$ by relation \ref{dot-al-inicio}.

  If $z>1$ and we assume that $y_re(\bif)=0$ for all $1\leq r <z,$ then by relations \ref{dot-salta} we have:

  \begin{equation*}
    y_{z}e(\bif)=y_{z-1}e(\bif)-\psi_{z-1}e(s_{z-1}\cdot\bif)\psi_{z-1},
  \end{equation*}

  that is
  \begin{equation}{\label{im-dot-salta2}}
\begin{tikzpicture}[xscale=0.5,yscale=0.65]
  \draw[](-3,0)--(-3,2);
  \node at (-2,1){$\cdots$};
  \draw[](-1,0)--(-1,2);
  \draw[](0,0)--(0,2);
  \draw[](1,0)--(1,2);
  \draw[fill] (1,1) circle [radius=0.15];
  \node[below]at (-3,0) {$i_1$};
  \node[below]at (1,0) {$i_{z}$};
  \node at (2,1){$\quad=\quad$};
  \draw[](3,0)--(3,2);
  \node at (4,1){$\cdots$};
  \draw[](5,0)--(5,2);
  \draw[](6,0)--(6,2);
  \draw[](7,0)--(7,2);
  \draw[fill] (6,1) circle [radius=0.15];
  \node[below]at (3,0) {$i_1$};
  \node[below]at (7,0) {$i_{z}$};

  \node at (8,1){$\quad \sign \quad$};
  \draw[](9,0)--(9,2);
  \node at (10,1){$\cdots$};
  \draw[](11,0)--(11,2);
  \draw[](12,0)to[out=90,in=270](13,1);
  \draw[](13,1)to[out=90,in=270](12,2);
  \draw[](13,0)to[out=90,in=270](12,1);
  \draw[](12,1)to[out=90,in=270](13,2);
  \node[below]at (9,0) {$i_1$};
  \node[below]at (13,0) {$i_{z}$};

\end{tikzpicture}
\end{equation}
By definition of $\bif$ we can see that if $1\leq r <z-1$ then $i_r\neq i_{z},i_{z}\pm1.$ Therefore, applying repeatedly relation \ref{lazo} we have:

\begin{equation*}
    y_{z}e(\bif)=y_{z-1}e(\bif)-(\Psi_{[1:z-1]})^{\ast} e(s_{[1:z-1]}\cdot\bif)(\Psi_{[1:z-1]}),
  \end{equation*}
where $\Psi_{[1:z-1]}=\psi_{1}\psi_{2}\cdots\psi_{z-1}.$

 \begin{equation}{\label{im-lazo-hasta-inicio}}
\begin{tikzpicture}[xscale=0.5,yscale=0.65]
  \draw[](-3,0)--(-3,2);
  \node at (-2,1){$\cdots$};
  \draw[](-1,0)--(-1,2);
  \draw[](0,0)--(0,2);
  \draw[](1,0)--(1,2);
  \draw[fill] (1,1) circle [radius=0.15];
  \node[below]at (-3,0) {$i_1$};
  \node[below]at (1,0) {$i_{z}$};
  \node at (2,1){$\quad=\quad$};
  \draw[](3,0)--(3,2);
  \node at (4,1){$\cdots$};
  \draw[](5,0)--(5,2);
  \draw[](6,0)--(6,2);
  \draw[](7,0)--(7,2);
  \draw[fill] (6,1) circle [radius=0.15];
  \node[below]at (3,0) {$i_1$};
  \node[below]at (7,0) {$i_{z}$};

  \node at (8,1){$\quad \sign \quad$};
  \draw[](10,0)--(10,2);
  \node at (11,1){$\cdots$};
  \draw[](12,0)--(12,2);
  \draw[](14,0)--(9,0.8);
  \draw[](9,1.2)--(9,0.8);
  \draw[](9,1.2)--(14,2);
  \node[below]at (10,0) {$i_1$};
  \node[below]at (14,0) {$i_{z}$};

\end{tikzpicture}
\end{equation}
Note that by definition of $N$ we have that $i_z\notin \{\kappa_1,\dots,\kappa_l\}.$ Then we conclude the proof by combining relation \ref{mal-inicio} and our inductive hypothesis.
\end{proof}

For the rest of the article we adopt the following notation: If $z=m_{(r,j)}$ we denote the element $y_ze(\bif)$ by $\Y_{(r,j)}.$

\begin{lemma}{\label{lemma-clean-Y}}
  If $z\in B(r,j)$ then
  \begin{equation}{\label{eq-clean-Y}}
    y_{z}e(\bif)=\Y_{(r,j)}.
  \end{equation}
\end{lemma}
\begin{proof}
 We only prove the case where $j\neq l-2.$ The other case is analogous.

  If $r=1$ then $B(1,j)=[m_{(1,j)}:m_{(1,j+1)}-1],$ and $m_{(1,j)}\leq z < m_{(1,j+1)},$

  If $z=m_{(1,j)},$ equation \ref{eq-clean-Y} is trivial.

  If $m_{(1,j)}<z < m_{(1,j+1)},$ then by relation \ref{dot-salta} we have:
  \begin{equation*}
    y_ze(\bif)=y_{z-1}e(\bif)-\psi_{z-1}e(s_{z-1}\cdot\bif)\psi_{z-1}.
  \end{equation*}
  (see equation \ref{im-dot-salta2})

  Since $r=1$ and $z<m_{(1,j+1)}$ (defined above) we have that $i_u\neq i_z,i_z\pm1$ for all $u<z-1$ then

  \begin{equation*}
    y_ze(\bif)=y_{z-1}e(\bif)-(\Psi_{[1:z-1]})^{\ast} e(s_{[1: {z-1}]}\cdot\bif)(\Psi_{[1:z-1]})
  \end{equation*}
  (see equation \ref{im-lazo-hasta-inicio})
  and since $i_{z}\notin\{\kappa_1,\dots,\kappa_l\}$ (lemma \ref{lemma-residues-mrj}), we have

  \begin{equation*}
   (\Psi_{[1:z-1]})^{\ast} e(s_{[1: {z-1}]}\cdot\bif)(\Psi_{[1:z-1]})=0.
  \end{equation*}
  (by relation \ref{mal-inicio}).

  Therefore
  \begin{equation*}
    y_ze(\bif)=y_{z-1}e(\bif).
  \end{equation*}
  Now by induction on $z$ we conclude the proof for this case:
  \begin{equation*}
    y_ze(\bif)=\Y_{(1,j)}.
  \end{equation*}

  If $r>1$ and $z\in B(r,j)=[m_{(r,j)}:m_{(r,j+1)}-1].$   Then we have two cases:

  If $z=m_{(r,j)}$ equation \ref{eq-clean-Y} is trivial.

  If $m_{(r,j)}<z<m_{(r,j+1)}$ then by relation \ref{dot-salta} we have:

  \begin{equation*}
    y_ze(\bif)=y_{z-1}e(\bif)-\psi_{z-1}e(s_{z-1}\cdot\bif)\psi_{z-1}.
  \end{equation*}
  (see equation \ref{im-dot-salta2}).

  By definition of $\bif$ we have that $i_u\neq i_z,i_{z}\pm1,$ for all $u$ such that $z-e+1<u<z-1.$

  Therefore by relation \ref{lazo} we have:
  \begin{equation*}
    y_{z}e(\bif)=y_{z-1}e(\bif)-(\Psi_{[x:z-1]})^{\ast}e(\bj)\Psi_{[x:z-1]},
  \end{equation*}
  where $x=z-e+2$ and $\bj=s_{[x:{z-1}]}\cdot\bif.$

 \begin{equation*}
\begin{tikzpicture}[xscale=0.5,yscale=0.65]
  \node at (-2,1){$y_{z}e(\bif)=y_{z-1}e(\bif)\quad\sign$};

  \draw[](2,0)--(2,2);
  \node[below]at (2,0) {$i_1$};
  \node at (3,1){$\cdots$};
  \draw[](5,0)--(5,2);
  \draw[](6,0)--(6,2);
  \node[below]at (6,0) {$i_t$};
  \draw[thick](7,0)--(7,2);
  \node[below]at (7,0) {$i_u$};
  \draw[thick](8,0)--(8,2);
  \node[below]at (8,0) {$i_v$};
  \draw[](10,0)--(10,2);
  \node at (11,1){$\cdots$};
  \draw[](13,0)--(13,2);
  \draw[thick](14,0)--(9,0.8);
  \draw[thick](9,1.2)--(9,0.8);
  \draw[thick](9,1.2)--(14,2);
  \node[below]at (10,0) {$i_x$};
  \node[below]at (14,0) {$i_{z}$};

\end{tikzpicture}
\end{equation*}
  (here $u=z-e,t=u-1,v=u+1$ then $i_u=i_z,i_t=i_z+1$ and $i_v=i_z-1$).

    Let $A=(\Psi_{[x:z-1]})^{\ast}e(\bj)\Psi_{[x:z-1]}.$ If we apply relation \ref{trio} over strings corresponding with $i_u,i_v$ and $i_z$ we obtain $A=A_1-A_2,$ where

\begin{equation*}
\begin{tikzpicture}[xscale=0.5,yscale=0.65]
  \node at (-2,1){$A_1=$};

  \draw[](2,0)--(2,2);
  \node[below]at (2,0) {$i_1$};
  \node at (3,1){$\cdots$};
  \draw[](5,0)--(5,2);
  \draw[](6,0)--(6,2);
  \node[below]at (6,0) {$i_t$};
  \draw[thick](7,0)--(9,0.8);
  \draw[thick](7,2)--(8,1.5);
  \node[below]at (7,0) {$i_u$};
  \draw[thick](8,0)--(6.5,0.8);
  \draw[thick](6.5,1.2)--(6.5,0.8);
  \draw[thick](8,2)--(6.5,1.2);
  \node[below]at (8,0) {$i_v$};
  \draw[](10,0)--(10,2);
  \node at (11,1){$\cdots$};
  \draw[](13,0)--(13,2);
  \draw[thick](14,0)--(8,0.8);
  \draw[thick](9,1.2)--(9,0.8);
  \draw[thick](8,1.5)--(8,0.8);
  \draw[thick](9,1.2)--(14,2);
  \node[below]at (10,0) {$i_x$};
  \node[below]at (14,0) {$i_{z}$};

\end{tikzpicture}
\end{equation*}
 and
\begin{equation*}
\begin{tikzpicture}[xscale=0.5,yscale=0.65]
  \node at (-2,1){$A_2=$};

  \draw[](2,0)--(2,2);
  \node[below]at (2,0) {$i_1$};
  \node at (3,1){$\cdots$};
  \draw[](5,0)--(5,2);
  \draw[](6,0)--(6,2);
  \node[below]at (6,0) {$i_t$};
  \draw[thick](7,0)--(7,2);
  \node[below]at (7,0) {$i_u$};
  \draw[thick](8,0)--(8,2);
  \node[below]at (8,0) {$i_v$};
  \draw[](10,0)--(10,2);
  \node at (11,1){$\cdots$};
  \draw[](13,0)--(13,2);
  \draw[thick](14,0)--(6.5,0.8);
  \draw[thick](6.5,1.2)--(6.5,0.8);
  \draw[thick](6.5,1.2)--(14,2);
  \draw[fill] (6.5,1) circle [radius=0.15];
  \node[below]at (10,0) {$i_x$};
  \node[below]at (14,0) {$i_{z}$};

\end{tikzpicture}
\end{equation*}
Note that $A_1$ belongs to the ideal generated by the idempotent $e(\bjf)$ where $$\bjf=(i_1,i_2,\dots,i_{u-1},i_{u}-1,i_{u},i_{u},i_{x},i_{x+1}\dots)$$
By definition of $\bif=(i_1,i_2,\dots)$ we have that $\bjf$ is a $\kappa-$blob impossible residue sequence (corollaries \ref{coro-bif-j-not-possible} and \ref{coro-bif-j-not-possible2}). Then $A_1=0.$

On the other hand if we apply relation \ref{dot-salta} to the strings corresponding to $i_t$ and $i_z$ in $A_2,$ then we obtain that $A_2=A_{21}-A_{22},$ where

\begin{equation*}
\begin{tikzpicture}[xscale=0.5,yscale=0.65]
  \node at (-2,1){$A_{21}=$};

  \draw[](2,0)--(2,2);
  \node[below]at (2,0) {$i_1$};
  \node at (3,1){$\cdots$};
  \draw[](5,0)--(5,2);
  \draw[](6,0)--(6,2);
  \node[below]at (6,0) {$i_t$};
  \draw[thick](7,0)--(7,2);
  \node[below]at (7,0) {$i_u$};
  \draw[thick](8,0)--(8,2);
  \node[below]at (8,0) {$i_v$};
  \draw[](10,0)--(10,2);
  \node at (11,1){$\cdots$};
  \draw[](13,0)--(13,2);
  \draw[thick](14,0)--(6.5,0.8);
  \draw[thick](6.5,1.2)--(6.5,0.8);
  \draw[thick](6.5,1.2)--(14,2);
  \draw[fill] (6,1) circle [radius=0.1];
  \node[below]at (10,0) {$i_x$};
  \node[below]at (14,0) {$i_{z}$};

\end{tikzpicture}
\end{equation*}
and

\begin{equation*}
\begin{tikzpicture}[xscale=0.5,yscale=0.65]
  \node at (-2,1){$A_{22}=$};

  \draw[](2,0)--(2,2);
  \node[below]at (2,0) {$i_1$};
  \node at (3,1){$\cdots$};
  \draw[](5,0)--(5,2);
  \draw[](6,0)--(6,2);
  \node[below]at (6,0) {$i_t$};
  \draw[thick](7,0)--(7,2);
  \node[below]at (7,0) {$i_u$};
  \draw[thick](8,0)--(8,2);
  \node[below]at (8,0) {$i_v$};
  \draw[](10,0)--(10,2);
  \node at (11,1){$\cdots$};
  \draw[](13,0)--(13,2);
  \draw[thick](14,0)--(5.5,0.8);
  \draw[thick](5.5,1.2)--(5.5,0.8);
  \draw[thick](5.5,1.2)--(14,2);
  \node[below]at (10,0) {$i_x$};
  \node[below]at (14,0) {$i_{z}$};

\end{tikzpicture}
\end{equation*}

Now by relation \ref{lazo} we have that $A_{21}=0$ (since $i_u=i_z$). On the other hand note that $A_{22}$ belongs to the ideal generated by $e(\bjf)\in \B$ where $\bjf=(i_1,i_2,\dots,i_{t-1},i_z,\dots)$ and since $i_z\neq i_t$ and by definition $i_z\notin\{\kappa_2,\dots,\kappa_l\},$ we conclude that $\bjf$ is a $\kappa-$blob impossible residue sequence (corollary \ref{coro-bif-j-not-possible}), then $e(\bjf)=0$ and therefore $A_{22}=0.$

Finally we have obtained that $y_ze(\bif)=y_{z-1}e(\bif),$ then by induction on $z$ we conclude that $y_ze(\bif)=\Y_{(r,j)}.$
\end{proof}

\begin{corollary}\label{coro-Y-generate-B}
  The subalgebra $\calD$ is generated by the set $$\{\Y_{(r,j)}:1\leq r \leq k;\quad 0\leq j \leq l-2\}.$$
\end{corollary}
\begin{proof}
  It follows directly by definition of $\calD$ and lemma \ref{lemma-clean-Y}.
\end{proof}

In the diagrammatic version of $\calD,$ it will be useful to replace all the strings corresponding to the same block, simply by a box representing the corresponding block. Each element $\Y_{(r,j)}$ could be represented by a dot  at the begining of the corresponding box

\begin{example}
For example we can represent:
\begin{equation*}
\begin{tikzpicture}[xscale=0.4,yscale=1.5]
  \node at (-2,0.5) {$e(\underline{\bi})=$};
  \draw[thick](1,0)--(1,1);
  \draw[thick](1,0)--(3,0);
  \draw[thick](3,0)--(3,1);
  \draw[thick](1,1)--(3,1);
  \node[below] at (2,-0.2) {$N$};
  \draw[thick](4,0)--(4,1);
  \draw[thick](4,0)--(6,0);
  \draw[thick](6,0)--(6,1);
  \draw[thick](4,1)--(6,1);
  \node[below] at (5,-0.2) {${B_{(1,0)}}$};
  \draw[thick](7,0)--(7,1);
  \draw[thick](9,0)--(9,1);
  \draw[thick](9,0)--(7,0);
  \draw[thick](9,1)--(7,1);
  \node[below] at (8,-0.2) {${B_{(1,1)}}$};
  \draw[thick](10,0)--(10,1);
  \draw[thick](12,0)--(12,1);
  \draw[thick](10,0)--(12,0);
  \draw[thick](10,1)--(12,1);
  \node[below] at (11,-0.2) {${B_{(1,2)}}$};
  \draw[thick](13,0)--(13,1);
  \draw[thick](15,0)--(15,1);
  \draw[thick](15,0)--(13,0);
  \draw[thick](15,1)--(13,1);
  \node[below] at (14,-0.2) {${B_{(2,0)}}$};
  \draw[thick](16,0)--(16,1);
  \draw[thick](18,0)--(18,1);
  \draw[thick](16,0)--(18,0);
  \draw[thick](16,1)--(18,1);
  \node[below] at (17,-0.2) {${B_{(2,1)}}$};
  \draw[thick](19,0)--(19,1);
  \draw[thick](21,0)--(21,1);
  \draw[thick](19,0)--(21,0);
  \draw[thick](19,1)--(21,1);
  \node[below] at (20,-0.2) {${B_{(2,2)}}$};

\end{tikzpicture}
\end{equation*}

\begin{equation*}
\begin{tikzpicture}[xscale=0.4,yscale=0.5]
  \node at (-2,1.5) {$\Y_{(1,1)}=$};
  \draw[thick](1,0)--(1,3);
  \draw[thick](1,0)--(3,0);
  \draw[thick](3,0)--(3,3);
  \draw[thick](1,3)--(3,3);
  \node[below] at (2,-0.2) {$N$};
  \draw[thick](4,0)--(4,3);
  \draw[thick](4,0)--(6,0);
  \draw[thick](6,0)--(6,3);
  \draw[thick](4,3)--(6,3);
  \node[below] at (5,-0.2) {${B_{(1,0)}}$};
  \draw[thick](7,0)--(7,3);
  \draw[thick](9,0)--(9,3);
  \draw[thick](9,0)--(7,0);
  \draw[thick](9,3)--(7,3);
  \draw[fill] (7,1.5) circle [radius=0.15];
  \node[below] at (8,-0.2) {${B_{(1,1)}}$};
  \draw[thick](10,0)--(10,3);
  \draw[thick](12,0)--(12,3);
  \draw[thick](10,0)--(12,0);
  \draw[thick](10,3)--(12,3);
  \node[below] at (11,-0.2) {${B_{(1,2)}}$};
  \draw[thick](13,0)--(13,3);
  \draw[thick](15,0)--(15,3);
  \draw[thick](15,0)--(13,0);
  \draw[thick](15,3)--(13,3);
  \node[below] at (14,-0.2) {${B_{(2,0)}}$};
  \draw[thick](16,0)--(16,3);
  \draw[thick](18,0)--(18,3);
  \draw[thick](16,0)--(18,0);
  \draw[thick](16,3)--(18,3);
  \node[below] at (17,-0.2) {${B_{(2,1)}}$};
  \draw[thick](19,0)--(19,3);
  \draw[thick](21,0)--(21,3);
  \draw[thick](19,0)--(21,0);
  \draw[thick](19,3)--(21,3);
  \node[below] at (20,-0.2) {${B_{(2,2)}}$};

\end{tikzpicture}
\end{equation*}

\begin{equation*}
\begin{tikzpicture}[xscale=0.4,yscale=0.5]
  \node at (-4,1.5) {$(\Y_{(1,0)})^{2}\Y_{(1,1)}\Y_{(2,2)}=$};
  \draw[thick](1,0)--(1,3);
  \draw[thick](1,0)--(3,0);
  \draw[thick](3,0)--(3,3);
  \draw[thick](1,3)--(3,3);
  \node[below] at (2,-0.2) {$N$};
  \draw[thick](4,0)--(4,3);
  \draw[thick](4,0)--(6,0);
  \draw[thick](6,0)--(6,3);
  \draw[thick](4,3)--(6,3);
  \draw[fill] (4,1) circle [radius=0.15];
  \draw[fill] (4,2) circle [radius=0.15];
  \node[below] at (5,-0.2) {${B_{(1,0)}}$};
  \draw[thick](7,0)--(7,3);
  \draw[thick](9,0)--(9,3);
  \draw[thick](9,0)--(7,0);
  \draw[thick](9,3)--(7,3);
  \draw[fill] (7,1.5) circle [radius=0.15];
  \node[below] at (8,-0.2) {${B_{(1,1)}}$};
  \draw[thick](10,0)--(10,3);
  \draw[thick](12,0)--(12,3);
  \draw[thick](10,0)--(12,0);
  \draw[thick](10,3)--(12,3);
  \node[below] at (11,-0.2) {${B_{(1,2)}}$};
  \draw[thick](13,0)--(13,3);
  \draw[thick](15,0)--(15,3);
  \draw[thick](15,0)--(13,0);
  \draw[thick](15,3)--(13,3);
  \node[below] at (14,-0.2) {${B_{(2,0)}}$};
  \draw[thick](16,0)--(16,3);
  \draw[thick](18,0)--(18,3);
  \draw[thick](16,0)--(18,0);
  \draw[thick](16,3)--(18,3);
  \node[below] at (17,-0.2) {${B_{(2,1)}}$};
  \draw[thick](19,0)--(19,3);
  \draw[thick](21,0)--(21,3);
  \draw[thick](19,0)--(21,0);
  \draw[thick](19,3)--(21,3);
  \draw[fill] (19,1.5) circle [radius=0.15];
  \node[below] at (20,-0.2) {${B_{(2,2)}}$};

\end{tikzpicture}
\end{equation*}
and so on.
\end{example}

It will be very useful to define the following set of elements in $\mathcal{D}.$

\begin{definition}\label{def-L-elements}
  For any pair $(r,j)$ with $1\leq r \leq k$ and $0\leq j \leq l-2,$ we define the element $\calL_{(r,j)}$ as:
  \begin{equation*}
    \calL_{(r,j)}=\left\{
    \begin{array}{cc}
      -\Y_{(1,0)} & \quad \textrm{if}\quad (r,j)=(1,0). \\
      \quad & \quad \\
      \Y_{(r,j-1)}-\Y_{(r,j)} & \quad \textrm{if}\quad j\neq 0. \\
       \quad & \quad \\
      \Y_{(r-1,l-2)}-\Y_{(r,0)} & \quad \textrm{if}\quad r>1\wedge j=0.
    \end{array}
    \right.
  \end{equation*}
\end{definition}

\begin{lemma}\label{lema-calL-is-L}
  Let $1\leq r \leq k$ and $0\leq j \leq l-2.$ Then
  \begin{equation*}
    \calL_{(r,j)}=L_{m_{(r,j)}}e(\bif).
  \end{equation*}
\end{lemma}
\begin{proof}
  It follows directly from definitions \ref{def-L-general}, \ref{def-L-elements} and lemma \ref{lemma-clean-Y}.
\end{proof}

\begin{lemma}\label{lemma-L-generate-D}
  The algebra $\mathcal{D}$ is generated by
  $$\{\calL_{(r,j)}:1\leq r \leq k;\quad 0\leq j \leq l-2\}$$
\end{lemma}
\begin{proof}
By definition each element $\calL_{(r,j)}$ belongs to $\mathcal{D}.$
  On the other hand by definition $\Y_{(1,0)}$  belongs to the subalgebra generated by the elements $\calL_{(r,j)}.$

   If $(r,j)\neq(1,0),$ we have
  \begin{equation*}
    \Y_{(r,j)}=\left\{\begin{array}{cc}
                        \Y_{(r,j-1)}-\calL_{(r,j)} & \quad\textrm{if}\quad j\neq 0 \\
                        \quad & \quad \\
                        \Y_{(r-1,l-2)}-\calL_{(r,0)} & \quad\textrm{if}\quad j= 0
                      \end{array}\right.
  \end{equation*}
  then by an inductive argument on the set of numbers $m_{(r,j)}$ we obtain that $\Y_{(r,j)}$ belongs to the algebra generated by the elements $\calL_{(r,j)}.$
  Finally by corollary \ref{coro-Y-generate-B} we obtain the desired result.
\end{proof}

\begin{corollary}\label{coro-telescopic-Y-in-L-2} Let $1\leq r\leq k$ and $0\leq j\leq l-2.$ Then
\begin{equation*}
  \sum_{m_{(s,t)}\leq m_{(r,j)}}\calL_{(s,t)}=-\Y_{(r,j)}.
\end{equation*}
\end{corollary}

\begin{proof}
  It follows by definition of $\calL_{(r,j)}$ and telescopic arguments.
\end{proof}

\subsection{Fundamental relations}\label{sec-fundamental-relations-for-D}

The next lemma provides a convenient diagrammatic representation for elements $\calL_{(r,j)}$

\begin{lemma}\label{lemma-diag-L}
Let $1\leq r\leq k,$ $0\leq  j \leq l-2$ and $z=m_{(r,j)}-1.$ Then
\begin{equation*}
  \calL_{(r,j)}=(\Psi_{[1:z]})^{\ast}e(\bj)\Psi_{[1:z]},\quad\textrm{where}\quad\bj=s_{[1:z]}\cdot\bif.
\end{equation*}
\end{lemma}

\begin{proof}
First we aboard the case where $r=1.$ Note that for every $1\leq u<m_{(1,0)}$ we have $i_u\neq i_{m(1,0)},i_{m(1,0)}\pm 1.$ By relation \ref{dot-salta} we have
  \begin{equation*}
    \Y_{(1,0)}=y_ze(\bif)-A
  \end{equation*}
  where $z=m_{(1,0)}-1,$ $A=(\Psi_{[1:z]})^{\ast}e(\bj)\Psi_{[1:z]}$ and $\bj=s_{[1:z]}\cdot\bif$

  \begin{equation*}
\begin{tikzpicture}[xscale=0.4,yscale=0.5]
  \node at (-2,1.5) {$A=$};
  \draw[thick](1,0)--(1,3);
  \draw[thick](1,0)--(3,0);
  \draw[thick](3,0)--(3,3);
  \draw[thick](1,3)--(3,3);
  \node[below] at (2,-0.2) {$N$};
  \draw[thick,red](4,0)--(0,1);
  \draw[thick,red](0,2)--(0,1);
  \draw[thick,red](0,2)--(4,3);

  \draw[thick,red](4.3,0)--(6,0);
  \draw[thick,red](6,0)--(6,3);
  \draw[thick,red](4.3,3)--(6,3);
  \draw[thick,red](4,3.3)--(4,3);
  \node[below] at (5,-0.2) {${\color{red}B_{(1,0)}}$};
  \draw[thick](7,0)--(7,3);
  \draw[thick](9,0)--(9,3);
  \draw[thick](9,0)--(7,0);
  \draw[thick](9,3)--(7,3);
  \node[below] at (8,-0.2) {${B_{(1,1)}}$};
  \node[below] at (10,1.5) {$\cdots$};

\end{tikzpicture}
\end{equation*}

Since $z=m_{(1,0)}-1$ we have that $z\in N$ then $y_ze(\bif)=0$ (lemma \ref{lemma-yz-inN-die}). Then $\calL_{(1,0)}=-\Y_{(1,0)}=A.$

Similarly if $0<j\leq l-2,$ we can apply relation \ref{dot-salta} and  relation \ref{lazo} to obtain
\begin{equation}\label{eq1-lemma-diag-L}
  \Y_{(1,j)}=y_ze(\bif)-A
\end{equation}

where $z=m_{(1,j)}-1,$ $A=(\Psi_{[1:z]})^{\ast}e(\bj)\Psi_{[1:z]}$ and $\bj=s_{[1:z]}\cdot \bif.$

\begin{equation*}
\begin{tikzpicture}[xscale=0.4,yscale=1.5]
  \node at (-2,0.5) {$A=$};
  \draw[thick](1,-0.1)--(1,1.1);
  \draw[thick](1,-0.1)--(3,-0.1);
  \draw[thick](3,-0.1)--(3,1.1);
  \draw[thick](1,1.1)--(3,1.1);
  \node[below] at (2,-0.3) {$N$};
  \draw[thick](4,-0.1)--(4,1.1);
  \draw[thick](4,-0.1)--(6,-0.1);
  \draw[thick](6,-0.1)--(6,1.1);
  \draw[thick](4,1.1)--(6,1.1);
  \node[below] at (5,-0.3) {$B_{(1,0)}$};
  \draw[thick](7,-0.1)--(7,1.1);
  \draw[thick](9,-0.1)--(9,1.1);
  \draw[thick](9,-0.1)--(7,-0.1);
  \draw[thick](9,1.1)--(7,1.1);
  \node[below] at (8,-0.3) {$B_{(1,1)}$};
  \node at (11,0.5) {$\cdots$};
  \draw[thick](13,-0.1)--(13,1.1);
  \draw[thick](15,-0.1)--(15,1.1);
  \draw[thick](15,-0.1)--(13,-0.1);
  \draw[thick](15,1.1)--(13,1.1);
  \draw[thick,red](16,0)--(0,0.4);
  \draw[thick,red](16,0)--(16,-0.1);
  \draw[thick,red](16,1)--(16,1.1);
  \draw[thick,red](0,0.4)--(0,0.6);
  \draw[thick,red](0,0.6)--(16,1);
  \draw[thick,red](18,-0.1)--(18,1.1);
  \draw[thick,red](16.3,-0.1)--(18,-0.1);
  \draw[thick,red](16.3,1.1)--(18,1.1);
  \node[below] at (17,-0.3) {${\color{red}B_{(1,j)}}$};
  \node at (20,0.5) {$\cdots$};

\end{tikzpicture}
\end{equation*}
By lemma \ref{lemma-clean-Y} we have that $y_ze(\bif)=\Y_{(1,j-1)}$ then by definition \ref{def-L-elements} and reorganizing equation \ref{eq1-lemma-diag-L}, we conclude that $\calL_{(1,j)}=A,$ as desired.

If $r>1$ then by relations \ref{dot-salta} and \ref{lazo} we have:
\begin{equation*}
  \Y_{(r,j)}=y_ze(\bif)-A,
\end{equation*}
where $A=(\Psi_{[x:z]})^{\ast}e(\bj)\Psi_{[x:z]},$ $\bj=s_{[x:z]}\cdot\bif$ and $x=m_{(r-1,j)}+2.$

\begin{equation*}
\begin{tikzpicture}[xscale=0.5,yscale=0.65]
  \node at (-2,1){$A=$};
  \draw[thick](0,0)--(0,2);
  \draw[thick](0,0)--(2,0);
  \draw[thick](0,2)--(2,2);
  \draw[thick](2,0)--(2,2);
  \node[below]at (1,-1.3) {$N$};
  \node at (3,1){$\cdots$};
  \draw[thick](4,0)--(4,2);
  \draw[thick](4,0)--(6,0);
  \draw[thick](4,2)--(6,2);
  \draw[thick](6,0)--(6,2);
  \node[below]at (5,-1.3) {$B_{(r-1,j-1)}$};
  \draw[red,thick](7.5,0)--(7.5,2);
  \draw[red,thick](7.5,2)--(10.5,2);
  \draw[red,thick](7.5,0)--(10.5,0);
  \node[below]at (7.5,-0.3) {${\color{red}i_u}$};
  \node[below]at (9,-1.3) {${\color{red}B_{(r-1,j)}}$};
  \draw[red](8.5,0)--(8.5,2);
  \node[below]at (8.5,-0.3) {${\color{red}i_v}$};
  \draw[red](9.5,0)--(9.5,2);
  \draw[red,thick](10.5,0)--(10.5,2);
  \node at (11,1){$\cdots$};
  \draw[thick](13,0)--(13,2);
  \draw[thick](11.5,0)--(11.5,2);
  \draw[thick](11.5,0)--(13,0);
  \draw[thick](11.5,2)--(13,2);
  \draw[red,thick](14,0)--(9,0.8);
  \draw[red,thick](9,1.2)--(9,0.8);
  \draw[red,thick](9,1.2)--(14,2);
  \draw[red,thick](14,2.2)--(14,2);
  \draw[red,thick](14,-0.2)--(14,0);
  \draw[red,thick](16.5,0)--(16.5,2);
  \draw[red,thick](14.2,0)--(16.5,0);
  \draw[red,thick](14.2,2)--(16.5,2);
  \node at (17,1){$\cdots$};
  \node[below]at (9.5,-0.3) {${\color{red}i_x}$};
  \node[below]at (14.2,-0.3) {${\color{red}i_{z+1}}$};
  \node[below]at (15.5,-1.3) {${\color{red}B_{(r,j)}}$};

\end{tikzpicture}
\end{equation*}

  (here $u=m_{(r-1,j)},v=u+1$ then $i_u=i_{m_{(r,j)}}$ and $i_v=i_u-1$). If we apply relation \ref{trio} on the strings corresponding to $i_u,i_v,i_{z+1}$ we obtain $A=A_1-A_2,$ where

\begin{equation*}
\begin{tikzpicture}[xscale=0.5,yscale=0.65]
  \node at (-2,1){$A_1=$};
  \draw[thick](0,0)--(0,2);
  \draw[thick](0,0)--(2,0);
  \draw[thick](0,2)--(2,2);
  \draw[thick](2,0)--(2,2);
  \node[below]at (1,-1.3) {$N$};
  \node at (3,1){$\cdots$};
  \draw[thick](4,0)--(4,2);
  \draw[thick](4,0)--(6,0);
  \draw[thick](4,2)--(6,2);
  \draw[thick](6,0)--(6,2);
  \node[below]at (5,-1.3) {$B_{(r-1,j-1)}$};
  \draw[red,thick](7.5,2)--(10.5,2);
  \draw[red,thick](7.5,0)--(10.5,0);
  \node[below]at (7.5,-0.3) {${\color{red}i_u}$};
  \node[below]at (9,-1.3) {${\color{red}B_{(r-1,j)}}$};
  \draw[red](8.5,0)--(7.5,0.5);
  \draw[red](7.5,1.5)--(7.5,0.5);
  \draw[red](7.5,1.5)--(8.5,2);

  \node[below]at (8.5,-0.3) {${\color{red}i_v}$};
  \draw[red](9.5,0)--(9.5,2);
  \draw[red,thick](10.5,0)--(10.5,2);
  \node at (11,1){$\cdots$};
  \draw[thick](13,0)--(13,2);
  \draw[thick](11.5,0)--(11.5,2);
  \draw[thick](11.5,0)--(13,0);
  \draw[thick](11.5,2)--(13,2);

  \draw[red,thick](14,0)--(8.5,0.8);
  \draw[red,thick](9,1.2)--(9,0.8);
  \draw[red,thick](7.5,0)--(9,0.8);
  \draw[red,thick](9,1.2)--(14,2);

  \draw[red,thick](8.5,1.2)--(7.5,2);
  \draw[red,thick](8.5,1.2)--(8.5,0.8);
  \draw[red,thick](14,2.2)--(14,2);
  \draw[red,thick](14,-0.2)--(14,0);
  \draw[red,thick](16.5,0)--(16.5,2);
  \draw[red,thick](14.2,0)--(16.5,0);
  \draw[red,thick](14.2,2)--(16.5,2);
  \node at (17,1){$\cdots$};
  \node[below]at (9.5,-0.3) {${\color{red}i_x}$};
  \node[below]at (14.2,-0.3) {${\color{red}i_{z+1}}$};
  \node[below]at (15.5,-1.3) {${\color{red}B_{(r,j)}}$};

\end{tikzpicture}
\end{equation*}

and

\begin{equation*}
\begin{tikzpicture}[xscale=0.5,yscale=0.65]
  \node at (-2,1){$A_2=$};
  \draw[thick](0,0)--(0,2);
  \draw[thick](0,0)--(2,0);
  \draw[thick](0,2)--(2,2);
  \draw[thick](2,0)--(2,2);
  \node[below]at (1,-1.3) {$N$};
  \node at (3,1){$\cdots$};
  \draw[thick](4,0)--(4,2);
  \draw[thick](4,0)--(6,0);
  \draw[thick](4,2)--(6,2);
  \draw[thick](6,0)--(6,2);
  \node[below]at (5,-1.3) {$B_{(r-1,j-1)}$};
  \draw[red,thick](7.5,0)--(7.5,2);
  \draw[red,thick](7.5,2)--(10.5,2);
  \draw[red,thick](7.5,0)--(10.5,0);
  \node[below]at (7.5,-0.3) {${\color{red}i_u}$};
  \node[below]at (9,-1.3) {${\color{red}B_{(r-1,j)}}$};
  \draw[red](8.5,0)--(8.5,2);
  \node[below]at (8.5,-0.3) {${\color{red}i_v}$};
  \draw[red](9.5,0)--(9.5,2);
  \draw[red,thick](10.5,0)--(10.5,2);
  \node at (11,1){$\cdots$};
  \draw[thick](13,0)--(13,2);
  \draw[thick](11.5,0)--(11.5,2);
  \draw[thick](11.5,0)--(13,0);
  \draw[thick](11.5,2)--(13,2);
  \draw[red,thick](14,0)--(7,0.8);
  \draw[red,thick](7,1.2)--(7,0.8);
  \draw[red,thick](7,1.2)--(14,2);
  \draw[fill,red] (7,1) circle [radius=0.1];
  \draw[red,thick](14,2.2)--(14,2);
  \draw[red,thick](14,-0.2)--(14,0);
  \draw[red,thick](16.5,0)--(16.5,2);
  \draw[red,thick](14.2,0)--(16.5,0);
  \draw[red,thick](14.2,2)--(16.5,2);
  \node at (17,1){$\cdots$};
  \node[below]at (9.5,-0.3) {${\color{red}i_x}$};
  \node[below]at (14.2,-0.3) {${\color{red}i_{z+1}}$};
  \node[below]at (15.5,-1.3) {${\color{red}B_{(r,j)}}$};

\end{tikzpicture}
\end{equation*}

Then $A_1$ belongs to the ideal generated by $e(\bj),$ where $\bj=(i_1,\dots,i_{u-1},i_{u}-1,i_u,i_u,\dots)$ and by corollary \ref{coro-bif-j-not-possible}, we conclude that $A_1=0.$

On the other hand applying relation \ref{dot-salta}, lemma \ref{lemma-clean-Y} and relation \ref{lazo} in $A_2,$ we obtain $A_2=A_{21}-A_{22},$ where

\begin{equation*}
\begin{tikzpicture}[xscale=0.5,yscale=0.65]
  \node at (-2,1){$A_{21}=$};
  \draw[thick](0,0)--(0,2);
  \draw[thick](0,0)--(2,0);
  \draw[thick](0,2)--(2,2);
  \draw[thick](2,0)--(2,2);
  \node[below]at (1,-1.3) {$N$};
  \node at (3,1){$\cdots$};
  \draw[thick](4,0)--(4,2);
  \draw[thick](4,0)--(6,0);
  \draw[thick](4,2)--(6,2);
  \draw[thick](6,0)--(6,2);
  \node[below]at (5,-1.3) {$B_{(r-1,j-1)}$};
  \draw[red,thick](7.5,0)--(7.5,2);
  \draw[red,thick](7.5,2)--(10.5,2);
  \draw[red,thick](7.5,0)--(10.5,0);
  \node[below]at (7.5,-0.3) {${\color{red}i_u}$};
  \node[below]at (9,-1.3) {${\color{red}B_{(r-1,j)}}$};
  \draw[red](8.5,0)--(8.5,2);
  \node[below]at (8.5,-0.3) {${\color{red}i_v}$};
  \draw[red](9.5,0)--(9.5,2);
  \draw[red,thick](10.5,0)--(10.5,2);
  \node at (11,1){$\cdots$};
  \draw[thick](13,0)--(13,2);
  \draw[thick](11.5,0)--(11.5,2);
  \draw[thick](11.5,0)--(13,0);
  \draw[thick](11.5,2)--(13,2);
  \draw[red,thick](14,0)--(7,0.8);
  \draw[red,thick](7,1.2)--(7,0.8);
  \draw[red,thick](7,1.2)--(14,2);
  \draw[fill] (4,1) circle [radius=0.1];
  \draw[red,thick](14,2.2)--(14,2);
  \draw[red,thick](14,-0.2)--(14,0);
  \draw[red,thick](16.5,0)--(16.5,2);
  \draw[red,thick](14.2,0)--(16.5,0);
  \draw[red,thick](14.2,2)--(16.5,2);
  \node at (17,1){$\cdots$};
  \node[below]at (9.5,-0.3) {${\color{red}i_x}$};
  \node[below]at (14.2,-0.3) {${\color{red}i_{z+1}}$};
  \node[below]at (15.5,-1.3) {${\color{red}B_{(r,j)}}$};

\end{tikzpicture}
\end{equation*}

and

\begin{equation*}
\begin{tikzpicture}[xscale=0.5,yscale=0.65]
  \node at (-2,1){$A_{22}=$};
  \draw[thick](0,0)--(0,2);
  \draw[thick](0,0)--(2,0);
  \draw[thick](0,2)--(2,2);
  \draw[thick](2,0)--(2,2);
  \node[below]at (1,-1.3) {$N$};
  \node at (3,1){$\cdots$};
  \draw[thick](4,0)--(4,2);
  \draw[thick](4,0)--(6,0);
  \draw[thick](4,2)--(6,2);
  \draw[thick](6,0)--(6,2);
  \node[below]at (5,-1.3) {$B_{(r-1,j-1)}$};
  \draw[red,thick](7.5,0)--(7.5,2);
  \draw[red,thick](7.5,2)--(10.5,2);
  \draw[red,thick](7.5,0)--(10.5,0);
  \node[below]at (7.5,-0.3) {${\color{red}i_u}$};
  \node[below]at (9,-1.3) {${\color{red}B_{(r-1,j)}}$};
  \draw[red](8.5,0)--(8.5,2);
  \node[below]at (8.5,-0.3) {${\color{red}i_v}$};
  \draw[red](9.5,0)--(9.5,2);
  \draw[red,thick](10.5,0)--(10.5,2);
  \node at (11,1){$\cdots$};
  \draw[thick](13,0)--(13,2);
  \draw[thick](11.5,0)--(11.5,2);
  \draw[thick](11.5,0)--(13,0);
  \draw[thick](11.5,2)--(13,2);
  \draw[red,thick](14,0.1)--(3.5,0.8);
  \draw[red,thick](3.5,1.2)--(3.5,0.8);
  \draw[red,thick](3.5,1.2)--(14,1.9);
  \draw[red,thick](14,2.2)--(14,1.9);
  \draw[red,thick](14,-0.2)--(14,0.1);
  \draw[red,thick](16.5,0)--(16.5,2);
  \draw[red,thick](14.2,0)--(16.5,0);
  \draw[red,thick](14.2,2)--(16.5,2);
  \node at (17,1){$\cdots$};
  \node[below]at (9.5,-0.3) {${\color{red}i_x}$};
  \node[below]at (14.2,-0.3) {${\color{red}i_{z+1}}$};
  \node[below]at (15.5,-1.3) {${\color{red}B_{(r,j)}}$};

\end{tikzpicture}
\end{equation*}

By relation \ref{lazo} we conclude that $A_{21}=0$ since $i_u=i_{z+1}.$ Then
\begin{equation*}
  \Y_{(r,j)}=y_ze(\bif)-A_{22}.
\end{equation*}

Now by lemma \ref{lemma-clean-Y} we obtain

\begin{equation}\label{eq-despejar-L}
  \Y_{(r,j)}=\Y_{(r,j-1)}-A_{22}.
\end{equation}

Applying relation \ref{lazo} we can pull the string corresponding with $i_{z+1}$ in $A_{22}$ to the first position or at least to the next block $B_{(s,j)}$ $(1\leq s\leq r-1)$

\begin{equation*}
\begin{tikzpicture}[xscale=0.5,yscale=0.65]
  \node at (-2,1){$A_{22}=$};
  \draw[thick](0,0)--(0,2);
  \draw[thick](0,0)--(2,0);
  \draw[thick](0,2)--(2,2);
  \draw[thick](2,0)--(2,2);
  \node[below]at (1,-1.3) {$N$};
  \node at (3,1){$\cdots$};
  \draw[thick](4,0)--(4,2);
  \draw[thick](4,0)--(6,0);
  \draw[thick](4,2)--(6,2);
  \draw[thick](6,0)--(6,2);
  \draw[red,thick](7.5,0)--(7.5,2);
  \draw[red,thick](7.5,2)--(10.5,2);
  \draw[red,thick](7.5,0)--(10.5,0);
  \node[below]at (9,-1.3) {${\color{red}B_{(s,j)}}$};
  \draw[red](8.5,0)--(8.5,2);
  \draw[red](9.5,0)--(9.5,2);
  \draw[red,thick](10.5,0)--(10.5,2);
  \node at (11,1){$\cdots$};
  \draw[thick](13,0)--(13,2);
  \draw[thick](11.5,0)--(11.5,2);
  \draw[thick](11.5,0)--(13,0);
  \draw[thick](11.5,2)--(13,2);
  \draw[red,thick](14,0)--(9,0.8);
  \draw[red,thick](9,1.2)--(9,0.8);
  \draw[red,thick](9,1.2)--(14,2);
  \draw[red,thick](14,2.2)--(14,2);
  \draw[red,thick](14,-0.2)--(14,0);
  \draw[red,thick](16.5,0)--(16.5,2);
  \draw[red,thick](14.2,0)--(16.5,0);
  \draw[red,thick](14.2,2)--(16.5,2);
  \node at (17,1){$\cdots$};
  \node[below]at (14.2,-0.3) {${\color{red}i_{z+1}}$};
  \node[below]at (15.5,-1.3) {${\color{red}B_{(r,j)}}$};

\end{tikzpicture}
\end{equation*}
then we can apply recursively the same arguments as above and cross throughout this block:
\begin{equation*}
\begin{tikzpicture}[xscale=0.5,yscale=0.65]
  \node at (-2,1){$A_{22}=$};
  \draw[thick](0,0)--(0,2);
  \draw[thick](0,0)--(2,0);
  \draw[thick](0,2)--(2,2);
  \draw[thick](2,0)--(2,2);
  \node[below]at (1,-1.3) {$N$};
  \node at (3,1){$\cdots$};
  \draw[thick](4,0)--(4,2);
  \draw[thick](4,0)--(6,0);
  \draw[thick](4,2)--(6,2);
  \draw[thick](6,0)--(6,2);
  \draw[red,thick](7.5,0)--(7.5,2);
  \draw[red,thick](7.5,2)--(10.5,2);
  \draw[red,thick](7.5,0)--(10.5,0);
  \node[below]at (9,-1.3) {${\color{red}B_{(s,j)}}$};
  \draw[red](8.5,0)--(8.5,2);
  \draw[red](9.5,0)--(9.5,2);
  \draw[red,thick](10.5,0)--(10.5,2);
  \node at (11,1){$\cdots$};
  \draw[thick](13,0)--(13,2);
  \draw[thick](11.5,0)--(11.5,2);
  \draw[thick](11.5,0)--(13,0);
  \draw[thick](11.5,2)--(13,2);
  \draw[red,thick](14,0.1)--(3.5,0.8);
  \draw[red,thick](3.5,1.2)--(3.5,0.8);
  \draw[red,thick](3.5,1.2)--(14,1.9);
  \draw[red,thick](14,2.2)--(14,1.9);
  \draw[red,thick](14,-0.2)--(14,0.1);
  \draw[red,thick](16.5,0)--(16.5,2);
  \draw[red,thick](14.2,0)--(16.5,0);
  \draw[red,thick](14.2,2)--(16.5,2);
  \node at (17,1){$\cdots$};
  \node[below]at (14.2,-0.3) {${\color{red}i_{z+1}}$};
  \node[below]at (15.5,-1.3) {${\color{red}B_{(r,j)}}$};

\end{tikzpicture}
\end{equation*}

Since there are finite number of this type of blocks, in some point we will be able to pull the string to the first position. Therefore  $$A_{22}=(\Psi_{[1:z]})^{\ast}e(\bj)\Psi_{[1:z]}\quad (\bj=s_{[1:z]}\cdot \bif).$$

If we reorder equation \ref{eq-despejar-L} we obtain the desired result.
\end{proof}

\begin{lemma}\label{lemma-square-Y10-die}
  \begin{equation*}
    (\Y_{(1,0)})^{2}=0.
  \end{equation*}
\end{lemma}
\begin{proof}
By definition \ref{def-L-elements} and lemma \ref{lemma-diag-L} we have
\begin{equation*}
-\Y_{(1,0)}=\calL_{(1,0)}=(\Psi_{[1:\epsilon]})^{\ast}e(\bj)\Psi_{[1:\epsilon]},
\end{equation*}
 where $\bj=s_{[1:\epsilon]}\cdot\bif,$ then
 \begin{equation*}
   (\Y_{(1,0)})^2=-\Y_{(1,0)}(\Psi_{[1:\epsilon]})^{\ast}e(\bj)\Psi_{[1:\epsilon]}.
 \end{equation*}

Finally by relations \ref{punto-abajo}, \ref{dot-al-inicio} we have $$(\Y_{(1,0)})^2=-(\Psi_{[1:z]})^{\ast}y_1e(\bj)\Psi_{[1:z]}=0$$

 \begin{equation*}
\begin{tikzpicture}[xscale=0.4,yscale=0.5]
  \node at (-3,1.5) {$(\Y_{(1,0)})^2=$};
  \draw[thick](1,0)--(1,3);
  \draw[thick](1,0)--(3,0);
  \draw[thick](3,0)--(3,3);
  \draw[thick](1,3)--(3,3);
  \node[below] at (2,-0.2) {$N$};

  \draw[thick,red](4,0)--(0,1);
  \draw[thick,red](0,2)--(0,1);
  \draw[thick,red](0,2)--(4,3);
  \draw[fill,red] (0,1.5) circle [radius=0.2];

  \draw[thick,red](4.3,0)--(6,0);
  \draw[thick,red](6,0)--(6,3);
  \draw[thick,red](4.3,3)--(6,3);
  \draw[thick,red](4,3.3)--(4,3);
  \node[below] at (5,-0.2) {${\color{red}B_{(1,0)}}$};
  \draw[thick](7,0)--(7,3);
  \draw[thick](9,0)--(9,3);
  \draw[thick](9,0)--(7,0);
  \draw[thick](9,3)--(7,3);
  \node[below] at (8,-0.2) {${B_{(1,1)}}$};
  \node[] at (10,1.5) {$\cdots$};
  \node[] at (12,1.5) {$=0$};

\end{tikzpicture}
\end{equation*}
\end{proof}

\begin{corollary}\label{coro-L10-die}
\begin{equation*}
 (\calL_{(1,0)})^2=0.
\end{equation*}
\end{corollary}
\begin{proof}
  It follows directly from definition \ref{def-L-elements} and lemma \ref{lemma-square-Y10-die}.
\end{proof}

\begin{lemma}\label{lemma-square-Y1j}
Let $0<j\leq l-2,$ then
  \begin{equation*}
    (\Y_{(1,j)})^{2}=\Y_{(1,j-1)}\Y_{(1,j)}.
  \end{equation*}
\end{lemma}
\begin{proof}
  By definition \ref{def-L-elements} and lemma \ref{lemma-diag-L} we have:
  \begin{equation*}
    \Y_{(1,j-1)}-\Y_{(1,j)}=\calL_{(1,j)}=(\Psi_{[1:z]})^{\ast}e(\bj)\Psi_{[1:z]},
  \end{equation*}
  where $z=m_{(1,j)}-1$ and $\bj=s_{[1:z]}\cdot \bif.$

  Therefore

  \begin{equation*}
    \Y_{(1,j-1)}\Y_{(1,j)}-(\Y_{(1,j)})^2=\Y_{(1,j)}\calL_{(1,j)}=(\Psi_{[1:z]})^{\ast}y_1e(\bj)\Psi_{[1:z]},
  \end{equation*}
  the last equation is obtained applying relation \ref{punto-abajo}

  \begin{equation*}
\begin{tikzpicture}[xscale=0.5,yscale=0.65]
  \node at (0,1){$\Y_{(1,j)}\calL_{(1,j)}=$};
  \draw[thick](4,0)--(4,2);
  \draw[fill,red] (3.5,1) circle [radius=0.1];
  \draw[thick](4,0)--(6,0);
  \draw[thick](4,2)--(6,2);
  \draw[thick](6,0)--(6,2);
  \draw[thick](7.5,0)--(7.5,2);
  \draw[thick](7.5,2)--(9.5,2);
  \draw[thick](7.5,0)--(9.5,0);
  \draw[thick](9.5,0)--(9.5,2);
  \node at (11,1){$\cdots$};
  \draw[thick](13,0)--(13,2);
  \draw[thick](11.5,0)--(11.5,2);
  \draw[thick](11.5,0)--(13,0);
  \draw[thick](11.5,2)--(13,2);
  \draw[red,thick](14,0.1)--(3.5,0.8);
  \draw[red,thick](3.5,1.2)--(3.5,0.8);
  \draw[red,thick](3.5,1.2)--(14,1.9);
  \draw[red,thick](14,2.2)--(14,1.9);
  \draw[red,thick](14,-0.2)--(14,0.1);
  \draw[red,thick](16.5,0)--(16.5,2);
  \draw[red,thick](14.2,0)--(16.5,0);
  \draw[red,thick](14.2,2)--(16.5,2);
  \node[below]at (15.5,-1.3) {${\color{red}B_{(1,j)}}$};

\end{tikzpicture}
\end{equation*}

  By relation \ref{dot-al-inicio} we obtain

  \begin{equation*}
    \Y_{(1,j-1)}\Y_{(1,j)}-(\Y_{(1,j)})^2=\Y_{(1,j)}\calL_{(1,j)}=0
  \end{equation*}
  and the desired result.
\end{proof}

\begin{corollary}\label{coro-square-L1j}
 \begin{equation*}
   (\calL_{(1,j)})^2=-\sum_{0\leq v\leq j-1}\calL_{(1,v)}\calL_{(1,j)}.
\end{equation*}
\end{corollary}
\begin{proof}
  It follows directly from definition \ref{def-L-elements}, corollary \ref{coro-telescopic-Y-in-L-2}, and lemma \ref{lemma-square-Y1j}.
\end{proof}



The following definition provide a convenient extension of elements $\calL_{(r,j)}.$
\begin{definition}\label{def-element-eraser-Lrj-left}
  Let $1 < r\leq k$ and $0\leq j\leq l-2,$ and $1\leq s<r.$   We define the elements $E^{(s)}(\calL_{(r,j)})$ as follows:
  \begin{equation*}
    E^{(s)}(\calL_{(r,j)})=(\psi_{1}\cdots\overbrace{\psi_{m_{(s,j)}}}\cdots\psi_{z})^{\ast}e(\bj)\Psi_{[1:z]}.
  \end{equation*}

  where $\overbrace{\psi_{m_{(s,j)}}}$ means that the factor $\psi_{m_{(s,j)}}$ was erased from the product $\Psi_{[1:z]}=\psi_1\cdots\psi_z.$

  Analogously if $1\leq s_1<s_2<r$ we define

  \begin{equation*}
    E^{(s_1,s_2)}(\calL_{(r,j)})=
    (\psi_1\cdots\overbrace{\psi_{m_{(s_1,j)}}}\cdots\overbrace{\psi_{m_{(s_2,j)}}}\cdots\psi_z)^{\ast}e(\bj)\Psi_{[1:z]},
  \end{equation*}

  and more generally we define (naturally) $E^{(s_1,s_2,\dots,s_g)}(\calL_{(r,j)})$ for any $1\leq s_1<s_2<\cdots<s_g<r.$
\end{definition}

  Diagrammatically, the elements $E^{(s_1,\dots,s_g)}\calL_{(r,j)}$ look like:

   \begin{equation*}
\begin{tikzpicture}[xscale=0.5,yscale=0.65]
  \node at (-2,1){$E^{(s)}(\calL_{(r,j)})=$};
  \node[below]at (5,-1.3) {$N$};
  \node at (6.9,1){$\cdots$};
  \draw[thick](4,0)--(4,2);
  \draw[thick](4,0)--(6,0);
  \draw[thick](4,2)--(6,2);
  \draw[thick](6,0)--(6,2);
  \draw[red,thick](7.5,0)--(7.5,1.5);
  \draw[red,thick](7.7,2)--(9.5,2);
  \draw[red,thick](7.7,0)--(9.5,0);
  \draw[red,thick](7.5,0)--(7.5,-0.1);
  \node[below]at (8.5,-1.3) {${\color{red}B_{(s,j)}}$};
  \draw[red,thick](9.5,0)--(9.5,2);
  \node at (10.5,1){$\cdots$};
  \draw[thick](13,0)--(13,2);
  \draw[thick](11.5,0)--(11.5,2);
  \draw[thick](11.5,0)--(13,0);
  \draw[thick](11.5,2)--(13,2);
  \draw[red,thick](14,0.1)--(3.5,0.8);
  \draw[red,thick](3.5,1.2)--(3.5,0.8);
  \draw[red,thick](3.5,1.2)--(7.5,1.8);
  \draw[red,thick](7.5,1.8)--(7.5,2.2);
  \draw[red,thick](14,2.2)--(14,1.9);
  \draw[red,thick](14,1.9)--(7.5,1.5);
  \draw[red,thick](14,-0.2)--(14,0.1);
  \draw[red,thick](16.5,0)--(16.5,2);
  \draw[red,thick](14.2,0)--(16.5,0);
  \draw[red,thick](14.2,2)--(16.5,2);
  \node at (17.5,1){$\cdots$};
  \node[below]at (15.5,-1.3) {${\color{red}B_{(r,j)}}$};

\end{tikzpicture}
\end{equation*}

\begin{equation*}
\begin{tikzpicture}[xscale=0.5,yscale=0.65]
  \node at (-4,1){$E^{(s_1,s_2)}(\calL_{(r,j)})=$};
  \draw[thick](0,0)--(0,2);
  \draw[thick](0,0)--(2,0);
  \draw[thick](0,2)--(2,2);
  \draw[thick](2,0)--(2,2);
  \node[below]at (1,-1.3) {$N$};
  \node at (2.5,1){$\cdots$};
  \node at (6.9,1){$\cdots$};
  \draw[thick,red](3.5,-0.1)--(3.5,1.5);
  \draw[thick,red](3.5,1.8)--(3.5,2.3);
  \draw[thick,red](3.7,0)--(5.5,0);
  \draw[thick,red](3.7,2)--(5.5,2);
  \draw[thick,red](5.5,0)--(5.5,2);
  \node[below]at (4.5,-1.3) {${\color{red}B_{(s_1,j)}}$};
  \draw[red,thick](7.5,0)--(7.5,1.5);
  \draw[red,thick](3.5,1.5)--(7.5,1.8);
  \draw[red,thick](7.7,2)--(9.5,2);
  \draw[red,thick](7.7,0)--(9.5,0);
  \draw[red,thick](7.5,0)--(7.5,-0.1);
  \node[below]at (8.5,-1.3) {${\color{red}B_{(s_2,j)}}$};
  \draw[red,thick](9.5,0)--(9.5,2);
  \node at (10.5,1){$\cdots$};
  \draw[thick](13,0)--(13,2);
  \draw[thick](11.5,0)--(11.5,2);
  \draw[thick](11.5,0)--(13,0);
  \draw[thick](11.5,2)--(13,2);
  \draw[red,thick](14,0.1)--(-0.5,0.8);
  \draw[red,thick](-0.5,1.2)--(-0.5,0.8);
  \draw[red,thick](-0.5,1.2)--(3.5,1.8);
  \draw[red,thick](7.5,1.8)--(7.5,2.3);
  \draw[red,thick](14,2.2)--(14,1.9);
  \draw[red,thick](14,1.9)--(7.5,1.5);
  \draw[red,thick](14,-0.2)--(14,0.1);
  \draw[red,thick](16.5,0)--(16.5,2);
  \draw[red,thick](14.2,0)--(16.5,0);
  \draw[red,thick](14.2,2)--(16.5,2);
  \node at (17.5,1){$\cdots$};
  \node[below]at (15.5,-1.3) {${\color{red}B_{(r,j)}}$};

\end{tikzpicture}
\end{equation*}
and so on.

\begin{lemma}\label{lemma-Y1jE1Lrj-die}
Let $1<s_2<\cdots <s_g<r\leq k$ and $0\leq j\leq l-2.$ Then
  \begin{equation*}
    \Y_{(1,j)}E^{(1,s_2,\dots,s_g)}(\calL_{(r,j)})=0.
  \end{equation*}
\end{lemma}

\begin{proof}
First note that
 \begin{equation*}
 \begin{tikzpicture}[xscale=0.5,yscale=0.65]
  \node at (-2,1){$\Y_{(1,j)}E^{(1)}(\calL_{(r,j)})=$};
  \node[below]at (5,-1.3) {$N$};
  \node at (6.9,1){$\cdots$};
  \draw[thick](4,0)--(4,2);
  \draw[thick](4,0)--(6,0);
  \draw[thick](4,2)--(6,2);
  \draw[thick](6,0)--(6,2);
  \draw[red,thick](7.5,0)--(7.5,1.5);
  \draw[red,thick](7.7,2)--(9.5,2);
  \draw[red,thick](7.7,0)--(9.5,0);
  \draw[red,thick](7.5,0)--(7.5,-0.1);
  \node[below]at (8.5,-1.3) {${\color{red}B_{(1,j)}}$};
  \draw[red,thick](9.5,0)--(9.5,2);
  \node at (10.5,1){$\cdots$};
  \draw[thick](13,0)--(13,2);
  \draw[thick](11.5,0)--(11.5,2);
  \draw[thick](11.5,0)--(13,0);
  \draw[thick](11.5,2)--(13,2);
  \draw[red,thick](14,0.1)--(3.5,0.8);
  \draw[red,thick](3.5,1.2)--(3.5,0.8);
  \draw[red,thick](3.5,1.2)--(7.5,1.8);
  \draw[red,thick](7.5,1.8)--(7.5,2.3);
  \draw[fill,red] (7.5,2.1) circle [radius=0.1];
  \draw[red,thick](14,2.2)--(14,1.9);
  \draw[red,thick](14,1.9)--(7.5,1.5);
  \draw[red,thick](14,-0.2)--(14,0.1);
  \draw[red,thick](16.5,0)--(16.5,2);
  \draw[red,thick](14.2,0)--(16.5,0);
  \draw[red,thick](14.2,2)--(16.5,2);
  \node at (17.5,1){$\cdots$};
  \node[below]at (15.5,-1.3) {${\color{red}B_{(r,j)}}$};

\end{tikzpicture}
\end{equation*}
By definition of $m_{(1,j)},$ and by relation \ref{punto-abajo}, we can transport the dot freely to the left (first position) and then we can apply relation \ref{dot-al-inicio} to obtain  $$\Y_{(1,j)}E^{(1)}(\calL_{(r,j)})=0.$$ The general case is analogous.
\end{proof}

\begin{lemma}\label{lemma-PsEsLrj-die}
 Let $1<r\leq k$ and $0\leq j\leq l-2.$ For each $1\leq s <r$ we have
\begin{equation*}
\left(\prod_{1\leq u\leq s}{\Y_{(u,j)}}\right)E^{(s)}(\calL_{(r,j)})=0.
\end{equation*}
\end{lemma}

\begin{proof}
Let us denote $$P_s=\prod_{1\leq u\leq s}{\Y_{(u,j)}}.$$

If $s=1$ the result follows from lemma \ref{lemma-Y1jE1Lrj-die} (in this case $P_1=\Y_{(1,j)}.)$ If $s>1,$ note that

\begin{equation*}
\begin{tikzpicture}[xscale=0.5,yscale=0.65]
  \node at (-4,1){$\Y_{(s,j)}E^{(s)}(\calL_{(r,j)})=$};
  \draw[thick](0,0)--(0,2);
  \draw[thick](0,0)--(2,0);
  \draw[thick](0,2)--(2,2);
  \draw[thick](2,0)--(2,2);
  \node[below]at (1,-1.3) {$N$};
  \node at (2.5,1){$\cdots$};
  \node at (6.9,1){$\cdots$};
  \draw[thick,red](3.5,0)--(3.5,2);
  \draw[thick,red](3.5,0)--(5.5,0);
  \draw[thick,red](3.5,2)--(5.5,2);
  \draw[thick,red](5.5,0)--(5.5,2);
  \node[below]at (4.5,-1.3) {${\color{red}B_{(u,j)}}$};
  \draw[red,thick](7.5,0)--(7.5,1.5);
  \draw[red,thick](7.7,2)--(9.5,2);
  \draw[red,thick](7.7,0)--(9.5,0);
  \draw[red,thick](7.5,0)--(7.5,-0.1);
  \node[below]at (8.5,-1.3) {${\color{red}B_{(s,j)}}$};
  \draw[red,thick](9.5,0)--(9.5,2);
  \node at (10.5,1){$\cdots$};
  \draw[thick](13,0)--(13,2);
  \draw[thick](11.5,0)--(11.5,2);
  \draw[thick](11.5,0)--(13,0);
  \draw[thick](11.5,2)--(13,2);
  \draw[red,thick](14,0.1)--(-0.5,0.8);
  \draw[red,thick](-0.5,1.2)--(-0.5,0.8);
  \draw[red,thick](-0.5,1.2)--(7.5,1.8);
  \draw[red,thick](7.5,1.8)--(7.5,2.3);
  \draw[fill,red] (7.5,2.1) circle [radius=0.1];
  \draw[red,thick](14,2.2)--(14,1.9);
  \draw[red,thick](14,1.9)--(7.5,1.5);
  \draw[red,thick](14,-0.2)--(14,0.1);
  \draw[red,thick](16.5,0)--(16.5,2);
  \draw[red,thick](14.2,0)--(16.5,0);
  \draw[red,thick](14.2,2)--(16.5,2);
  \node at (17.5,1){$\cdots$};
  \node[below]at (15.5,-1.3) {${\color{red}B_{(r,j)}}$};

\end{tikzpicture}
\end{equation*}
  By relation \ref{punto-abajo} we can transport the dot from position $m_{(s,j)}$ to the first position throughout the corresponding string. Each time when the dot pass throughout a cross between sister strings, that is, strings corresponding to ${m_{(u,j)}}$ for $u<s,$ we have to subtract the element $E^{(u,s)}(\calL_{(r,j)}).$
Therefore and using relation \ref{dot-al-inicio}, we obtain
\begin{equation*}
  \Y_{(s,j)}E^{(s)}(\calL_{(r,j)})=-\sum_{1\leq u<s}{E^{(u,s)}}(\calL_{(r,j)}).
\end{equation*}

Analogously for each $1<u<s$ we have

\begin{equation*}
  \Y_{(u,j)}E^{(u,s)}(\calL_{(r,j)})=\pm\sum_{1\leq v<u}{E^{(v,u,s)}}(\calL_{(r,j)}).
\end{equation*}

Following inductively we can see that

\begin{equation*}
  \Y_{(u_1,j)}E^{(u_1,u_2,\dots,u_g,s)}(\calL_{(r,j)})=
  \pm\sum_{1\leq v<u_1}{E^{(v,u_1,u_2,\dots,u_g,s)}}(\calL_{(r,j)}).
\end{equation*}

Therefore we have

\begin{equation*}
  P_sE^{(s)}(\calL_{(r,j)})=\Y_{(1,j)}\sum_{g}\left(\sum_{1<u_2<\cdots <u_g <s}\mathcal{P}(\Y)_{(u_2,\dots,u_g)}E^{(1,u_2,\dots,u_g,s)}(\calL_{(r,j)})\right).
\end{equation*}
where in each case $\mathcal{P}(\Y)_{(u_2,\dots ,u_g)}$ is a polynomial in the variables $\Y_{(2,j)},\dots,\Y_{(s-1,j)}.$ By distributivity, commutativity and lemma \ref{lemma-Y1jE1Lrj-die} we obtain the desired result.
\end{proof}

\begin{lemma}\label{lemma-EsLrj-as-LsjLrj}
  Let $0 < r\leq k$ and $0\leq j\leq l-2,$ and $0\leq s<r.$ Then
  \begin{equation*}
    E^{(s)}(\calL_{(r,j)})=-\calL_{(s,j)}\calL_{(r,j)}.
  \end{equation*}
\end{lemma}

\begin{proof}
Let $z=m_{(r,j)}-1,v=m_{(s,j)}+1$ and $\bj=\s_{[1:z]}\cdot\bif.$

  By relation \ref{punto-arriba} and lemma \ref{lemma-diag-L} we have that
  \begin{equation*}
    \Y_{(s,j)}\calL_{(r,j)}=(\Psi_{[1:z]})^{\ast}y_ve(\bj)\Psi_{[1:z]}+E^{(s)}(\calL_{(r,j)})
  \end{equation*}

  \begin{equation*}
\begin{tikzpicture}[xscale=0.5,yscale=0.65]
  \node at (0,1){$\Y_{(s,j)}\calL_{(r,j)}=$};
  \node[below]at (4.2,-1.3) {$N$};
  \node at (5.1,1){$\cdots$};
  \draw[thick](4,0)--(4,2);
  \draw[thick](4,0)--(4.5,0);
  \draw[thick](4,2)--(4.5,2);
  \draw[thick](4.5,0)--(4.5,2);
  \draw[thick](6,0)--(6,2);
  \draw[thick](6,0)--(7,0);
  \draw[thick](6,2)--(7,2);
  \draw[thick](7,0)--(7,2);
  \draw[red,thick](8,0)--(8,2.2);
  \draw[fill,red] (8,2) circle [radius=0.11];
  \draw[red,thick](8.25,2)--(9.5,2);
  \draw[red,thick](8.25,0)--(9.5,0);
  \node[below]at (8.5,-1.3) {${\color{red}B_{(s,j)}}$};
  \draw[red,thick](9.5,0)--(9.5,2);
  \node at (10.5,1){$\cdots$};
  \draw[thick](13,0)--(13,2);
  \draw[thick](12,0)--(12,2);
  \draw[thick](12,0)--(13,0);
  \draw[thick](12,2)--(13,2);
  \draw[red,thick](14,0.1)--(3.5,0.8);
  \draw[red,thick](3.5,1.2)--(3.5,0.8);
  \draw[red,thick](3.5,1.2)--(14,1.9);
  \draw[red,thick](14,2.2)--(14,1.9);
  \draw[red,thick](14,-0.2)--(14,0.1);
  \draw[red,thick](15.5,0)--(15.5,2);
  \draw[red,thick](14.2,0)--(15.5,0);
  \draw[red,thick](14.2,2)--(15.5,2);
  \node at (16.5,1){$\cdots$};
  \node[below]at (14.5,-1.3) {${\color{red}B_{(r,j)}}$};

\end{tikzpicture}
\end{equation*}

\begin{equation*}
\begin{tikzpicture}[xscale=0.5,yscale=0.65]
  \node at (0,1){$\Y_{(s,j)}\calL_{(r,j)}=$};
  \node[below]at (4.2,-1.3) {$N$};
  \node at (5.1,1){$\cdots$};
  \draw[thick](4,0)--(4,2);
  \draw[thick](4,0)--(4.5,0);
  \draw[thick](4,2)--(4.5,2);
  \draw[thick](4.5,0)--(4.5,2);
  \draw[thick](6,0)--(6,2);
  \draw[thick](6,0)--(7,0);
  \draw[thick](6,2)--(7,2);
  \draw[thick](7,0)--(7,2);
  \draw[red,thick](8,0)--(8,2.2);
  \draw[fill,red] (8,1) circle [radius=0.11];
  \draw[red,thick](8.25,2)--(9.5,2);
  \draw[red,thick](8.25,0)--(9.5,0);
  \node[below]at (8.5,-1.3) {${\color{red}B_{(s,j)}}$};
  \draw[red,thick](9.5,0)--(9.5,2);
  \node at (10.5,1){$\cdots$};
  \draw[thick](13,0)--(13,2);
  \draw[thick](12,0)--(12,2);
  \draw[thick](12,0)--(13,0);
  \draw[thick](12,2)--(13,2);
  \draw[red,thick](14,0.1)--(3.5,0.8);
  \draw[red,thick](3.5,1.2)--(3.5,0.8);
  \draw[red,thick](3.5,1.2)--(14,1.9);
  \draw[red,thick](14,2.2)--(14,1.9);
  \draw[red,thick](14,-0.2)--(14,0.1);
  \draw[red,thick](15.5,0)--(15.5,2);
  \draw[red,thick](14.2,0)--(15.5,0);
  \draw[red,thick](14.2,2)--(15.5,2);
  \node at (16.5,1){$\cdots$};
  \node[below]at (14.5,-1.3) {${\color{red}B_{(r,j)}}$};
  \node at (19.5,1){$+E^{(s)}(\calL_{(r,j)}).$};
\end{tikzpicture}
\end{equation*}

Let $A=(\Psi_{[1:z]})^{\ast}y_ve(\bj)\Psi_{[1:z]},$ then applying relation \ref{dot-salta} and relation \ref{lazo} we obtain a decomposition $A=A_1-A_2,$ where

\begin{equation*}
\begin{tikzpicture}[xscale=0.5,yscale=0.65]
  \node at (0,1){$A_1=$};
  \node[below]at (4.2,-1.3) {$N$};
  \node at (5.1,1){$\cdots$};
  \draw[thick](4,0)--(4,2);
  \draw[thick](4,0)--(4.5,0);
  \draw[thick](4,2)--(4.5,2);
  \draw[thick](4.5,0)--(4.5,2);
  \draw[thick](6,0)--(6,2);
  \draw[thick](6,0)--(7,0);
  \draw[thick](6,2)--(7,2);
  \draw[thick](7,0)--(7,2);
  \draw[red,thick](8,0)--(8,2);
  \draw[fill] (7,1) circle [radius=0.11];
  \draw[red,thick](8,2)--(9.5,2);
  \draw[red,thick](8,0)--(9.5,0);
  \node[below]at (8.5,-1.3) {${\color{red}B_{(s,j)}}$};
  \draw[red,thick](9.5,0)--(9.5,2);
  \node at (10.5,1){$\cdots$};
  \draw[thick](13,0)--(13,2);
  \draw[thick](12,0)--(12,2);
  \draw[thick](12,0)--(13,0);
  \draw[thick](12,2)--(13,2);
  \draw[red,thick](14,0.1)--(3.5,0.8);
  \draw[red,thick](3.5,1.2)--(3.5,0.8);
  \draw[red,thick](3.5,1.2)--(14,1.9);
  \draw[red,thick](14,2.2)--(14,1.9);
  \draw[red,thick](14,-0.2)--(14,0.1);
  \draw[red,thick](15.5,0)--(15.5,2);
  \draw[red,thick](14.2,0)--(15.5,0);
  \draw[red,thick](14.2,2)--(15.5,2);
  \node at (16.5,1){$\cdots$};
  \node[below]at (14.5,-1.3) {${\color{red}B_{(r,j)}}$};
\end{tikzpicture}
\end{equation*}

and
\begin{equation*}
\begin{tikzpicture}[xscale=0.5,yscale=0.65]
  \node at (0,1){$A_{2}=$};
  \node[below]at (4.2,-1.3) {$N$};
  \draw[thick](4,0)--(4,2);
  \draw[thick](4,0)--(4.5,0);
  \draw[thick](4,2)--(4.5,2);
  \draw[thick](4.5,0)--(4.5,2);
  \draw[thick](6,0)--(6,2);
  \draw[thick](6,0)--(7,0);
  \draw[thick](6,2)--(7,2);
  \draw[thick](7,0)--(7,2);
  \draw[red,thick](8,-0.1)--(8,0.7);
  \draw[red,thick](8,1.3)--(8,2);
  \draw[red,thick](8,0.7)--(5.5,0.8);
  \draw[red,thick](5.5,1.2)--(8,1.3);
  \draw[red,thick](5.5,1.2)--(5.5,0.8);
  \draw[red,thick](8.25,2)--(9.5,2);
  \draw[red,thick](8.25,0)--(9.5,0);
  \node[below]at (8.5,-1.3) {${\color{red}B_{(s,j)}}$};
  \draw[red,thick](9.5,0)--(9.5,2);
  \node at (10.5,1){$\cdots$};
  \draw[thick](13,0)--(13,2);
  \draw[thick](12,0)--(12,2);
  \draw[thick](12,0)--(13,0);
  \draw[thick](12,2)--(13,2);
  \draw[red,thick](14,0.1)--(3.5,0.8);
  \draw[red,thick](3.5,1.2)--(3.5,0.8);
  \draw[red,thick](3.5,1.2)--(14,1.9);
  \draw[red,thick](14,2.2)--(14,1.9);
  \draw[red,thick](14,-0.2)--(14,0.1);
  \draw[red,thick](15.5,0)--(15.5,2);
  \draw[red,thick](14.2,0)--(15.5,0);
  \draw[red,thick](14.2,2)--(15.5,2);
  \node at (16.5,1){$\cdots$};
  \node[below]at (14.5,-1.3) {${\color{red}B_{(r,j)}}$};
\end{tikzpicture}
\end{equation*}

Applying relation \ref{punto-arriba} and lemma \ref{lemma-clean-Y} we obtain $A_1=\Y_{(s,j-1)}\calL_{(r,j)}$

\begin{equation*}
\begin{tikzpicture}[xscale=0.5,yscale=0.65]
  \node at (0,1){$A_1=$};
  \node[below]at (4.2,-1.3) {$N$};
  \node at (5.1,1){$\cdots$};
  \draw[thick](4,0)--(4,2);
  \draw[thick](4,0)--(4.5,0);
  \draw[thick](4,2)--(4.5,2);
  \draw[thick](4.5,0)--(4.5,2);
  \draw[thick](6,0)--(6,2);
  \draw[thick](6,0)--(7,0);
  \draw[thick](6,2)--(7,2);
  \draw[thick](7,0)--(7,2);
  \draw[red,thick](8,0)--(8,2);
  \draw[fill] (6,1.8) circle [radius=0.11];
  \draw[red,thick](8,2)--(9.5,2);
  \draw[red,thick](8,0)--(9.5,0);
  \node[below]at (8.5,-1.3) {${\color{red}B_{(s,j)}}$};
  \draw[red,thick](9.5,0)--(9.5,2);
  \node at (10.5,1){$\cdots$};
  \draw[thick](13,0)--(13,2);
  \draw[thick](12,0)--(12,2);
  \draw[thick](12,0)--(13,0);
  \draw[thick](12,2)--(13,2);
  \draw[red,thick](14,0.1)--(3.5,0.8);
  \draw[red,thick](3.5,1.2)--(3.5,0.8);
  \draw[red,thick](3.5,1.2)--(14,1.9);
  \draw[red,thick](14,2.2)--(14,1.9);
  \draw[red,thick](14,-0.2)--(14,0.1);
  \draw[red,thick](15.5,0)--(15.5,2);
  \draw[red,thick](14.2,0)--(15.5,0);
  \draw[red,thick](14.2,2)--(15.5,2);
  \node at (16.5,1){$\cdots$};
  \node[below]at (14.5,-1.3) {${\color{red}B_{(r,j)}}$};
\end{tikzpicture}
\end{equation*}

On the other hand, by applying similar arguments as we used in the proof of lemma \ref{lemma-diag-L} and using corollary \ref{coro-bif-j-not-possible3}, we can pull the string corresponding to $i_{m_{(s,j)}}$ to the left to the second position:

\begin{equation*}
\begin{tikzpicture}[xscale=0.5,yscale=0.65]
  \node at (0,1){$A_{2}=$};
  \node[below]at (4.2,-1.3) {$N$};
  \draw[thick](4,0)--(4,2);
  \draw[thick](4,0)--(4.5,0);
  \draw[thick](4,2)--(4.5,2);
  \draw[thick](4.5,0)--(4.5,2);
  \draw[thick](6,0)--(6,2);
  \draw[thick](6,0)--(7,0);
  \draw[thick](6,2)--(7,2);
  \draw[thick](7,0)--(7,2);
  \draw[red,thick](8,-0.1)--(8,0.7);
  \draw[red,thick](8,1.3)--(8,2);
  \draw[red,thick](8,0.7)--(3.8,0.95);
  \draw[red,thick](3.8,1.05)--(8,1.3);
  \draw[red,thick](3.8,1.05)--(3.8,0.95);
  \draw[red,thick](8.25,2)--(9.5,2);
  \draw[red,thick](8.25,0)--(9.5,0);
  \node[below]at (8.5,-1.3) {${\color{red}B_{(s,j)}}$};
  \draw[red,thick](9.5,0)--(9.5,2);
  \node at (10.5,1){$\cdots$};
  \draw[thick](13,0)--(13,2);
  \draw[thick](12,0)--(12,2);
  \draw[thick](12,0)--(13,0);
  \draw[thick](12,2)--(13,2);
  \draw[red,thick](14,0.1)--(3.5,0.8);
  \draw[red,thick](3.5,1.2)--(3.5,0.8);
  \draw[red,thick](3.5,1.2)--(14,1.9);
  \draw[red,thick](14,2.2)--(14,1.9);
  \draw[red,thick](14,-0.2)--(14,0.1);
  \draw[red,thick](15.5,0)--(15.5,2);
  \draw[red,thick](14.2,0)--(15.5,0);
  \draw[red,thick](14.2,2)--(15.5,2);
  \node at (16.5,1){$\cdots$};
  \node[below]at (14.5,-1.3) {${\color{red}B_{(r,j)}}$};
\end{tikzpicture}
\end{equation*}
Therefore $A_2$ belongs to the ideal generated by some idempotent $e(\bjf)$ where $\bjf$ is a residue sequence with their first and second residues equal. As we mention in example \ref{ex-blob-impossible} this is a $\kappa-$blob impossible residue sequence, and then $A_2=0.$

Therefore we have
\begin{equation*}
  \Y_{(s,j)}\calL_{(r,j)}=\Y_{(s,j-1)}\calL_{(r,j)}+E^{(s)}(\calL_{(r,j)})
\end{equation*}
The desired result follows reorganising the last equation.

(Note that we have worked on the case where $j\neq 0.$ The case $j=0$ is analogous).
\end{proof}

\begin{lemma}\label{lemma-square-Yrj}
  Let $1<r\leq k$ and $0\leq j\leq l-2.$ Then we have:
  \begin{equation*}
    (\Y_{(r,j)})^{2}=\left\{\begin{array}{cc}
                              \Y_{(r,j-1)}\Y_{(r,j)}-\sum_{1\leq s<r}\calL_{(s,j)}\calL_{(r,j)} & \quad \textrm{if}\quad j\neq 0 \\
                              \quad & \quad \\
                              \Y_{(r-1,l-2)}\Y_{(r,0)}-\sum_{1\leq s<r}\calL_{(s,0)}\calL_{(r,0)} & \quad \textrm{if}\quad j= 0
                            \end{array}\right.
  \end{equation*}
\end{lemma}

\begin{proof}
  We only prove the case $j\neq 0,$ the other case is analogous.
  By lemma \ref{lemma-diag-L} we have
  \begin{equation*}
    \Y_{(r,j)}\calL_{(r,j)}=\Y_{(r,j)}(\Psi_{[1:z]})^{\ast}e(\bj)\Psi_{[1,z]}
  \end{equation*}
  where $z$ and $\bj$ are as in the proof of the mentioned lemma:
  \begin{equation*}
\begin{tikzpicture}[xscale=0.5,yscale=0.65]
  \node at (0,1){$\Y_{(r,j)}\calL_{(r,j)}=$};
  \draw[thick](4,0)--(4,2);
  \draw[thick](4,0)--(6,0);
  \draw[thick](4,2)--(6,2);
  \draw[thick](6,0)--(6,2);
  \draw[thick,red](7.5,0)--(7.5,2);
  \draw[thick,red](7.5,2)--(9.5,2);
  \draw[thick,red](7.5,0)--(9.5,0);
  \node[below]at (9,-1.3) {${\color{red}B_{(s,j)}}$};
  \draw[thick,red](9.5,0)--(9.5,2);
  \node at (11,1){$\cdots$};
  \draw[thick](13,0)--(13,2);
  \draw[thick](11.5,0)--(11.5,2);
  \draw[thick](11.5,0)--(13,0);
  \draw[thick](11.5,2)--(13,2);
  \draw[red,thick](14,0.1)--(3.5,0.8);
  \draw[red,thick](3.5,1.2)--(3.5,0.8);
  \draw[red,thick](3.5,1.2)--(14,1.9);
  \draw[red,thick](14,2.3)--(14,1.9);
  \draw[red,thick](14,-0.2)--(14,0.1);
  \draw[fill,red] (14,2.1) circle [radius=0.1];
  \draw[red,thick](16.5,0)--(16.5,2);
  \draw[red,thick](14.2,0)--(16.5,0);
  \draw[red,thick](14.2,2)--(16.5,2);
  \node[below]at (15.5,-1.3) {${\color{red}B_{(r,j)}}$};

\end{tikzpicture}
\end{equation*}
   Using relation \ref{punto-abajo} we can transport the dot to the left throughout the string corresponding with $i_{m_{(r,j)}}$, to obtain the element $(\Psi_{[1:z]})^{\ast}y_1e(\bj)\Psi_{[1:z]},$ but each time we cross a string corresponding with $i_{m_{(s,j)}}$ we have to subtract the element $E^{(s)}(\calL_{(r,j)}).$ Therefore:
\begin{equation*}
  \Y_{(r,j)}\calL_{(r,j)}=(\Psi_{[1:z]})^{\ast}y_1e(\bj)\Psi_{[1:z]}-\sum_{1\leq s<r}E^{(s)}(\calL_{(r,j)}).
\end{equation*}

By relation \ref{dot-al-inicio} and lemma \ref{lemma-EsLrj-as-LsjLrj} we obtain
\begin{equation*}
  \Y_{(r,j)}\calL_{(r,j)}=\sum_{1\leq s<r}\calL_{(s,j)}\calL_{(r,j)}.
\end{equation*}

By definition $\calL_{(r,j)}=\Y_{(r,j-1)}-\Y_{(r,j)},$ then reorganizing the last equation we obtain the desired result.
\end{proof}

\begin{corollary}\label{coro-square-Lrj}
  Let $1<r\leq k$ and $0\leq j\leq l-2.$ Then we have:
  \begin{equation*}
    (\calL_{(r,j)})^{2}=\sum_{m_{(s,t)}\leq m_{(r,j)}}C_{(s,t)}\calL_{(s,t)}\calL_{(r,j)}
  \end{equation*}
  where
  \begin{equation*}
    C_{(s,t)}=\left\{\begin{array}{cc}
                -2 & \quad\textrm{if}\quad t=j \\
                \quad & \quad \\
                -1 & \quad\textrm{otherwise}
              \end{array}\right.
  \end{equation*}
\end{corollary}

\begin{proof}
  We only prove the case where $j\neq 0.$ The other case is analogous.

  By definition \ref{def-L-elements} and commutativity we have
  \begin{equation*}
    (\calL_{(r,j)})^2=(\Y_{(r,j-1)})^2-2\Y_{(r,j-1)}\Y_{(r,j)}+(\Y_{(r,j)})^2
  \end{equation*}
  Now by lemma \ref{lemma-square-Yrj} we have

  \begin{equation*}
    (\calL_{(r,j)})^2=(\Y_{(r,j-1)})^2-\Y_{(r,j-1)}\Y_{(r,j)}-\sum_{1\leq s<r}\calL_{(s,j)}\calL_{(r,j)}.
  \end{equation*}
  or equivalently

  \begin{equation*}
    (\calL_{(r,j)})^2=(\Y_{(r,j-1)})\calL_{(r,j)}-\sum_{1\leq s<r}\calL_{(s,j)}\calL_{(r,j)}.
  \end{equation*}
  The result follows applying corollary \ref{coro-telescopic-Y-in-L-2}.
\end{proof}

\begin{lemma}\label{lemma-Y-prod-reduction-1}
  Let $1<r\leq k$ and $0\leq j\leq l-2.$  Then we have
  \begin{equation*}
  \left(\prod_{1\leq s<r}\Y_{(s,j)}\right)(\Y_{(r,j)})^{2}=\left\{
  \begin{array}{cc}
    \left(\prod_{1\leq s< r}\Y_{(s,j)}\right)\Y_{(r,j-1)}\Y_{(r,j)} & \quad\textrm{if}\quad j\neq 0 \\
    \quad & \quad \\
    \left(\prod_{1\leq s< r}\Y_{(s,0)}\right)\Y_{(r-1,l-2)}\Y_{(r,0)} & \quad\textrm{if}\quad j= 0
  \end{array}
  \right.
  \end{equation*}
\end{lemma}
\begin{proof}
  We only prove the case where $j\neq0,$ the other case is analogous. We denote $$P=\prod_{1\leq s<r}\Y_{(s,j)}.$$

  By lemma \ref{lemma-square-Yrj} and lemma \ref{lemma-EsLrj-as-LsjLrj} we have
  \begin{equation*}
    (\Y_{(r,j)})^{2}=\Y_{(r,j-1)}\Y_{(r,j)}+\sum_{1\leq s<r}{E^{(s)}}(\calL_{(r,j)})
  \end{equation*}
  therefore  we have
  \begin{equation*}
    P(\Y_{(r,j)})^{2}=P\Y_{(r,j-1)}\Y_{(r,j)}+\sum_{1\leq s<r}PE^{(s)}(\calL_{(r,j)})
  \end{equation*}

  The assertion of the lemma follows if we prove that $PE^{(s)}(\calL_{(r,j)})=0$ for each $1\leq s<r.$ Let $P_s$ as in lemma \ref{lemma-PsEsLrj-die}, then $P_s$ divides $P,$ and by lemma \ref{lemma-PsEsLrj-die}, $PE^{(s)}(\calL_{(r,j)})=0.$ Therefore we obtain the desired result.
\end{proof}

\begin{lemma}\label{lemma-Y-prod-reduction-2}
  Let $1\leq r\leq k$ and $0\leq j\leq l-2.$  Then we have
  \begin{equation*}
  \left(\prod_{m_{(s,t)}<m_{(r,j)}}\Y_{(s,t)}\right)(\Y_{(r,j)})^{2}=0.
  \end{equation*}
\end{lemma}

\begin{proof}
  We use induction on the set of numbers $m_{(r,j)}.$

  For $(r,j)=(1,0)$ we set $$\prod_{m_{(s,t)}<m_{(1,0)}}\Y_{(s,t)}=e(\bif).$$ In this case the assertion of the lemma follows directly from lemma \ref{lemma-square-Y10-die},   This is the base of our induction.

  For $(r,j)\neq (1,0)$ we use lemma \ref{lemma-square-Y1j} or lemma \ref{lemma-Y-prod-reduction-1} (depending on $r$) to obtain

  \begin{equation*}
    \left(\prod_{m_{(s,t)}<m_{(r,j)}}\Y_{(s,t)}\right)(\Y_{(r,j)})^{2}=
    \left(\prod_{m_{(s,t)}<m_{(r,j-1)}}\Y_{(s,t)}\right)(\Y_{(r,j-1)})^2\Y_{(r,j)}=0.
  \end{equation*}
  the last equation follows by induction. (We proved the case $j\neq0,$ the other case is analogous).
\end{proof}

\subsection{Vector space structure}\label{sec-vector-space-structure}

In this subsection we study the vector space structure of the algebra $\calD.$

\begin{lemma}\label{lemma-D-is-spenned-by-words}
  As a vector space, the algebra $\calD$ is generated by the set:
  $$W_0=\left\{\prod_{(r,j)}(\Y_{(r,j)})^{\alpha_{(r,j)}}:\alpha_{(r,j)}\in\{0,1\}\right\}.$$
  (The product $\prod_{(r,j)}$ is taken over all pairs $(r,j)$ with $1\leq r\leq k$ and $0\leq j\leq l-2.$)
\end{lemma}

\begin{proof}
  By corollary \ref{coro-Y-generate-B} we have that $\mathcal{D}$ is generated as a vector space by the set
  $$W=\left\{\prod_{(r,j)}(\Y_{(r,j)})^{\alpha_{(r,j)}}:\alpha_{(r,j)}\geq 0\right\}$$
  (where $\prod_{(r,j)}(\Y_{(r,j)})^{0}$ is defined as $e(\bif)$).
  Therefore it is enough to prove that each element in $W$ can be written as a linear combination of elements in $W_0.$

  Let $$W_1=\left\{\prod_{(r,j)}(\Y_{(r,j)})^{\alpha_{(r,j)}}\in W:\alpha_{(s,t)}\geq 2\quad \textrm{for some}\quad (s,t)\right\},$$
  then $W=W_0\cup W_1.$

  For each element $P=\prod_{(r,j)}(\Y_{(r,j)})^{\alpha_{(r,j)}}\in W_1$ we define the number $$m(P)=\min\{m_{(s,t)}:\alpha_{(s,t)}\geq 2\},$$ and if $m(P)=m_{(s,t)}$ we denote $h(P)=\alpha_{(s,t)}.$

   We define a partial order on $W_1$ as follows:
  \begin{equation*}
    P_1\prec P_2\quad\textrm{if and only if}\quad \left(\left(m(P_1)<m(P_2)\right)\vee \left(m(P_1)=m(P_2)\wedge h(P_1)<h(P_2)\right)\right).
  \end{equation*}

  Let $P=\prod_{(r,j)}(\Y_{(r,j)})^{\alpha_{(r,j)}}\in W_1.$

  \begin{enumerate}
    \item If $m(P)=m_{(1,0)},$ then by lemma \ref{lemma-square-Y10-die} we have that $P=0$ and the assertion is trivial.
    \item  If $m(P)=m_{(s,t)}>m_{(1,0)}$ then by lemmas \ref{lemma-square-Y1j} and \ref{lemma-square-Yrj} we have that:

  \begin{equation}\label{eq1-lemma-V-is-spanned-by-words}
    P=\left(\prod_{m_{(r,j)}<m_{(s,t)}}(\Y_{(r,j)})^{\alpha_{(r,j)}}\right)H(\Y_{(s,t)})^{h(P)-2}\left(\prod_{m_{(r,j)}>m_{(s,t)}}(\Y_{(r,j)})^{\alpha_{(r,j)}}\right),
  \end{equation}
  where
  \begin{equation*}
    H=\Y_{(s,t-1)}\Y_{(s,t)}+(1-\delta_{s,1})\sum_{u<s}\calL_{(u,t)}\calL_{(s,t)}
  \end{equation*}
  (here $\delta_{s,1}$ is the usual Kronecker's delta function).

  Therefore reorganizing equation \ref{eq1-lemma-V-is-spanned-by-words} and using definition \ref{def-L-elements} we can see that $P$ can be written as a linear combination of elements $Q\in W_1$ such that $Q\prec P:$
  $$P=\sum_{Q\prec P}{C_{Q}Q}\quad (C_{Q}\in\F.)$$
  Applying induction on  $\prec$ we obtain that $P$ can be written as linear combination of elements in $W_0.$
  \end{enumerate}

  Therefore each element of $W$ can be written as a linear combination of elements in $W_0$ as desired.
\end{proof}

\begin{corollary}\label{coro-upperbound-for-dimD}
  \begin{equation*}
    \dim(\mathcal{D})\leq 2^{n},\quad \textrm{where}\quad n=k(l-1).
  \end{equation*}
\end{corollary}
\begin{proof}
  It follows directly from lemma \ref{lemma-D-is-spenned-by-words}.
\end{proof}

\begin{corollary}\label{coro-upperbound-for-degree-D}
 If $h$ is an homogeneous element in $\calD,$ then
 \begin{equation*}
   \deg(h)\leq 2n,\quad \textrm{where}\quad n=k(l-1).
 \end{equation*}
\end{corollary}
\begin{proof}
  By lemma \ref{lemma-D-is-spenned-by-words}, $h$ can be written as a linear combination of elements in $W_0$ (all with the same degree). By definition, the maximal degree of an element in $W_0$ is equal to $2n.$
\end{proof}

\begin{corollary}\label{coro-upperbound-for-degree-D-2}
 If $h$ is an homogeneous element in $\calD,$ and
 \begin{equation*}
   \deg(h)= 2n,\quad \textrm{where}\quad n=k(l-1).
 \end{equation*}
 then
 \begin{equation*}
   h=a\prod_{(r,j)}\Y_{(r,j)},\quad(a\in\F).
 \end{equation*}
 where the product runs over all the pairs $(r,j).$
\end{corollary}
\begin{proof}
  By definition, $\prod_{(r,j)}\Y_{(r,j)}$ is the unique homogeneous element of $W_0$ with degree equal to $2n.$
\end{proof}


\begin{definition}\label{def-Prj-Lrj-elements}
Let $1\leq r \leq k$ and $0\leq j \leq l-2.$ We define the element

\begin{equation*}
  \calP_{(r,j)}=\prod_{m_{(s,t)}\leq m_{(r,j)}}{\Y_{(s,t)}}.
\end{equation*}
\end{definition}

\begin{lemma}\label{lemma-Prj-equal-Qrj}
  For each pair $(r,j)$ with $1\leq r \leq k$ and $0\leq j \leq l-2.$ We have

\begin{equation*}
  \calP_{(r,j)}=\pm \prod_{m_{(s,t)}\leq m_{(r,j)}}{\calL_{(s,t)}}.
\end{equation*}
\end{lemma}

\begin{proof}
Let us denote
\begin{equation*}
  \calQ_{(r,j)}=\prod_{m_{(s,t)}\leq m_{(r,j)}}{\calL_{(s,t)}}.
\end{equation*}

  We use induction on the set of numbers $m_{(r,j)}.$

  If $(r,j)=(1,0)$ the assertion follows directly from definition \ref{def-L-elements}, since
  \begin{equation*}
    \calP_{(1,0)}=\Y_{(1,0)}=-\calL_{(1,0)}.
  \end{equation*}

  If $(r,j)\neq (1,0)$ then by definition \ref{def-Prj-Lrj-elements} and corollary \ref{coro-telescopic-Y-in-L-2} we have \begin{equation*}
  \calP_{(r,j)}=\calP_{(r,j-1)}\Y_{(r,j)}=\calP_{(r,j-1)}\Y_{(r,j-1)}-\calP_{(r,j-1)}\calL_{(r,j)}.
\end{equation*}
By lemma \ref{lemma-Y-prod-reduction-2} we have $\calP_{(r,j-1)}\Y_{(r,j-1)}=0$ and by induction we obtain

\begin{equation*}
  \calP_{(r,j)}=\pm\mathcal{Q}_{(r,j-1)}\calL_{(r,j)}=\pm\mathcal{Q}_{(r,j)}
\end{equation*}
  as desired.

  (We only proved the case where $j\neq0.$ The other case is analogous).
\end{proof}

\begin{lemma}\label{lemma-Qrj-lin-indep}
  The set
  \begin{equation*}
   \left\{\calP_{(r,j)}:1\leq r \leq k,\quad 0\leq j \leq l-2\right\}
\end{equation*}
is a linearly independent subset of $\mathcal{D}.$
\end{lemma}

\begin{proof}
 Is a direct consequence of lemma \ref{lemma-HvsH-reduction}, lemma \ref{lema-calL-is-L} and lemma \ref{lemma-Prj-equal-Qrj}.
\end{proof}

\begin{corollary}\label{coro-Prj-neq-0}
  Let $1\leq r\leq k$ and $0\leq j \leq l-2.$ Then
  \begin{equation*}
    \calP_{(r,j)}\neq 0
  \end{equation*}
\end{corollary}

\begin{proof}
  Is a direct consequence of lemma \ref{lemma-Qrj-lin-indep}.
\end{proof}

\begin{theorem}\label{theo-monomial-basis}
The set
$$W_0=\left\{\prod_{(r,j)}(\Y_{(r,j)})^{\alpha_{(r,j)}}:\alpha_{(r,j)}\in\{0,1\}\right\}$$
is a basis for $\calD.$
\end{theorem}

\begin{proof}
  By lemma \ref{lemma-D-is-spenned-by-words} it is enough to prove  that the set $W_0$ is linearly independent.

  Following analogous arguments as in \cite{Esp-Pl} we define for each element $P=\prod_{(r,j)}(\Y_{(r,j)})^{\alpha_{(r,j)}}\in W_0$ the number
  $$f(P)=\min\{m_{(r,j)}:\alpha_{(r,j)}=1\}$$ and the element
  $$\calF(P)=\prod_{(r,j)}(\Y_{(r,j)})^{\bar{\alpha}_{(r,j)}}\in W_0$$
  where $\bar{\alpha}_{(r,j)}=1-\alpha_{(r,j)}.$
  It is clear that
  \begin{equation*}
    \deg(\calF(P))=2n-\deg(P).
  \end{equation*}

  and
  \begin{equation*}
    P\calF(P)=\calP_{(k,l-2)}.
  \end{equation*}

  For $P,Q\in W_0,$ we define
  \begin{equation*}
    P<Q\quad \textrm{if and only if}\quad ((\deg(P)<\deg(Q))\vee (\deg(P)=\deg(Q)\wedge f(P)>f(Q)).
  \end{equation*}

  Note that under this order relation, $W_0$ has a minimum and a maximum, the elements $$e(\bif)=\prod_{(r,j)}(\Y_{(r,j)})^{0}\quad\textrm{and}\quad \calP_{(k,l-2)}$$ respectively.

   Now, consider the following equation

  \begin{equation}\label{eq-W0-is-basis-1}
    \sum_{P\in W_0}C_P P=0
  \end{equation}
  where $C_P\in \F$ are certain scalars.

  \begin{enumerate}
   \item  If $P\in W_0$ with $P\neq e(\bif),$ and the element $h_P=\calP_{(k,l-2)}P$ is different to $0,$ then it is an homogeneous element with   $$\deg(h_P)=2n+\deg(P),$$
       but this contradicts corollary \ref{coro-upperbound-for-degree-D-2}. Therefore $h_P=0$ whenever $P\neq e(\bif).$
             
             If we multiply by $\calP_{(k,l-2)}$ in equation \ref{eq-W0-is-basis-1} we obtain
                $$C_{e(\bif)}\calP_{(k,l-2)}=0,$$
                and by corollary \ref{coro-Prj-neq-0} we conclude that $C_{e(\bif)}=0.$
    \item Fix an element $Q\in W_0$ and assume that the coefficients $C_P$ in equation \ref{eq-W0-is-basis-1} are equal to zero, for each element $P\in W_0$ such that $P<Q.$

        Then equation \ref{eq-W0-is-basis-1} can be written as

        \begin{equation}\label{eq-W0-is-basis-2}
        \sum_{P\geq Q}C_P P=0.
        \end{equation}

        If $P\in W_0$ is such that $\deg(P)>\deg(Q),$ then the element $P\calF(Q)$ is equal to zero or it is an homogeneous element of $\calD,$ such that
        \begin{equation*}
          \deg(P\calF(Q))>2n.
        \end{equation*}
        By lemma \ref{coro-upperbound-for-degree-D-2} we have
        \begin{equation*}
          P\calF(Q)=0.
        \end{equation*}
        On the other hand if $P\in W_0$ is such that $\deg(P)=\deg(Q),$ but $f(P)=m_{(r,j)}<f(Q),$ then the element $P\calF(Q)$ is divisible by
        \begin{equation*}
          H=\calP_{(r,j)}\Y_{(r,j)}=\left(\prod_{m_{(s,t)}<m_{(r,j)}}\Y_{(s,t)}\right)\left(\Y_{(r,j)}\right)^2,
        \end{equation*}
        but by lemma \ref{lemma-Y-prod-reduction-2}, $H=0.$ Therefore if we multiply by $\calF(Q)$ in equation \ref{eq-W0-is-basis-2} we obtain

        \begin{equation*}
          C_{Q}Q\calF(Q)=0.
        \end{equation*}
        Since $Q\calF(Q)=\calP_{(k,l-2)}\neq 0,$ we conclude that $C_Q=0.$
  \end{enumerate}
  By the last analysis and induction on the order $<$ in $W_0,$ we conclude that the set $W_0$ is linearly independent.
\end{proof}

\begin{corollary}\label{coro-dimD}
  \begin{equation*}
    \dim(\mathcal{D})= 2^{n},\quad \textrm{where}\quad n=k(l-1).
  \end{equation*}
\end{corollary}

\subsection{Optimal presentation}\label{sec-optimal-presentation}

In this section we obtain an \emph{optimal presentation} of $\calD$ by generators and relations. We understand the concept of \emph{optimal presentation}, as a presentation of the algebra but with a minimal set of generators and relations.

\begin{definition}\label{def-optimal-D}
  We define the algebra $\bcalD=\bcalD_{k,l}(\F)$ as the commutative $\F-$algebra, generated by
  $$\left\{\bcalY_{(r,j)}:1\leq r\leq k;\quad 0\leq j\leq l-2\right\}$$
  subject to the following relations:
  \begin{enumerate}
    \item\begin{equation}\label{rel-square-bY10-die}
           (\bcalY_{(1,0)})^2=0.
         \end{equation}
    \item \begin{equation}\label{rel-square-bY1j}
           (\bcalY_{(1,j)})^2=\bcalY_{(1,j-1)}\bcalY_{(1,j)},\quad \textrm{if}\quad j\neq 0.
         \end{equation}
    \item If $r>1$
    \begin{equation}\label{rel-square-bYrj}
           (\bcalY_{(r,j)})^2=
           \left\{ \begin{array}{cc}
             \bcalY_{(r,j-1)}\bcalY_{(r,j)}-\sum_{1\leq s<r}{\bcalL_{(s,j)}\bcalL_{(r,j)}} & \quad \textrm{if}\quad j\neq 0 \\
               \quad & \quad \\
             \bcalY_{(r-1,l-2)}\bcalY_{(r,0)}-\sum_{1\leq s<r}{\bcalL_{(s,0)}\bcalL_{(r,0)}} & \quad \textrm{if}\quad j= 0
              \end{array}\right.
         \end{equation}

    where
    \begin{equation}\label{def-bcalL-elements}
      \bcalL_{(r,j)}=\left\{ \begin{array}{cc}
                               -\bcalY_{(r,0)} & \quad \textrm{if}\quad r=1\wedge j= 0 \\
                               \quad & \quad \\
                               \bcalY_{(r,j-1)}-\bcalY_{(r,j)} & \quad \textrm{if}\quad j\neq 0 \\
                               \quad & \quad \\
                              \bcalY_{(r-1,l-2)}-\bcalY_{(r,0)} & \quad \textrm{if}\quad r>1\wedge j= 0
                             \end{array}\right.
    \end{equation}
  \end{enumerate}
\end{definition}

\begin{lemma}\label{lemma-span-bcalD}
As a vector space, $\bcalD$ is generated by the set
  \begin{equation*}
    \boldsymbol{W}_0=\left\{\prod_{(r,j)}\bcalY_{(r,j)} ^{\alpha_{(r,j)}}:\alpha_{(r,j)}\in\{0,1\}\right\}
  \end{equation*}
  (the product runs over all pairs $(r,j)$).
\end{lemma}

\begin{proof}
  The proof is totally analogous to the proof of lemma \ref{lemma-D-is-spenned-by-words}, but we have to replace in the arguments, lemmas \ref{lemma-square-Y10-die}, \ref{lemma-square-Y1j} and \ref{lemma-square-Yrj} by relations \ref{rel-square-bY10-die},\ref{rel-square-bY1j} and \ref{rel-square-bYrj} respectively.
\end{proof}

\begin{corollary}\label{coro-upperbound-for-dim-bcalD}
  \begin{equation*}
    \dim(\bcalD)\leq 2^{n}\quad \textrm{where}\quad n=k(l-1).
  \end{equation*}
\end{corollary}

\begin{theorem}\label{theo-isomorphims-calD-bcalD}
  The algebras $\bcalD$ and $\calD$ are isomorphic.
\end{theorem}

\begin{proof}
  By lemmas \ref{lemma-square-Y10-die}, \ref{lemma-square-Y1j} and \ref{lemma-square-Yrj} there is an epimorphism of algebras $\Gamma:\bcalD\rightarrow \calD$ given by the assignment $\bcalY_{(r,j)}\mapsto\Y_{(r,j)}.$
   In particular $\dim(\bcalD)\geq \dim(\calD).$
  By corollaries \ref{coro-dimD} and \ref{coro-upperbound-for-dim-bcalD} we conclude that $\Gamma$ is indeed an isomorphism of algebras.
\end{proof}

\begin{corollary}\label{coro-alternative-optimal-presentation}
  The algebra $\bcalD$ is isomorphic to the algebra $\mathbb{D}$ generated by
  $$\{\mathbb{L}_{(r,j)}:1\leq r\leq k;\quad 0\leq j \leq l-2 \}$$
  subject to the following relations:
  \begin{enumerate}
    \item \begin{equation}
            (\mathbb{L}_{(1,0)})^2=0.
          \end{equation}
    \item \begin{equation}
            (\mathbb{L}_{(r,j)})^2=\sum_{m_{(s,t)\leq m_{(r,j)}}}C_{(s,t)}\mathbb{L}_{(s,t)}\mathbb{L}_{(r,j)}
          \end{equation}
    where:
    \begin{equation}
      C_{(s,t)}=\left\{\begin{array}{cc}
                       -2 & \quad\textrm{if}\quad t=j \\
                       \quad & \quad \\
                       -1 & \quad \textrm{otherwhise}
                     \end{array}\right.
    \end{equation}
  \end{enumerate}
\end{corollary}
\begin{proof}
  Theorem \ref{theo-isomorphims-calD-bcalD}, lemma \ref{lemma-L-generate-D}, corollary \ref{coro-telescopic-Y-in-L-2}, corollary \ref{coro-square-L1j} and corollary \ref{coro-square-Lrj} implies that the assignment $\mathbb{L}_{(r,j)}\mapsto \bcalL_{(r,j)}$
  defines an isomorphism of algebras between $\mathbb{D}$ and $\bcalD.$
\end{proof}

\subsection{Direct generalizations and final comments}\label{sec-direct-general}

  Throughout this article we have assumed that $m=\epsilon+k\een.$  If we take an arbitrary $m>\epsilon,$ say $m=\epsilon+k\een+\delta,$  for some $0< \delta < \een,$ then we can consider the algebra $\bcalD$ as a subalgebra of $\mathcal{B}_{\boldsymbol{m}}(\bif),$ where $\boldsymbol{m}=\epsilon+(k+1)\een,$ and $\bif$ is understood as the fundamental vertical sequence of length $\boldsymbol{m}.$ In the next corollary we obtain a presentation for the Gelfand-Tsetlin subalgebra of  $\mathcal{B}_{\boldsymbol{m}}(\bif),$ based on the one given in definition \ref{def-optimal-D}:

  \begin{corollary}\label{coro-subalgebra-Da}
  Let $1\leq a\leq \boldsymbol{m}$ and define $\calD_{\leq a}$ the subalgebra of $\mathcal{B}_{\boldsymbol{m}}(\bif)$ generated by the set:
  $$\{y_je(\bif):1\leq j\leq a\}$$
  then $\calD_{\leq a}$ is isomorphic to the subalgebra $\bcalD_{\leq a}$ of $\bcalD=\bcalD_{(k+1),l}(\F),$  generated by the set
  $$\{\bcalY_{(r,j)}:m_{(r,j)}\leq a \}.$$
  In particular
  \begin{equation*}
    \dim(\bcalD_{\leq a})=2^{n}\quad \textrm{where}\quad n=|\{(r,j):m_{(r,j)}\leq a\}|.
  \end{equation*}
\end{corollary}
\begin{proof}
  It follows directly from theorem \ref{theo-isomorphims-calD-bcalD} .
\end{proof}

\begin{remark}
  In corollary \ref{coro-subalgebra-Da}, if $a<m_{(1,0)}$ then the set $\{\bcalY_{(r,j)}:m_{(r,j)}\leq a \}$ is empty. In this case we understood the subalgebra $\bcalD_{\leq a}$ as the one-dimensional subalgebra generated by the identity element of $\bcalD.$
\end{remark}

From this point to the end of this article, we assume $m$ is an arbitrary positive number.

  All of our analysis in this article were made using the fundamental vertical sequence $\bif.$ We would like to know more on how is the structure of the Gelfand-Tsetlin subalgebra of $\B(\bi)$ for other residue sequences $\bi.$ For this article, we only consider a family of residue sequences not too different to $\bif:$

  Let us call \emph{vertical residue sequence}, to any residue sequence $\bjf$ of the form
\begin{equation}\label{quasi-fund-res-seq1}
 \bjf=(j_1,\dots,j_m),\quad \textrm{where}\quad j_i=\kappa_t-i+1 \mod \een.
\end{equation}
  for some fixed $t\in\{1,\dots,l\}.$

 In particular the fundamental vertical sequence $\bif$ is a vertical residue sequence.

  If $\bjf$ is as in equation \ref{quasi-fund-res-seq1}, then $\bjf=\bi^{\bT},$ where $\bT=\bT^{\bmu}$ and $\bmu$ is a multipartition of $m$ given by
  $$\bmu=(1^{(0)},\dots,1^{(0)},1^{(m)},1^{(0)},\dots,1^{(0)}), $$
  that is, $[\bmu]$ is a \emph{vertical} Young diagram. In particular, a vertical residue sequence is a $\kappa-$blob possible residue sequence.

  \begin{lemma}
    If $\bjf$ is a vertical residue sequence, then the Gelfand-Tsetlin subalgebra of $\B(\bjf),$ is an algebra $\calD$ isomorphic to the one corresponding to the fundamental vertical sequence $\bif.$
  \end{lemma}

  \begin{proof}
    Is not difficult to see that each argument used along this article with respect to $\bif$, has an analogous version for $\bjf.$
     (note that the numbers $\epsilon,b_j$  and $m_{(r,j)}$ will change in that case, but the presentation of $\calD$ will be the same).
  \end{proof}

\begin{example}
   If $\een=13,l=4,m=29$ $\kappa=(0,4,6,10)$ and we take the fundamental vertical sequence

   $$\bif=(0,12,11,10,9,8,7,6,5,4,3,2,1,0,12,11,10,9,8,7,6,5,4,3,2,1,0,12,11),$$
   then we have

    \begin{equation*}
\left\{
\begin{array}{ccc}
  \epsilon&=\hat{\kappa}_1-\hat{\kappa}_4+\een &=3 \\
  \quad & \quad &\quad\\
   b_{0}&=\hat{\kappa}_{4}-\hat{\kappa}_{3} & =4\\
   \quad & \quad&\quad \\
   b_{1}&=\hat{\kappa}_{3}-\hat{\kappa}_{2} & =2\\
   \quad & \quad&\quad \\
  b_{2}&=\hat{\kappa}_2-\hat{\kappa}_4+\een & =7
\end{array}
\right.
\end{equation*}

   and

   \begin{equation*}
\left\{
\begin{array}{ccc}
  m_{(1,0)}&=\epsilon+1 &=4 \\
  \quad & \quad &\quad\\
   m_{(1,1)}&=m_{(1,0)}+b_0 & =8\\
   \quad & \quad&\quad \\
   m_{(1,2)}&=m_{(1,1)}+b_1& =10\\
   \quad & \quad&\quad \\
   m_{(2,0)}&=m_{(1,0)}+\een& =17\\
   \quad & \quad&\quad \\
   m_{(2,1)}&=m_{(1,1)}+\een& =21\\
  \quad & \quad&\quad \\
   m_{(2,2)}&=m_{(1,2)}+\een& =23
\end{array}
\right.
\end{equation*}

   On the other hand, if we take
   $$\bjf=(6,5,4,3,2,1,0,12,11,10,9,8,7,6,5,4,3,2,1,0,12,11,10,9,8,7,6,5,4),$$ instead of $\bif,$ then the definition of numbers $\epsilon$ and $b_j$ is not the same as for $\bif.$ In this case it is given by

   \begin{equation*}
\left\{
\begin{array}{ccc}
  \epsilon&=\hat{\kappa}_3-\hat{\kappa}_2 &=2 \\
  \quad & \quad &\quad\\
   b_{0}&=\hat{\kappa}_{2}-\hat{\kappa}_{1} & =4\\
   \quad & \quad&\quad \\
   b_{1}&=\hat{\kappa}_{1}-\hat{\kappa}_{4}+\een & =3\\
   \quad & \quad&\quad \\
  b_{2}&=\hat{\kappa}_4-\hat{\kappa}_2 & =6
\end{array}
\right.
\end{equation*}

The numbers $m_{(r,j)}$ will change too. In this case are given by

\begin{equation*}
\left\{
\begin{array}{ccc}
  m_{(1,0)}&=\epsilon+1 &=3 \\
  \quad & \quad &\quad\\
   m_{(1,1)}&=m_{(1,0)}+b_0 & =7\\
   \quad & \quad&\quad \\
   m_{(1,2)}&=m_{(1,1)}+b_1& =10\\
   \quad & \quad&\quad \\
   m_{(2,0)}&=m_{(1,0)}+\een& =16\\
   \quad & \quad&\quad \\
   m_{(2,1)}&=m_{(1,1)}+\een& =20\\
  \quad & \quad&\quad \\
   m_{(2,2)}&=m_{(1,2)}+\een& =23
\end{array}
\right.
\end{equation*}

Anyway, one can check that the Gelfand-Tsetlin subalbegra of $\B(\bif)$ and the Gelfand-Tsetlin subalgebra of $\B(\bjf)$ are both isomorphic to the algebra $\bcalD=\bcalD_{2,4}(\F)$ described in definition \ref{def-optimal-D}.
\end{example}

  On the other extreme, we have the \emph{maximal residue sequence}, $\bi^{\max},$  that is the residue sequence corresponding to the maximal tableau whose shape is $\bmu^{\max},$ the maximal (with respect to $\rhd$) $l-$multipartition of $m.$ Since \cite{Lobos-Ryom-Hansen}, we know that  $\B(\bi^{\max})$ is a one-dimensional algebra (generated by $e(\bi^{\max})$), therefore any subalgebra (in particular the Gelfand-Tsetlin subalgebra) of  $\B(\bi^{\max})$ is $\B(\bi^{\max})$ itself.

\medskip

Finally we say that a residue sequence $\bi$ is \emph{quasi-vertical} if it has the form:
\begin{equation}\label{eq-quasi-vertical-res-seq}
  \bi=\bi^{\max}|_{h}\otimes \bj,
\end{equation}
where $h=r_0l<m,$ for some $r_0>0$ and $\bj=(j_1,\dots,j_g)\quad(g=m-h),$ is given by
\begin{equation}\label{eq-quasi-vertical-res-seq-2}
  j_k=\kappa_t-r_0-k+1\mod \een.
\end{equation}
for some fixed $t\in\{1,\dots,l\}.$

If $\bi$ is a quasi-vertical residue sequence, then $\bi=\bi^{\bT},$ where $\bT=\bT^{\bmu}$ and
$$\bmu=(1^{(r_0)},\dots,1^{(r_0)},1^{(r_0+g)},1^{(r_0)},\dots,1^{(r_0)}),$$ then a quasi-vertical residue sequence is $\kappa-$blob possible.

\begin{example}\label{ex-quasi-vertical}
If $\een=13,l=4$ $\kappa=(0,4,6,10),$ $m=18.$

Let $\bi=(0,4,6,10,12,3,5,9,11,2,4,8,3,2,1,0,12,11),$
then $\bi$ is a quasi-vertical residue sequence. In fact
\begin{equation}
  \bi=\bi^{\max}|_{12}\otimes(3,2,1,0,12,11)=\bi^{\bT}
\end{equation}
where

\begin{equation*}
  \includegraphics[scale=0.27]{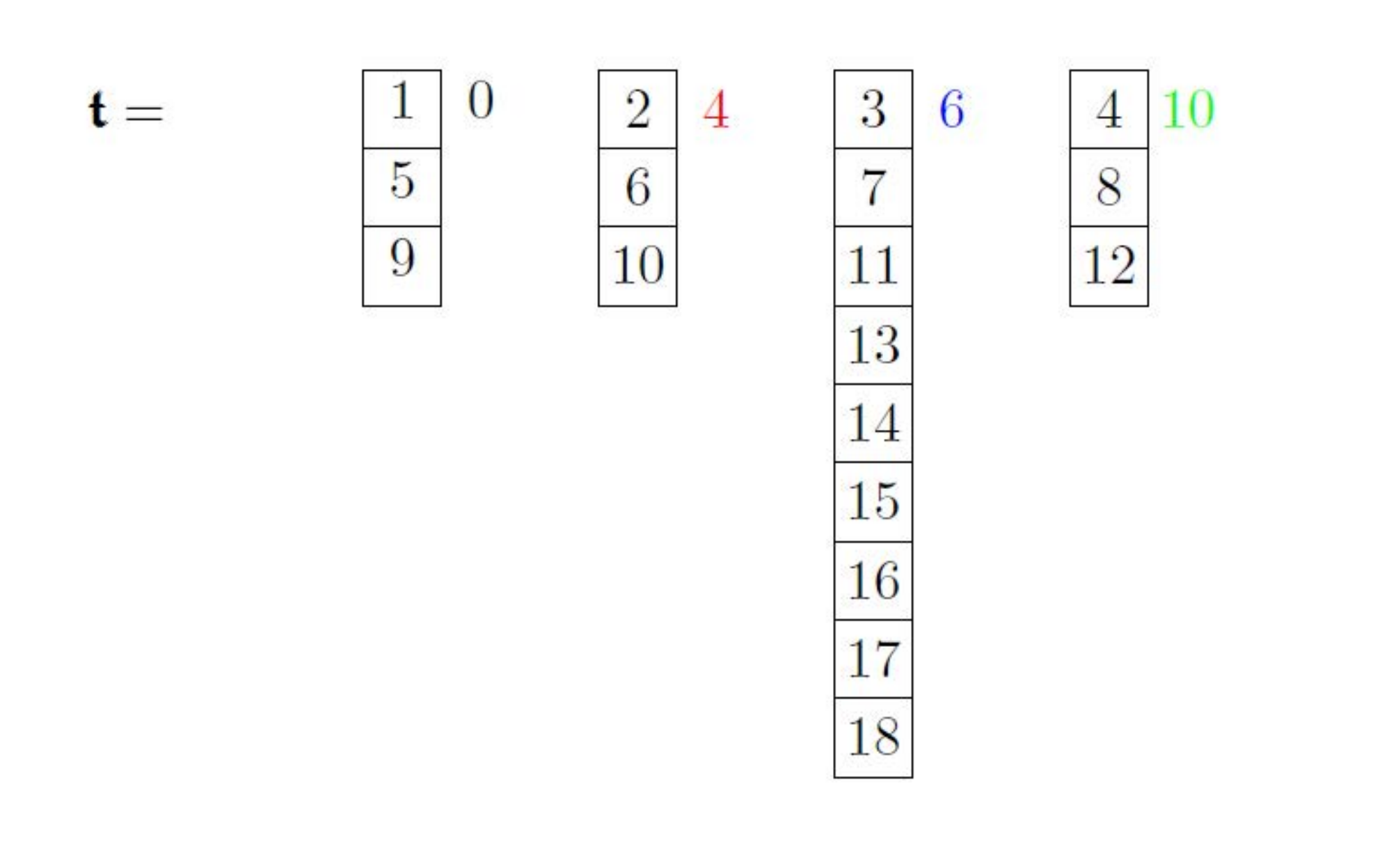}
\end{equation*}

  that is, $\bT=\bT^{\bmu}$ for $\bmu=(1^{(3)},1^{(3)},1^{(9)},1^{(3)}).$
\end{example}

Given a strong adjacency-free multicharge $\kappan=(\kappan_1,\dots,\kappan_l)$ and $r_0>0,$ we define
\begin{equation}
  \kappan-r_0=(\kappan_1-r_0,\dots,\kappan_l-r_0).
\end{equation}
Note that $\kappan-r_0$ is also a strong adjacency-free multicharge.

\begin{lemma}\label{lemma-quasi-vert-case}
Let $\bi=\bi^{\max}|_h\otimes \bj$ be a quasi-vertical residue sequence of length $m$ (as in equation \ref{eq-quasi-vertical-res-seq}). Let $h=r_0l$ and $g=m-h.$ Let $\bnu=\kappan-r_0.$
Let $\mathcal{U}=\mathcal{B}_{l,m}^{\F}(\kappan,\een)$ and $\mathcal{V}= \mathcal{B}_{l,g}^{\F}(\bnu,\een).$
Let $\mathcal{U}(\bi)$ (resp.$\mathcal{V}(\bj)$) the idempotent truncation of $\mathcal{U}$ with respect to the residue sequence $\bi$ (resp. of $\mathcal{V}$ with respect to $\bj.)$
Then the Gelfand-Tsetlin subalgebras of $\mathcal{U}(\bi)$ and $\mathcal{V}(\bj)$ are isomorphic.
\end{lemma}

The lemma basically asserts that the quasi-vertical case, reduces to a vertical case.

\begin{proof}
  Note that:

  \begin{enumerate}
    \item \label{eq-yk-eimax-die-2}
    $      y_ke(\bi)=0,\quad\textrm{for all}\quad 1\leq k\leq h+1.$ (see \cite{Lobos-Ryom-Hansen})
    \item\label{eq-mal-inicio-adapted} If $\bjf=\bi^{\max}|_{h}\otimes j,$ where $j\notin\{\kappa_1-r_0,\dots,\kappa_l-r_0\},$ then
    $e(\bjf)=0.$ (since $\bjf$ is $\kappa-$blob impossible).
    \item \label{eq-otro-mal-inicio-adapted} If $\bjf=\bi^{\max}|_{h}\otimes (j_1,j_2),$ where $j_1\in\{\kappa_1-r_0,\dots,\kappa_l-r_0\},$ and $j_2=j_1+1,$ then $e(\bjf)=0.$ (since $\bjf$ is $\kappa-$blob impossible).
  \end{enumerate}
  Conditions \ref{eq-yk-eimax-die-2},\ref{eq-mal-inicio-adapted} and \ref{eq-otro-mal-inicio-adapted} imply that we can forget the first $h$ components of $\bi$ and work on $\mathcal{V}(\bj)$ instead of $\mathcal{U}(\bi),$ obtaining the same presentations for the corresponding Gelfand-Tsetlin subalgebras.
\end{proof}

\medskip
Following notation coming from \cite{Lobos-Ryom-Hansen},  given two (different) residue sequences $\bi,\bj\in\II^{m},$  we write $ \bi \overset{k}{\sim} \bj $ if $ \bi = s_k \bj $ where $ i_k \neq i_{k+1}  \pm 1$
and we let $ \sim $ be the equivalence relation on $ \II^m  $ generated by all the $ \overset{k}{\sim}$'s. Is not difficult to see that if $\bi$ is a $\kappa-$blob possible residue sequence and $\bi\sim\bj,$ then $\bj$ is also a $\kappa-$blob possible residue sequence. Moreover we have:

\begin{lemma}\label{lemma-isom-equiv-idemp-trunc}
  Given two $\kappa-$blob possible residue sequences $\bi,\bj\in\II^{m}$ such as $\bi\sim\bj$ then the corresponding idempotent truncations $\B(\bi)$ and $\B(\bj)$ are isomorphic.
\end{lemma}
\begin{proof}
  By definition of $\sim,$ there is an element $w=s_{k_{1}}\cdots s_{k_n}$ (reduced expression) in $\Si $ such that $\bj=w\cdot\bi$ with no permutation between \emph{relatives} residues. If we write, $W=\psi_{s_{k_{1}}}\cdots \psi_{s_{k_{l}}}$ then the assignment
  \begin{equation*}
    a\mapsto W^{\ast}a W,\quad (a\in\B(\bi))
  \end{equation*}
  defines an homomorphism of algebras $\Delta:\B(\bi)\rightarrow\B(\bj).$ In fact we have:
  \begin{enumerate}
    \item By relation \ref{lazo}, $W^{\ast}e(\bi)W=e(\bj)$ since there are no crosses between relatives strings in the diagrammatic representation of $W^{\ast}e(\bi)W.$
    \item If $a_1,a_2\in\B(\bi),$ then by distributivity $$W^{\ast}(a_1+a_2) W=W^{\ast}a_1 W+W^{\ast}a_2 W.$$
    \item If $a_1,a_2\in\B(\bi),$ then by relation \ref{lazo} $$W^{\ast}a_1WW^{\ast}a_2W=W^{\ast}(a_1a_2) W.$$
  \end{enumerate}
  Moreover, relation \ref{lazo} implies that the assignment
  \begin{equation*}
    a\mapsto W a W^{\ast},\quad (a\in\B(\bj))
  \end{equation*}
  defines an inverse for $\Delta.$ Therefore the algebras  $\B(\bi),\B(\bj)$ are isomorphic.
\end{proof}

\begin{corollary}
  Given two $\kappa-$blob possible residue sequences $\bi,\bj\in\II^{m}$ such as $\bi\sim\bj$ then the corresponding Gelfand-Tsetlin subalgebras of $\B(\bi)$ and $\B(\bj)$ are isomorphic.
\end{corollary}

With the last corollary, we have covered a larger family of idempotent truncations of $\B$ where the Gelfand-Tsetlin subalgebra is basically given by definition \ref{def-optimal-D}. Is not difficult to see that if $\bj$ is  vertical residue sequence, then the class of $\bj$ under $\sim$ is indeed composed only by $\bj.$ The case quasi-vertical provide more examples of idempotent truncations with the \emph{same} Gelfand-Tsetlin subalgebra.

\medskip

The study of Gelfand-Tsetlin subalgebras of idempotent truncations produced by residue sequences neither vertical nor quasi-vertical nor equivalent (under $\sim$) to a quasi-vertical residue sequence, has been left open for future work.

\medskip
An interesting phenomenon that we can observe in this article is a relation between the dimension of the Gelfand-Tsetlin subalgebra $\calD$ of $\B(\bi)$ and the number of (one-column) standard tableaux with residue sequence equal to $\bi.$ One can easily check that for any vertical or quasi-vertical residue sequence $\bi,$ we have:
  \begin{equation}\label{eq-dimD-is-card-std}
    \dim(\calD)=|\std(\bi)|
  \end{equation}
(see for example corollary \ref{coro-cardinality-of-std-bif} and  corollary \ref{coro-dimD} for the fundamental case $\bi=\bif.$)

This relation between dimension of Gelfand-Tsetlin subalgebras and number of standard tableaux is not new. One can see this phenomenon in the work of Murphy on semisimple Hecke algebras of type $A$ (see \cite{Murphy3}) or in the work of Mathas on semisimple Ariki-Koike algebras (\cite{MatCoef}). Moreover, we know from the work of Mathas (see \cite{Mat-So}), that if $A$ is any cellular algebra with cell datum $(\Lambda,T,C)$ and with a family of \emph{separating} Jucys-Murphy elements (therefore $A$ is semisimple), then the Gelfand-Tsetlin subalgebra (subalgebra generated by the JM-elements) of $A$  admits a basis composed by a complete set of orthogonal primitive idempotents parameterized by the set $T(\Lambda)$ (see \cite{Mat-So} for notation). In this article we obtained not semisimple examples where the relation still works.

This opens the following question, that we expect, to be answered in future works:
\medskip

Given $\bi\in \II^{m},$ an arbitrary $\kappa-$blob possible residue sequence:
\begin{description}
  \item[Q:]  Is the dimension of the Gelfand-Tsetlin subalgebra of $\B(\bi)$ equal to the number of (one-column) standard tableaux with residue sequence $\bi$?
\end{description}




\sc
diego.lobos@pucv.cl, Pontificia Universidad Cat\'olica de Valpara\'iso, Chile.

\end{document}